\documentclass{birkmult}[11pt,a4paper]
\usepackage{mathrsfs}
\usepackage[centertags]{amsmath}
\usepackage{amsfonts}
\usepackage{amsthm}
\usepackage{color}
\usepackage{amssymb}

\usepackage{fancyhdr}

\newcommand{\Leb}{\mathrm{Leb}}

\newcommand\coma[1]{{\color{red}#1}}
\newcommand\dela[1]{{\color{red}}}

\theoremstyle{plain}

\newtheorem{assu}{Assumption}[section]

\input colordvi

\newcommand\E{{\mathbb E}}

\newcommand\N{{\mathbb N}}
\newcommand\R{{\mathbb R}}
 \newtheorem{thm}{Theorem}[section]
 \newtheorem{cor}[thm]{Corollary}
 \newtheorem{lem}[thm]{Lemma}
 \newtheorem{prop}[thm]{Proposition}
 \theoremstyle{definition}
 \newtheorem{defn}[thm]{Definition}
 \theoremstyle{remark}
 \newtheorem{rem}[thm]{Remark}
 
 \numberwithin{equation}{section}

\def\d{\text{\rm{d}}}

\newcommand{\n}{\|}
\,
\newcommand{\F}{\mathscr{F}}
\renewcommand{\P}{\mathbb{P}}

\newcommand{\e}{\varepsilon}

\begin{document}
	\title[Maximal inequalities]
	 {{Maximal inequalities and exponential estimates for stochastic convolutions driven by L\'{e}vy-type processes in Banach spaces with application to stochastic quasi-geostrophic equations$^\dagger$}
	 \footnote{$\dagger$ This work is supported  by NSFC (11501509,11571147,11822106,11831014), NSF of Jiangsu Province
(BK20160004), and the Qing Lan Project and PAPD of Jiangsu Higher Education Institutions.}}
\author[J. Zhu,  Z. Brze\'{z}niak and W. Liu]{}

\begin{center}
\begin{minipage}{145mm}

	\date{today}

\end{minipage}
\end{center}

\maketitle	

\centerline{\scshape Jiahui Zhu$^a$, Zdzis{\l}aw Brze\'{z}niak$^b$, Wei Liu$^{c,}$\footnote{Corresponding author: weiliu@jsnu.edu.cn}}
\medskip
 {\footnotesize
\centerline{ $a.$   School of Science, Zhejiang University of Technology, Hangzhou 310019, China }
 \centerline{ $b.$ Department of Mathematics, University of York,  York YO10 5DD, UK}
\centerline{  $c.$ School of Mathematics and Statistics, Jiangsu Normal University, Xuzhou 221116, China}
 }

\begin{abstract}
We present remarkably simple proofs of Burkholder-Davis-Gundy inequalities for stochastic integrals and maximal inequalities for stochastic convolutions in Banach spaces driven by L\'{e}vy-type processes. Exponential estimates for stochastic convolutions are obtained and two versions of It\^{o}'s formula in Banach spaces are also derived. 	
 Based on the obtained maximal inequality, the existence and uniqueness of mild solutions of stochastic quasi-geostrophic equation with L\'{e}vy noise is established.


 \end{abstract}

	%%% ----------------------------------------------------------------------
	%\tableofcontents
	\section{\bf{Introduction}}

	Over the past few decades, stochastic partial differential equations (SPDEs) have attracted considerable attention from researchers in a wide variety of fields, including biology, physics, engineering and finance etc. (cf. \cite{Da+Zab-1996, LR15} and the references therein).
In the study of SPDEs, the Burkholder-Davis-Gundy (BDG) inequality and maximal inequality play vital roles
	  in proving the existence, uniqueness, and regularity of solutions of SPDEs. There are quite a number of contributions on the study of BDG and maximal inequalities when the state space is a Hilbert space; see \cite{[Da+Zab]}, \cite{[HauSei]},  \cite{[Ichikawa]}, \cite{[Tub]} and \cite{[Metivier]}.  However, many interesting problems in the theory of SPDEs whose natural settings in function spaces are not Hilbert spaces, but rather Banach spaces (e.g. some Sobolev spaces).  Nevertheless, literature and research studies related to these inequalities on general Banach spaces  are very limited  and this is the motivation of our paper.
	
	The overall goal of this work is to investigate BDG inequalities and maximal regularities of stochastic convolutions driven by L\'{e}vy processes in Banach spaces.  We will derive, in Appendix \ref{app-Ito}, two general versions of It\^o's  formula for L\'{e}vy-type processes in Banach spaces, which are crucial for the proof of our inequalities. 	
	We will work in the martingale type $r$ Banach spaces with $1<r\leq 2$.
	This assumption is necessary for establishing a theory of stochastic integration in Banach spaces.
 Typical examples of such spaces are $L^p$ spaces with $p\in[r,\infty)$ and Sobolev spaces. 	

 	Now let us state  our problem more explicitly.
	Let $(E,|\cdot|_E)$ be a separable Banach space of martingale type $r$ with $1<r\leq 2$ and  $(\Omega,\mathcal{F},\mathbb{P})$ be a probability space with the filtration $(\mathcal{F}_t)_{t\geq0}$ satisfying the usual hypotheses, and let  $(Z,\mathcal{Z})$ be a measurable space. We first consider the following process
	\begin{align*}
		   u_t=\int_0^t\int_Z\xi(s,z)\,\tilde{N}(\d s,\d z), \;\; t\in[0,T],\;T>0,
	\end{align*}
	where
	$\tilde{N}$ is a compensated Poisson random measure on $(Z,\mathcal{Z})$ and $\xi$ is an $E$-valued $\mathcal{P}\otimes\mathcal{Z}$-measurable process. In this paper we will establish the following BDG and $L^p$ inequalities
		\begin{align}
		    &      \E\sup_{t\in[0,T]} |u_t|_E^p\leq C_{p,r}\,\E\Big(\int_0^T\int_Z|\xi(s,z)|_E^r\,N(\d s,\d z)\Big)^{\frac{p}{r}} \text{  for all }1\leq p<\infty;\label{intro-eq-2}\\
		     &  \E\sup_{t\in[0,T]} |u_t|_E^p\leq C_{p,r}\min\left\{\,\E\int_0^T\int_Z|\xi(s,z)|^p_E\,\nu(\d z)\,\d s,
   \E\Big(\int_0^T\int_Z|\xi(s,z)|_E^r\,\nu(\d z)\,\d s\Big)^{\frac{p}{r}}\right\}\text{  for all }1\leq p\leq r;\label{intro-eq-3}\\
				 &\E\sup_{t\in[0,T]}|u_t|_E^{p}\leq C_{p,r}\,\E\Big(\int_0^T\int_Z |\xi(s,z)|^p_E\,\nu(\d z)\,\d s\Big)\nonumber\\
		 &\hspace{3cm}+C_{p,r}\,\E \Big(\int_0^T\int_Z|\xi(s,z)|^{r}_E\,\nu(\d z)\,\d s\Big)^{\frac{p}{r}}\text{  for all }r\leq p<\infty.\label{intro-eq-main-1}
		   \end{align}
	
	Here and in what follows, $C_{p,r}$ is a generic constant depending only on $p$ and
	$r$, which may change from line to line. An inequality similar to \eqref{intro-eq-2} with $p=r^n$, $n\in\mathbb{N}$ was proved by Hausenblas in \cite{[Hau2011]} by using a discrete approximation argument of stochastic integrals. In this paper, by following the It\^{o} formula approach in the spirit of \cite{Novikov,[Tub]} and using some ingredients from \cite{[Neerven+Zhu]}, we give a short and  simple proof of  \eqref{intro-eq-2} and extend the inequality to all $1\leq p<\infty$; see Theorem \ref{Th-burk}. 	
 In a recent work \cite{[Dir14]} by Dirksen, upper-bound inequalities similar to \eqref{intro-eq-3} and \eqref{intro-eq-main-1} with $E$ being $L^q$ spaces were obtained via a noncommutative probability approach.  	   In this paper, we propose a different approach and generalize the inequalities to all Banach spaces of martingale type $r$ with $1 < r\leq 2$.  Note that all $L^q$ spaces with $q\in[r,\infty)$ are of martingale type $r$.
	 We want to remark that our method is more direct and closer in spirit to the ideas used in Hilbert space and finite dimensional spaces.

	Now we assume that $(e^{tA})_{t\geq0}$ is a $C_0$-semigroup on $(E,|\cdot|_E)$ with a generator $A$ such that $\|e^{tA}\|\leq e^{t\alpha} $  for some $\alpha \geq  0$.  Let $X_t=\int_0^t\int_Z e^{(t-s)A}\xi(s,z)\,\tilde{N}(\d s,\d z)$. 	 In the second part of the present paper, we will prove that there exists a c\`adl\`ag modification $\bar{X}$ of $X$ such that the following inequalities hold
	\begin{enumerate}
	\item[(i)] for all $r\leq p<\infty$,
		\begin{align}
	 \E\sup_{0\leq t\leq T} \big|\bar{X}_t\big|^p\leq e^{\alpha p T}C_{p,r}\,\E\Big(\int_0^T\int_Z |\xi(s,z)|^r_E\, N(\d s,\d z)\Big)^{\frac{p}{r}},
	 \end{align}
	 \item[(ii)] for all $0< p\leq r$,
	 \begin{align}
	  \E\sup_{0\leq t\leq T} \big|\bar{X}_t\big|^p\leq e^{\alpha p T}C_{p,r}\,\E \Big(\int_0^T\int_Z|\xi(s,z)|^{r}_E\,\nu(\d z)\,\d s\Big)^{\frac{p}{r}},
		\end{align}
		\item[(iii)] for all $r\leq p<\infty$,
		\begin{align}
	 \E\sup_{0\leq t\leq T} \big|\bar{X}_t\big|^p\leq e^{\alpha p T}C_{p,r}\Big[\E\Big(\int_0^T\int_Z |\xi(s,z)|^p_E\,\nu(\d z)\,\d s\Big)
		 +\E \Big(\int_0^T\int_Z|\xi(s,z)|^{r}_E\,\nu(\d z)\,\d s\Big)^{\frac{p}{r}}\Big].\label{intro-main-max}
	\end{align}
	\end{enumerate}

	Considering maximal inequalities in infinite dimensional spaces, there exist mainly three different approaches in literature. The first approach is based on the factorization method and stochastic Fubini theorem introduced by Da Prato, Kwapie\'{n} and Zabczyk in \cite{[Da+Kw+Zab]}. By following this approach, a weaker inequality of  $L^p$ estimates of stochastic convolutions for Wiener processes in Hilbert spaces was obtain in \cite{[Da+Zab]}. Another approach is to apply isometric dilation theorems on semigroups that admit dilations; see \cite{[HauSei]}, \cite{[Sei]}, \cite{[VerWei]} and the references therein.
		This unitary dilation method was used in \cite{[HauSei]} to get the maximal inequality for a stochastic convolution driven by a Wiener process and $C_0$-contraction semigroups in a Hilbert space $H$. Later  Seidler in \cite{[Sei]} obtained $L^p$-estimates for stochastic convolutions of positive contraction semigroups in a 2-smooth Banach space $E=L^q(\mu)$, $q\geq 2$ with a sharp constant $C_p=o(\sqrt{p})$.
		We also refer to Veraar and Weis \cite{[VerWei]}, Dirksen, Maas and van Neerven \cite{[DMN]} and Dirksen \cite{[Dirksen]}.
		
	A classical approach
	to prove the
	maximal inequality for stochastic convolution in finite dimensional spaces is to apply the It\^o  formula to a $C^2$-mapping $E \ni x\mapsto |x|^p_E$, $2\leq p<\infty$, and then the classical Burkholder inequality; see e.g. \cite{[Tub]}, \cite{[Ichikawa]}. When $E$ is a Hilbert space, the function $\psi_p: E \ni x\mapsto |x|^p_E$,  is always  of $C^2$ class, see Ichikawa \cite{[Ichikawa]} and Tubaro \cite{[Tub]} for inequalities of stochastic convolutions in Hilbert spaces. While dealing with more general Banach spaces, we encounter a difficulty that the mapping $\psi_p$ may  not even be twice Fr\'{e}chet differentiable. %For this
	In the Gaussian case, this particular problem is addressed in \cite{[Brz+Pes]} where the second named author and Peszat made a certain assumption on $E$, which they call $(H_p)$. Under a similar assumption on the Banach space $E$, the first two authors of this paper and Hausenblas derived in \cite{Zhu2017} a version of maximal inequality for L\'{e}vy-type noises. More precisely, they assumed  that  for some $p\in[2,\infty)$, the function $\psi_p$  is of $C^2$ and its first and second Fr\'{e}chet derivatives are bounded by some constant multiples of $|x|_E^{p-1}$ and $|x|_E^{p-2}$. By means of the It\^o  formula and the Davis inequality, they obtained a maximal inequality for contraction semigroups.
	Note that various spaces do have martingale type $2$ property, but fail to satisfy condition $(H_p)$.  For instance, $L^2(L^q(0,1))$, $q>2$, are martingale type $2$ Banach spaces, but according to Theorem 3.9 in \cite{[Leo+Sun]}, the norm of  $L^2(L^q(0,1))$ is not twice Fr\'{e}chet differentiable away from the origin.
In this paper, we follow the It\^o formula approach and work in a more general setting of Banach spaces. Compared to \cite{[Brz+Pes]} and \cite{Zhu2017}, in the present paper we weaken the assumptions on the Banach spaces. We only assume that the Banach space is of martingale type $r$, $1<r\leq 2$. Note that this is not an easy question. Due to the lack of the twice differentiability property of the $p$-th power of the norm in martingale type $r$ Banach spaces, a straightforward extension of the inequalities seems impossible and a new technique is required.
To do this, we derive two general versions of It\^{o}'s formula in Banach spaces (see Appendix B) which can be applied to the norm function $\psi_p$ directly.
 By following the It\^{o} formula approach and employing the BDG inequality, e.g. \eqref{intro-eq-2}, it allows us to estimate each term properly.   Our proofs here are more succinct.

Let $(e^{tA})_{t\geq0}$ be a $C_0$-contraction semigroup on a martingale type $2$ Banach space $(E,|\cdot|_E)$
		and let $\xi\in\mathcal{M}_T^2(\mathcal{P}\otimes\mathcal{Z},\d t\times\mathbb{P}\times\nu;E)$.
 We will also show the following exponential tail estimates for stochastic convolutions driven by compensated Poisson random measures: if there exists $\lambda>0$ and $M_{\lambda}>0$ such that
	    \begin{align}
	    	    \int_0^T\int_Z e^{\lambda^{\frac{1}{2}}|\xi(s,z)|_E}\lambda|\xi(s,z)|_E^2\,\nu(\d z)\,\d s\leq M_{\lambda},	     \end{align}
	then for every $R>0$ there exists
	a constant $C_{\lambda}>0$ such that
	\begin{align}\label{exp-tail-eq-intro}
		    \mathbb{P}\Big( \sup_{0\leq t\leq T}|X_t|_E \geq R \Big)\leq C_{\lambda}e^{-(1+\lambda R^2)^{\frac{1}{2}}}.
	\end{align}

	Let $E$ be a martingale type $2$ Banach space. We will also study stochastic convolution processes of the form
	\begin{align}
		  X_t=\int_0^te^{(t-s)A}g_s\, \d W_s+\int_0^t\int_Z e^{(t-s)A}\xi(s,z)\,\tilde{N}(\d s,\d z),\;t\in[0,T],\;T>0.
	\end{align}
	Here $g:\mathbb{R}_+\times\Omega\rightarrow \gamma(H;E)$ is a progressively measurable process which is stochastic integrable with respect to a cylindrical Wiener process $W$ on $H$ and $\xi$ is defined as above.
	 We will prove that there exists a c\`adl\`ag modification $\bar{X}$ of $X$ such that
	 	\begin{align}
		\E\sup_{0\leq t\leq T} |\bar{X}_t|_E^p\leq e^{\alpha T}C_{p,r}&\left[\E\Big(\int_0^T \|g_s\|^2_{\gamma(H,E)}\,\d s\Big)^{\frac{p}2}+\E\Big(\int_0^T\int_Z |\xi(s,z)|^p_E\,\nu(\d z)\,\d s\Big)\right.\nonumber\\
		 &\left.+\E \Big(\int_0^T\int_Z|\xi(s,z)|^{2}_E\,\nu(\d z)\,\d s\Big)^{\frac{p}{2}}\right]\text{  for all }2\leq p<\infty.\label{intro-eq-10}
	\end{align}

We emphasize that maximal inequalities developed in this paper  are applicable to many nonlinear SPDEs, including stochastic Euler equations, stochastic reaction-diffusion equations, and stochastic Navier-Stokes equations, etc. As an example we consider in Section \ref{SNSE} the existence and uniqueness of mild solutions of stochastic two dimensional quasi-geostrophic equations via the application of maximal inequalities. We establish the existence and uniqueness of mild solutions of stochastic quasi-geostrophic equations under much weaker assumptions in terms of $L^p$ theory.

The rest of the paper is organized as follows. In the next section we will study BDG inequalities for stochastic integrals. Maximal inequalities and exponential estimates for stochastic convolutions will be treated in Section \ref{sec-max-ineq}. In Section 4 we show the existence and uniqueness of solutions to stochastic quasi-geostrophic equations as the application of our main results. Finally, in Appendix A we give a brief review of some results of stochastic integral w.r.t. Poisson random measures, while in Appendix B we prove two versions of It\^{o}'s formula for Banach space valued L\'{e}vy processes.

	 \section{\bf{BDG inequalities for stochastic integrals driven by L\'{e}}vy  processes}

	      Throughout the whole section we assume that $(E,|\cdot|_E)$ is a real separable Banach space of martingale type $r$, $1<r\leq 2$, see Appendix A for the definition.
	      Let  $(\Omega,\mathcal{F},\mathbb{F},\mathbb{P})$, where $\mathbb{F} = (\mathcal{F}_t )_{t\geq 0}$, be a filtered probability space satisfying the usual hypothesis.
	
	      Let $(Z,\mathcal{Z})$ be a measurable space and $\nu$ be a $\sigma$-finite measure on it.  Assume that  $\pi$ is a $\sigma$-finite and stationary Poisson point process with the characteristic measure $\nu$ such that the counting measure
  $N$ associated with $\pi$ is a time homogeneous Poisson random measure; see Appendix A for the existence of such a process and also \cite{[Ikeda]}. We denote by $\mathcal{P}$ the predictable $\sigma$-field on $[0,T]\times\Omega$, i.e. the $\sigma$-field generated by all left continuous and $\mathcal{F}_t$-adapted
		 real-valued processes on $[0,T]\times\Omega$.

	   For the convenience of the reader we repeat some notations and standard facts about stochastic integration w.r.t. the compensated Poisson random measure which we will be frequently used in the paper. The detailed discussions can be found in Appendix A.
	
	   We use the symbol $\mathcal{M}_T^r(\mathcal{P}\otimes\mathcal{Z},\d t\times\mathbb{P}\times\nu;E)$
		 to denote the space of all $\mathcal{P}\otimes\mathcal{Z}$-measurable functions
		$\xi:[0,T]\times\Omega\times Z\to E$ such that $\xi\in L^r(\Omega;L^r([0,T]\times Z;E))$.
			
		 Take $\xi\in\mathcal{M}_T^r(\mathcal{P}\otimes\mathcal{Z},\d t\times\mathbb{P}\times\nu;E)$. According to \cite{[Brz-Hau]}, the stochastic integral process
		 \begin{align}
		 I_t(\xi)=\int_0^t\int_Z \xi(s,\cdot,z)\,\tilde{N}(\d s,\d z),\;\; t\in[0,T],
		 \end{align}
		  is a c\`{a}dl\`{a}g $r$-integrable $E$-valued martingale and it satisfies the following: inequality
	         	\begin{align}\label{sec-2-eq-10}
		\mathbb{E}\Big{|}\int_0^T\int_Z \xi(s,z)\,\tilde{N}(\d s,\d z)\Big{|}_E^r
		\leq C_r(E)\,\mathbb{E}\int_0^T\int_Z|\xi(s,z)|_E^r\,\nu(\d z)\,\d s.
		\end{align}
		If
			$ \xi $ is a $\mathcal{P}\otimes\mathcal{Z}$-measurable function in $\mathcal{M}_T^1(\mathcal{P}\otimes\mathcal{Z},\d t\times\mathbb{P}\times\nu;E)$,
			then one can show (see Appendix A for the proof) that
			\begin{align}\label{sec-2-eq-13-1}
				\E\int_0^T\int_Z\xi(s,z)\,N(\d s,\d z)=\E\int_0^T\int_Z\xi(s,z)\,\nu(\d z)\,\d s.
			\end{align}
		
			In particular, if
			$\xi\in\mathcal{M}_T^1(\mathcal{P}\otimes\mathcal{Z},\d t\times\mathbb{P}\times\nu;E)\cap \mathcal{M}_T^r(\mathcal{P}\otimes\mathcal{Z},\d t\times\mathbb{P}\times\nu;E)$, then the stochastic integral $I_t(\xi)$ can be written as a sum of two Bochner integrals
			\begin{align}\label{eq-108}
			     \int_0^t\int_Z\xi(s,\cdot,z)\,\tilde{N}(\d s,\d z)=
			     \sum_{s\in(0,t]\cap
			     \mathcal{D}(\pi)}\xi(s,\cdot,\pi(s))-\int_0^t\int_Z\xi(s,\cdot,z)\,\nu(\d z)\,\d s,\quad \P\text{-a.s.}
			\end{align}
			Here $\mathcal{D}(\pi)$ is the domain of the Poisson point process $\pi$.

	  Set
		\begin{align}\label{u_t-def}
		u_t:=\int_0^t\int_Z\xi(s,z)\,\tilde{N}(\d s,\d z),\;\; t\in[0,T].
		\end{align}
	In this section, we aim to establish some types of inequalities and $L^p$ estimates for $(u_t)_{0\leq t\leq T}$.
	
		%--------------------------
		\begin{rem}\label{rem-lem-Kallenberg}  Let $f:[0,T]\times\Omega\times Z\rightarrow \mathbb{R}$ be a real-valued function in $\mathcal{M}_T^1(\mathcal{P}\otimes\mathcal{Z},\d t\times\mathbb{P}\times\nu;\mathbb{R})$.  			Recall that $\int_0^t\int_Z |f(s,z)|\tilde{N}(\d s,\d z)$, $t\in[0,T]$, is an $\mathbb{R}$-valued and $1$-summable martingale; see \cite{[Ikeda]}. Moreover, it can be written as a sum of two Lebesgue integrals in the following way:
			\begin{align*}
				   \int_0^t\int_Z |f(s,z)|\,\tilde{N}(\d s,\d z)=\int_0^t\int_Z |f(s,z)|\,N(\d s,\d z)-\int_0^t\int_Z |f(s,z)|\,\nu(\d z)\,\d s,\ \ \mathbb{P}\text{-a.s.}
				\end{align*}
	So we can deduce that
			\begin{align*}
			        \E\big( \int_0^T\int_Z|f(s,z)|\,N(\d s,\d z)\big|\mathcal{F}_t\big)
			        \geq\E\Big( \int_0^T\int_Z|f(s,z)|\,\nu(\d z)\,\d s-\int_0^t\int_Z|f(s,z)|\nu(\d z)\,\d s\Big|\mathcal{F}_t\Big).
			\end{align*}
			  A simple application of Proposition 25.21 in \cite{[Kallenberg]} with $X=\int_0^T\int_Z|f(s,z)|\,N(\d s,\d z)$ and \\ $A_t=\int_0^t\int_Z|f(s,z)|\,\nu(\d z)\,\d s$, $t\in [0,T]$, yields for all $1\leq p<\infty$,
		   $$ \E\Big(\int_0^T\int_Z|f(s,z)|\,\nu(\d z)\,\d s\Big)^{p} \leq p^p\,\E\Big(\int_0^T\int_Z|f(s,z)|\,N(\d s,\d z)\Big)^p.$$
		\end{rem}
	
	    In Proposition \ref{prop-main-1} and Theorem \ref{Th-burk} below, we will formulate and prove some inequalities for stochastic integrals with respect to compensated Poisson random measures.  We shall follow a common strategy for proving a BDG inequality for real-valued local martingales used in \cite{[Kallenberg]} to obtain inequality \eqref{burk-eq-2}. The idea of this proof can be traced back to \cite{[Davis]}, where a version of the BDG inequality was derived for discrete-time martingales with $p=1$.
	 Similar estimates, analogous to \eqref{main-eq-1}, \eqref{Lr-1}, and \eqref{Lp-1} in finite dimension space, were first studied by Novikov in \cite{Novikov}. The main ingredient of his proof is a version of the change-of-variable formula (It\^{o} formula) for the transformation of a stochastic integral which requires the function only to be continuously differentiable.
	  In a Hilbert space setting, $r=2$, inequalities \eqref{main-eq-1} and \eqref{burk-eq-2} together with the inequality \eqref{burk-eq-1} in Theorem \ref{Th-burk} can be read as
	\begin{align}
	         &\E\sup_{t\in[0,T]} |u_t|_E^p\leq C\,\E\,\langle u\rangle_T^{p/2}, \ \ 0<p<2,\label{rem-1-eq-1}\\
	         & \E\sup_{t\in[0,T]} |u_t|_E^p\leq C\,\E\,[ u]_T^{p/2}, \ \ 1\leq p<\infty.\label{rem-1-eq-2}
	\end{align}
	Here $\langle u\rangle_t=\int_0^t\int_Z|\xi(s,z)|^2\,\nu(\d z)\d s$ and $[u]_t=\int_0^t\int_Z|\xi(s,z)|^2\,N(\d s,\d z)$ are, respectively, the Meyer  and the quadratic variation processes of $u$; see \cite{[Metivier]} for the definitions. Inequality \eqref{rem-1-eq-2} is known as the BDG inequality. Inequalities \eqref{rem-1-eq-1} with $p=2$ and \eqref{rem-1-eq-2} with $2\leq p<\infty$ were studied by M\'{e}tivier \cite{[Metivier]} for Hilbert space valued right-continuous local martingales. Ichikawa in \cite{[Ichikawa]} established a stopped version of inequality \eqref{rem-1-eq-1} for right-continuous martingales in Hilbert spaces. A BDG inequality for a stochastic integral with respect to a cylindrical Brownian motion in Orlicz-type spaces was obtained recently by Xie and Zhang in \cite{XZ17}.
	
	Let us point out that inequality \eqref{main-eq-1} below has already  been stated and proved in \cite{[Brz-Hau]}, but we include it here for  the readers' convenience.
	
	Throughout the paper, we use $C_{a_1,a_2,\cdots}$ to denote a generic positive constant whose value may change from
line to line, but depends only on the designated variables $a_1,a_2,\cdots$.

		\begin{prop}\label{prop-main-1}
		  Let $\xi\in\mathcal{M}_T^r(\mathcal{P}\otimes\mathcal{Z},\d t\times\mathbb{P}\times\nu;E)$ and $(u_t)_{t\in[0,T]}$ be defined by \eqref{u_t-def}.  Then the following inequalities hold provided the expressions on the right hand sides are finite:
		   \begin{enumerate}
		      \item for $0<p\leq r$,
		      	\begin{align}\label{main-eq-1}
	       \E\sup_{t\in[0,T]} |u_t|_E^p\leq C_{p,r}\,\E\Big(\int_0^T\int_Z|\xi(s,z)|_E^r\,\nu(\d z)\,\d s\Big)^{\frac{p}{r}};
	\end{align}
		      \item for $1\leq p\leq r$,
		   \begin{align}\label{burk-eq-2}
		          \E\sup_{t\in[0,T]} |u_t|_E^p\leq C_{p,r}\,\E\Big(\int_0^T\int_Z|\xi(s,z)|_E^r\,N(\d s,\d z)\Big)^{\frac{p}{r}};
		   \end{align}
		   \item for $1\leq p\leq r$,
		   \begin{align}\label{Lr-1}
		       \E\sup_{t\in[0,T]} |u_t|_E^p\leq C_{p,r}\,\E\int_0^T\int_Z|\xi(s,z)|^p_E\,\nu(\d z)\,\d s;
		   \end{align}
		      \item for $1\leq p\leq r$,
	          \begin{align}\label{Bur-P-1}
		       \E\sup_{t\in[0,T]} |u_t|_E^p\leq C_{p,r}\,\E\int_0^T\int_Z|\xi(s,z)|_E^p\,N(\d s,\d z).
		   \end{align}
		   \end{enumerate}
		\end{prop}

		\begin{proof}
			(1) See Corollary C.2 in \cite{[Brz-Hau]}.\\	
	 \noindent(2). Let us fix $r \geq 1$. Note that $\triangle u_t=\xi(t,\pi(t))1_{\mathcal{D}(\pi)}(\pi(t))$ and also observe that
			\begin{align}
			\begin{split}\label{prp-main-eq-1}
				\E \sum_{s\leq t}|\xi(s,\pi(s))|_E1_{\{|\xi(s,\pi(s))|_E>2\sup\limits_{v<s}|\xi(v,\pi(v))|_E\}}&\leq 2\E\sum_{s\leq t}\Big(\sup_{v\leq s}|\xi(v,\pi(v))|_E-\sup_{v<s}|\xi(v,\pi(v))|_E\Big) \\
				&\leq 2\E\sup_{s\leq t} |\xi(s,\pi(s))|_E\leq 2\E (\sum_{s\leq t}|\xi(s,\pi(s))|_E^r)^{\frac{1}{r}}\\
				&=2\E\Big(\int_0^t\int_Z|\xi(s,z)|_E^r\,N(\d s,\d z)\Big)^{\frac{1}{r}}.
								\end{split}
	        \end{align}
	Thus we can define a process $A_t$, $t\in[0,T]$, by
	    \begin{align*}
			    A_t&:=\sum_{s\leq t}\xi(s,\pi(s))1_{\mathcal{D}(\pi)}(\pi(s))1_{\{|\triangle u_s|_E>2\sup\limits_{v<s}|\triangle u_v|_E\}},\;\;t\in[0,T].
			\end{align*}
	Since the process $(u_t)$ is right-continuous with left limits, the function $X_t:=\sup_{s<t}|\triangle u_s|_E$
	is left-continuous. So this together with the adaptedness implies that $X_t$ is $\mathcal{P}$-measurable. By assumption, the function $\xi$ is $\mathcal{P}\otimes\mathcal{Z}$-measurable. Hence we infer that 	$\{(s,\omega,z):|\xi(s,\omega,z)|_E>2\sup_{s< t}|\triangle u_s|_E\}\in\mathcal{P}\otimes\mathcal{Z}$. Similarly, one can find that $\{(s,\omega,z):|\xi(s,\omega,z)|_E\leq2\sup_{s< t}|\triangle u_s|_E\}\in\mathcal{P}\otimes\mathcal{Z}$.\\
	  Notice that $A_t=\int_0^t\int_Z\xi(s,z)1_{\{|\xi(s,z)|_E>2\sup\limits_{v<s}|\triangle u_v|_E\}}\;N(\d s,\d z)$.
	
	 By the predictability of the function $\xi(s,z)1_{\{|\xi(s,z)|_E>2\sup\limits_{v<s}|\triangle u_v|_E\}}$, we have
	\begin{align*}
		   \E\int_0^t\int_Z\xi(s,z)1_{\{|\xi(s,z)|_E>2\sup\limits_{v<s}|\triangle u_v|_E\}}\,\nu(\d z)\,\d s=\E\int_0^t\int_Z\xi(s,z)1_{\{|\xi(s,z)|_E>2\sup\limits_{v<s}|\triangle u_v|_E\}}\,N(\d s,\d z)<\infty.
		\end{align*}
		 Hence by Equality \eqref{A.prop-3-1}, we infer that the process  $ \hat{A}_t:=\int_0^t\int_Z\xi(s,z)1_{\{|\xi(s,z)|_E>2\sup\limits_{v<s}|\triangle u_v|_E\}}\,\nu(\d z)\,\d s$, $t\in [0,T]$ is the compensator of the process $A$. Moreover the process
		      \begin{align*}
		    	   D_t:=A_t-\hat{A}_t= \int_0^t\int_Z\xi(s,z)1_{\{|\xi(s,z)|_E>2\sup\limits_{v<s}|\triangle u_v|_E\}}\,\tilde{N}(\d s,\d z), \;\; t\in [0,T]
		      \end{align*}
	is an $L_1$-integrable martingale. Hence,
			by Remark \ref{rem-lem-Kallenberg}  and an argument analogous to \eqref{prp-main-eq-1}, we infer, for every $1\leq p<\infty$,
			\begin{align}
			\begin{split}\label{pro-N-eq-11}
				  \E\sup_{t\leq T}|D_t|_E^p
				  				\leq& 2^{p-1}\E(A_T)^p
				+2^{p-1}\E(\hat{A}_T)^p\\
				         \leq& 2^{p-1}(1+p^p)\E\Big(\int_0^T\int_Z|\xi(s,z)|_E1_{\{|\xi(s,z)|_E>2\sup\limits_{v<s}|\triangle u_v|_E\}}\,N(\d s,\d z)\Big)^p \\
				      \leq& C_p \E\Big(\int_0^t\int_Z|\xi(s,z)|_E^r\, N(\d s,\d z)\Big)^{\frac{p}{r}}.
				      \end{split}
			\end{align}	
			Set
			\begin{align*}
				    H_t:=u_t-D_t
				=\int_0^t\int_Z\xi(s,z)1_{\{|\xi(s,z)|_E\leq2\sup\limits_{v<s}|\triangle u_v|_E\}}\,\tilde{N}(\d s,\d z),\;\; t\in[0,T].
			\end{align*}
			Clearly, the process $H=(H_t)_{t\in[0,T]}$,  is an $r$-integrable $E$-valued martingale. 	
			Let us fix an auxiliary number $\lambda>0$ and introduce the following random variable
			\begin{align}\label{eqn-tau lambda}
				     \tau_{\lambda}:=\inf\Big\{t:\Big(\int_0^t\int_Z|\xi(s,z)|_E^r1_{\{|\xi(s,z)|_E\leq 2\sup\limits_{v<s}|\triangle u_v|_E\}}\,N(\d s,\d z)\Big)^{\frac{1}{r}}\vee\sup_{s\leq t}|\triangle u_s|_E>\lambda\Big\}\wedge T.
			 \end{align}
		   By the right-continuity of $\int_0^t\int_Z|\xi(s,z)|_E^r1_{\{|\xi(s,z)|_E\leq 2\sup\limits_{v<s}|\triangle u_v|_E\}}\,N(\d s,\d z)$ and $\sup\limits_{s\leq t}|\triangle u_s|_E$, we infer that $\tau_{\lambda}$ is a stopping time; see Example I.5.1 in \cite{[Ikeda]}.
			Hence, 			by using first the Chebyshev and Doob inequalities and then \eqref{eqn-tau lambda} we infer
			\begin{align*}
				    \mathbb{P}\{\sup_{t\leq \tau_{\lambda}}|H_t|_E>\lambda \}
								&\leq \frac{C_r}{\lambda^r}\,\E\Big(\int_0^{\tau_{\lambda}}\int_Z|\xi(s,z)|_E^r1_{\{|\xi(s,z)|_E\leq2\sup\limits_{v<s}|\triangle u_v|_E\}}\,N(\d s,\d z)\Big)\\
							&\leq \frac{C_r}{\lambda^r}\,\E\Big(\int_0^{\tau_{\lambda}-}\int_Z|\xi(s,z)|_E^r1_{\{|\xi(s,z)|_E\leq2\sup\limits_{v<s}|\triangle u_v|_E\}}\,N(\d s,\d z)+2^r\sup_{v<\tau_{\lambda}}|\xi(v,\pi(v))|_E^r\Big)  \\
			  			&\leq C_r\frac{\lambda^r}{\lambda^r}.
			  \end{align*}
			  Here $C_r$ is a constant depending only on $r$ but whose value may change from line to line.
		                  Hence it follows that
		\begin{align}\label{Pro-N-eq-10}
			   \mathbb{P}\{\sup_{t\leq \tau_{\lambda}}|H_t|_E>\lambda \}\leq C_r\frac{1}{\lambda^r}\,\E\Big(\Big(\int_0^T\int_Z|\xi(s,z)|_E^r\, N(\d s,\d z)\Big)\wedge \lambda^r\Big).
		 \end{align}
		Observe that
		\begin{align}
		\begin{split}\label{eq-Th-1-proof-1}
			  \mathbb{P}\{\sup_{t\leq T}|H_t|_E>\lambda\}&\leq \mathbb{P}\{\sup_{t\leq T}|H_t|_E>\lambda,\tau_{\lambda}=T\}+\mathbb{P}\{\tau_{\lambda}<T\}\\
			&\leq\mathbb{P}\{\sup_{t\leq \tau_{\lambda}}|H_t|_E>\lambda\}+\mathbb{P}\{\sup_{t\leq T}|\xi(t,\pi(t))|_E>\lambda\}\\
			&\hspace{0.5cm}+\mathbb{P}\Big\{ \Big(\int_0^T\int_Z|\xi(s,z)|_E^r1_{\{|\xi(s,z)|_E\leq2\sup\limits_{v<s}|\triangle u_v|_E\}}\,N(\d s,\d z)\Big)^{\frac{1}{r}}>\lambda \Big\}.
			\end{split}
		\end{align}
		Let us now fix $p$:  $1\leq p<r$. 		By using the following standard equalities from \cite{[Ichikawa]},
		 \begin{align*}
					      \E X^p&=\int_0^{\infty}p\lambda^{p-1}\mathbb{P}\{X>\lambda\}\,\d\lambda,\\
					    \E X^p&=\frac{r-p}{r}\int_0^{\infty}\E(X^r\wedge \lambda^r) p\lambda^{p-r-1}\,\d\lambda,
				\end{align*}	
		as well as  inequalities \eqref{pro-N-eq-11}, \eqref{Pro-N-eq-10}, and \eqref{eq-Th-1-proof-1}, we get
		\begin{align*}
			  \E \sup_{s\leq T}|H_t|_E^p
			  			                          \leq& \int_0^{\infty}p\lambda^{p-1}\mathbb{P}\{\sup_{t\leq \tau_{\lambda}}|H_t|_E>\lambda\}\,\d\lambda
			+ \int_0^{\infty}p\lambda^{p-1}\mathbb{P}\{\sup_{t\leq T}|\xi(t,\pi(t))|_E>\lambda\}\,\d\lambda\\
			&+\int_0^{\infty}p\lambda^{p-1}\mathbb{P}\Big\{ \int_0^T\int_Z|\xi(s,z)|_E^r1_{\{|\xi(s,z)|_E\leq2\sup\limits_{v<s}|\triangle u_v|_E\}}\;N(\d s,\d z)\Big)^{\frac{1}{r}}>\lambda \Big\}\,\d\lambda\\
			\leq& C_r\,\int_0^{\infty}p\lambda^{p-r-1}\E\Big(\Big(\int_0^T\int_Z|\xi(s,z)|_E^r\, N(\d s,\d z)\Big)\wedge \lambda^r\Big)\;\d\lambda\\
			&+\E\Big(\sup_{t\leq T}|\xi(s,\pi(s))|_E\Big)^{p}
			+\E\Big( \int_0^T\int_Z|\xi(s,z)|_E^r
						\,N(\d s,\d z)\Big)^{\frac{p}{r}}\\
			\leq& C_{p,r}\,\E\Big( \int_0^T\int_Z|\xi(s,z)|_E^r\,N(\d s,\d z)\Big)^{\frac{p}{r}}.  	
		\end{align*}
		This together with \eqref{pro-N-eq-11} allows us to infer
		\begin{align*}
			\E\sup_{t\leq T}|u_t|_E^p
			       \leq C_{p,r}\,\E\Big( \int_0^T\int_Z|\xi(s,z)|_E^r\,N(\d s,\d z)\Big)^{\frac{p}{r}}.
		 \end{align*}
		%\item[(3)]
	  (3).  Let us fix numbers $r\geq 1 $ and $p\in [1,r]$. Interpreting $\int_0^T\int_Z|\xi(s,z)|_E^r\,N(\d s,\d z)$ as $\sum\limits_{s\in[0,T]}|\xi(s,\pi(s))|^r$, we easily obtain
	  	\begin{align}
		\begin{split}\label{burk-proof-nu-1}
		     \E\big(\int_0^T\int_Z|\xi(s,z)|_E^r\;N(\d s,\d z)\big)^{\frac{p}{r}}
		     &=\E\big(\sum_{s\in[0,T]}|\xi(s,\pi(s))|_E^r1_{\mathcal{D}(\pi)}(\pi(s))\big)^{\frac{p}{r}}\\
		     &\leq \E\sum_{s\in[0,T]}|\xi(s,\pi(s))|_E^p1_{\mathcal{D}(\pi)}(\pi(s))\\
		     &=\E\int_0^T\int_Z|\xi(s,z)|_E^p\,N(\d s,\d z)\\
		     &=\E\int_0^T\int_Z|\xi(s,z)|^p_E\,\nu(\d z)\,\d s.
		     \end{split}
		\end{align}
		Combining the above inequality with \eqref{burk-eq-2} we get the result.
		\\
	   \noindent(4).   This is implicit in the proof of part (3).
		\end{proof}

	  Before proceeding further, let us recall that according to \cite[Theorem 3.1]{[Pisier]}(see also \cite{Deville}), every Banach space $(E,|\cdot|_E)$ of martingale type $r$, $1<r\leq 2$, admits an equivalent norm, for simplicity denoted also by $|\cdot|_E$, such that $(E,|\cdot|_E)$ is an $r$-smooth Banach space and,  for every $p \in (r,\infty)$ the first derivative of the $p$-th power of the norm $\psi_p(\cdot):=|\cdot|_E^p$,  is locally $(r-1)$-H\"{o}lder continuous on $E$, i.e.
	       \begin{align}\label{Ap-rem-eq-2}
			\|\psi_p^\prime(x)-\psi_p^\prime(y)\|_{L(E)}\leq C_{p,r}(|x|_E+|y|_E)^{p-r}|x-y|_E^{r-1},\;\;\; x,y\in E.
		\end{align}
		The proof of the above inequality can be found in \cite{[Neerven+Zhu]}.
		%--------------------------Theorem-----------------------
	\begin{thm}\label{Th-burk} Suppose that $\xi\in\mathcal{M}_T^r(\mathcal{P}\otimes\mathcal{Z},\d t\times\mathbb{P}\times\nu;E)$ and $(u_t)_{t\in[0,T]}$ is defined by \eqref{u_t-def}.  Then for $r\leq p<\infty$, \dela{\coma{Don't we need to mention the modification here?}}
		\begin{align}\label{burk-eq-1}
			        \E\sup_{t\in[0,T]} |u_t|_E^p\leq C_{p,r}\,\E\Big(\int_0^T\int_Z |\xi(s,z)|_E^r \,N(\d s,\d z)\Big)^{\frac{p}{r}}.
		\end{align}

	\end{thm}

	%--------proof-----------------------	
	\begin{proof}
Let us fix $r$ and $p$:  $1<r\leq 2$, $r\leq p<\infty$. For the simplicity of notation we  put $\psi:=\psi_p$.
			  An application of the It\^o  formula \eqref{theo-Ito-3} to the function $\psi$ and the process $(u_t)_{t\in[0,T]}$ yields for $t\in[0,T]$, $\mathbb{P}$-a.s.
		\begin{align*}%\label{EQ-1}
		   \psi(u_t)
		   =I_1(t)+I_2(t),
		\end{align*}
		where
		$$I_1(t)=\int_0^t\int_Z\psi^\prime(u_{s-})(\xi(s,z))\,\tilde{N}(\d s,\d z),$$
		$$I_2(t)=\int_0^t\int_Z\Big{[}\psi(u_{s-}+\xi(s,z))-\psi(u_{s-})-\psi^\prime(u_{s-})(\xi(s,z))\Big{]}\,N(\d s,\d z).$$
		Let $\varepsilon>0$. Applying inequality \eqref{burk-eq-2} with $p=1$ gives 		
		\begin{align}
		\E\sup_{0\leq t\leq T}|I_1(t)|
		&\leq C\,\E\left(\int_0^T\int_Z|\psi^\prime(u_{s-})(\xi(s,z))|^r\,N(\d s,\d z)\right)^{\frac1r}\nonumber\\
		&\leq C_p\,\E \sup_{0\leq t\leq T}|{u_t}|_E^{p-1}\Big(\int_0^T\int_Z|\xi(s,z)|^r_E\,N(\d s,\d z)\Big)^{\frac1r}\nonumber\\
		&\leq \varepsilon C_p \,\E\sup_{0\leq t\leq T}|{u_t}|_E^{p}+C_{\frac{1}{\varepsilon},p}\,\E\Big(\int_0^T\int_Z|\xi(s,z)|^r_E\,N(\d s,\d z)\Big)^{\frac{p}{r}},\label{burk-ineq-I_1}
		\end{align}
		where in the last step we used H\"{o}lder's inequality and Young's inequality.

		On the other hand, applying the mean value theorem, we find that for every $s\in(0,T]$
		\begin{align}
		\begin{split}\label{Bur-N-proof-eq1}
			 &\Big|\psi(u_{s-}+\xi(s,\pi(s)))-\psi(u_{s-})-\psi^\prime(u_{s-})(\xi(s,\pi(s)))\Big|\\
			&\leq\Big|\psi^\prime(u_{s-}+\theta \xi(s,\pi(s)))-\psi^\prime(u_{s-})\Big|_{L(E,\mathbb{R})}|\xi(s,\pi(s))|_E.
		  \end{split}
	    \end{align}
			Observe that for all $0\leq s\leq T$,
			\begin{align*}
			|u_{s-}|_E\leq\sup_{0\leq s\leq
			T}|u_{s-}|_E\leq\sup_{0\leq s\leq
			T}|u_{s}|_E.
			\end{align*}
		  Also, since $u_{s-}+\xi(s,\pi(s))=u_s$, we get
			\begin{align*}
			|u_{s-}+\xi(s,\pi(s))|_E\leq\sup_{0\leq s\leq
			T}|u_s|_E.
			\end{align*}
		  	By using the
				fact that $|x+\theta y|_E\leq \max\{|x|_E,|x+y|_E\}$ for $0<\theta<1$ and all $x,y\in
		   	E$, we obtain
			\begin{align}\label{th-eq-2}
				   \big{|}u_{s-}+\theta\xi(s,\pi(s))\big{|}_E\leq \max\big{\{}|u_{s-}|_E,\big{|}u_{s-}+\xi(s,\pi(s))\big{|}_E\big{\}}\leq\sup_{0\leq s\leq T}|u_s|_E.
				\end{align}
	      Hence by \eqref{Ap-rem-eq-2}, we infer
	    \begin{align}\label{th-eq-1}
		|\psi^\prime(u_{s-}+\theta \xi(s,\pi(s)))-\psi^\prime(u_{s-})\Big|_{L(E,\mathbb{R})}&\leq C_{p,r}(|u_{s-}|_{E}+|u_{s-}+\theta \xi(s,\pi(s))|_{E})^{p-r}|\xi(s,\pi(x))|_{E}^{r-1}\nonumber\\
		&\leq C_{p,r}\sup_{s\in[0,T]}|u_s|^{p-r}_E|\xi(s,\pi(s))|^{r-1}_E.
	    	    \end{align}
		It then follows from \eqref{Bur-N-proof-eq1} and \eqref{th-eq-1}  that for every $s\in(0,T]$,
		\begin{align}\label{Bur-N-proof-mian-eq1}
			 &\Big|\psi(u_{s-}+\xi(s,\pi(s)))-\psi(u_{s-})-\psi^\prime(u_{s-})(\xi(s,\pi(s)))\Big|
		\leq C_{p,r}\sup_{s\in[0,T]}|u_s|^{p-r}_E|\xi(s,\pi(s))|^{r}_E.
	    \end{align}
	Hence by H\"{o}lder's  and Young's inequalities we have
	\begin{align}
		\E\sup_{t\in[0,T]}|I_2(t)|&\leq \E  \int_0^{T}\int_Z\Big{|}\psi(u_{s-}+\xi(s,z))-\psi(u(s-)-\psi^\prime(u_{s-})(\xi(s,z))\Big{|}\,N(\d s,\d z)\nonumber\\
			&\leq C_{p,r}\,\E\sup_{t\in[0,T]}|u_t|^{p-r}_E\int_0^T\int_Z |\xi(s,\pi(s))|^r_E  \,N(\d s,\d z)           \label{sec-1-eq-21}\\
				&\leq C_{p,r}\,\left(\varepsilon\,\E\sup_{t\in[0,T]}|u_t|_E^p\right)^{\frac{p-r}{p}}\Big(\Big(\frac{1}{\varepsilon}\Big)^{\frac{p-r}{r}}\,\E \Big(\int_0^{T}\int_Z|\xi(s,z)|_E^r\,N(\d s,\d z)\Big)^{\frac{p}{r}}\Big)^{\frac{r}{p}}\nonumber  \\
				&\leq C_{p,r}\,\varepsilon\, \E\sup_{t\in[0,T]}|u_t|^p_E+C_{p,r}\Big(\frac{1}{\varepsilon}\Big)^{\frac{p-r}{r}}\,\E \Big(\int_0^{T}\int_Z|\xi(s,z)|_E^r \,N(\d s,\d z)\Big)^{\frac{p}{r}}.\nonumber
	\end{align}	

				Combining \eqref{burk-ineq-I_1} and \eqref{sec-1-eq-21} and choosing a suitable small value of $\varepsilon$, we get
	 				\begin{align*}
				    \E\sup_{t\in[0,T\wedge\sigma_n]}|u_t|_E^p\leq C_{p,r}\,\E \Big(\int_0^{T}\int_Z|\xi(s,z)|_E^r \;N(\d s,\d z)\Big)^{\frac{p}{r}}.
					\end{align*}
	\end{proof}	

	 	The following $L^p$ estimate \eqref{Lp-1} of $(u_t)_{t\in[0,T]}$ for all $r\leq p<\infty$, as a consequence of Theorem \ref{Th-burk}, is an important and useful tool in studying solutions of SPDEs, especially the regularity of solutions.
	%----------------------------
	\begin{cor}\label{coro-burk-1}  For all $r\leq p<\infty$, we have
		\begin{align}\label{Lp-1}
		 \E\sup_{t\in[0,T]}\big| u_t \big|_E^{p}\leq C_{p,r}\,\Big[\E\Big(\int_0^T\int_Z |\xi(s,z)|^p_E\,\nu(\d z)\,\d s\Big)+\E \Big(\int_0^T\int_Z|\xi(s,z)|^{r}_E\,\nu(\d z)\,\d s\Big)^{\frac{p}{r}}\Big].
		\end{align}
	\end{cor}

	\begin{proof} Let us assume that the right-hand side of \eqref{Lp-1} is finite, as otherwise the inequality follows trivially.
To prove inequality \eqref{Lp-1}, we start with the case where $r<p\leq r^2$. According to the inequality \eqref{burk-eq-1}, by  taking into account the following equality
	      $$\int_0^t\int_Z|\xi(s,z)|^r_E\;N(\d s,\d z)=\int_0^t\int_Z|\xi(s,z)|_E^r\;\tilde{N}(\d s,\d z)+\int_0^t\int_Z|\xi(s,z)|_E^r\;\nu(\d z)\,\d s,$$ we obtain
	       \begin{align*}
	               \E\sup_{t\in[0,T]}\Big| \int_0^t\int_Z \xi(s,z)\; \tilde{N}(\d s,\d z) \Big|_E^{p}&\leq C_{p,r}\E\Big(\int_0^t\int_Z |\xi(s,z)|^r_E\;N(\d s,\d z)\Big)^{\frac{p}{r}}\\
	              &\leq C_{p,r}\E\Big(\int_0^t\int_Z |\xi(s,z)|_E^r\;\tilde{N}(\d s,\d z)\Big)^{\frac{p}{r}}+    C_{p,r}\E\Big(\int_0^t\int_Z |\xi(s,z)|^r_E\;\nu(\d z)\,\d s\Big)^{\frac{p}{r}}.
	       \end{align*}
	       Notice that $1<\frac{p}{r}\leq r$, hence by applying the inequality \eqref{Lr-1} to the first term on the right-hand side of the above inequality, we infer that
	       \begin{align*}
	              \E\sup_{t\in[0,T]}\Big| \int_0^t\int_Z \xi(s,z)\; \tilde{N}(\d s,\d z) \Big|_E^{p}
	                  \leq C_{p,r}\E\Big(\int_0^t\int_Z |\xi(s,z)|^p_E\;\nu(\d z)\,\d s\Big)+C_{p,r}\E\Big(\int_0^t\int_Z |\xi(s,z)|^r_E\;\nu(\d z)\,\d s\Big)^{\frac{p}{r}}.
	       \end{align*}
This completes the proof of  inequality \eqref{Lp-1} in  the case when $r<p\leq r^2$.

	       Let now assume that $r^2<p\leq r^3$.  Applying inequality \eqref{burk-eq-1} twice and then  inequality \eqref{Lr-1} yields  the following estimate:
	       \begin{align*}
	                     \E\sup_{t\in[0,T]}\Big| \int_0^t\int_Z \xi(s,z)\; \tilde{N}(\d s,\d z) \Big|_E^{p}&\leq C_{p,r}\E\Big(\int_0^t\int_Z |\xi(s,z)|^r_E\;N(\d s,\d z)\Big)^{\frac{p}{r}}\\
	              &\leq C_{p,r}\E\Big(\int_0^t\int_Z |\xi(s,z)|_E^r\;\tilde{N}(\d s,\d z)\Big)^{\frac{p}{r}}+C_{p,r}\E\Big(\int_0^t\int_Z |\xi(s,z)|^r_E\;\nu(\d z)\,\d s\Big)^{\frac{p}{r}}\\
	              &\leq C_{p,r}\E\Big(\int_0^t\int_Z |\xi(s,z)|^{r^2}_E\;N(\d s,\d z)\Big)^{\frac{p}{r^2}}+C_{p,r}\E\Big(\int_0^t\int_Z |\xi(s,z)|^r_E\;\nu(\d z)\,\d s\Big)^{\frac{p}{r}}\\
	              &\leq C_{p,r}\E\Big(\int_0^t\int_Z |\xi(s,z)|^{r^2}_E\;\tilde{N}(\d s,\d z)\Big)^{\frac{p}{r^2}}+C_{p,r}\E\Big(\int_0^t\int_Z |\xi(s,z)|^{r^2}_E\;\nu(\d z)\,\d s\Big)^{\frac{p}{r^2}}\\
	              &\hspace{1cm}+C_{p,r}\E\Big(\int_0^t\int_Z |\xi(s,z)|^r_E\;\nu(\d z)\,\d s\Big)^{\frac{p}{r}}\\
	              &\leq C_{p,r}\E\Big(\int_0^t\int_Z |\xi(s,z)|^p_E\;\nu(\d z)\,\d s\Big)+C_{p,r}\E\Big(\int_0^t\int_Z |\xi(s,z)|_E^{r^2}\;\nu(\d z)\,\d s\Big)^{\frac{p}{r^2}}\\
	              &\hspace{1cm}+C_{p,r}\E\Big(\int_0^t\int_Z |\xi(s,z)|^r_E\;\nu(\d z)\,\d s\Big)^{\frac{p}{r}}.
	       \end{align*}
  Suppose that  $r^{n-1}<p\leq r^n$ for some natural number $n  \geq 3$. Then  by induction we obtain the following estimate:
	                   \begin{align*}
	                               \E\sup_{t\in[0,T]}\Big| \int_0^t\int_Z \xi(s,z) \;\tilde{N}(\d s,\d z) \Big|^{p}\leq C_{p,r}\E\Big(\int_0^t\int_Z |\xi(s,z)|_E^p\;\nu(\d z)\,\d s\Big)+C_{p,r}\sum_{j=1}^{n-1}\E \Big(\int_0^T\int_Z|\xi(s,z)|^{r^j}_E\nu(\d z)\,\d s\Big)^{\frac{p}{r^j}}.
	                   \end{align*}
	                   Note that for $r\leq m\leq p$, by making use of H\"{o}lder's inequality and Young's inequality, we have
	                   \begin{align*}
	                         \Big(\int_0^T\int_Z|\xi(s,z)|^m_E\;\nu(\d z)\d r\Big)^{\frac{1}{m}}
	                         	                         &\leq \Big(\int_0^T\int_Z|\xi(s,z)|^r_E\;\nu(\d z)\d s\Big)^{\frac{p-m}{p-r}\cdot\frac{1}{m}} \Big(\int_0^T\int_Z|\xi(s,z)|^p_E\;\nu(\d z)\d s\Big)^{\frac{m-r}{p-r}\cdot\frac{1}{m}}\\
	                         &\leq \frac{r(p-m)}{m(p-r)} \Big(\int_0^T\int_Z|\xi(s,z)|^r_E\;\nu(\d z)\d s\Big)^{\frac{1}{r}}+\frac{p(m-r)}{m(p-r)} \Big(\int_0^T\int_Z|\xi(s,z)|^p_E\;\nu(\d z)\d s\Big)^{\frac{1}{p}}\\
	                         &\leq \Big(\int_0^T\int_Z|\xi(s,z)|^r_E\;\nu(\d z)\d s\Big)^{\frac{1}{r}}+\Big(\int_0^T\int_Z|\xi(s,z)|^p_E\;\nu(\d z)\d s\Big)^{\frac{1}{p}}.
	                   \end{align*}
	                         Hence, by combing the last two inequalities  we infer that
	                                      \begin{align*}
	                               \E\sup_{t\in[0,T]}\Big| \int_0^t\int_Z \xi(s,z)\; \tilde{N}(ds,dz) \Big|_E^{p}\leq C_{p,r}\E\Big(\int_0^T\int_Z |\xi(s,z)|^p_E\;\nu(\d z)\,\d s\Big)+C_{p,r}\E \Big(\int_0^T\int_Z|\xi(s,z)|^{r}_E\;\nu(\d z)\,\d s\Big)^{\frac{p}{r}}.
	                   \end{align*}
The proof is complete.
	\end{proof}
	\begin{rem}
	As a byproduct of  the proof of the above corollary, one can deduce  that for $r\leq p<\infty$, the following inequality holds:
		\begin{align}\label{rem-cor-eq-main}
			    \E\Big( \int_0^T\int_Z |\xi(s,z)|^r_E\; N(\d s,\d z)\Big)^{\frac{p}{r}}\leq C_{p,r}\E\Big(\int_0^T\int_Z |\xi(s,z)|^p_E\;\nu(\d z)\,\d s\Big)+C_{p,r}\E \Big(\int_0^T\int_Z|\xi(s,z)|^{r}_E\;\nu(\d z)\,\d s\Big)^{\frac{p}{r}}.
		\end{align}
	    For a comparison, let us recall Remark \ref{rem-lem-Kallenberg} which says for $p\geq r$, there holds the following inequality:
		\begin{align*}
	    	 \E\Big(\int_0^T\int_Z|\xi(s,z)|^r_E\;\nu(\d z)\,\d s\Big)^{\frac{p}{r}}\leq \Big(\frac{p}{r}\Big)^{\frac{p}{r}}\E\Big(\int_0^T\int_Z|\xi(s,z)|^r_E\;N(\d s,\d z)\Big)^\frac{p}{r}.
		\end{align*}
	  	\end{rem}

%--------------------------------------

With the help of two versions of the It\^{o} formula given  in Appendix B and Corollary \ref{coro-burk-1} we can now prove the following BDG inequality for L\'{e}vy-type processes in martingale type $2$ Banach spaces ($r=2$). Let $(W_t)_{t\geq0}$ be a cylindrical Brownian motion in a Hilbert space $H$ and let $\gamma(H;E)$ be the space of all $\gamma$-radonifying operators from $H$ to $E$. Suppose that $(g_t)_{t\geq0}$ is a progressively measurable process with values in $\gamma(H;E)$ such that
	\begin{align*}
		\E\Big(\int_0^T\|g_s\|^2_{\gamma(H;E)}\d s\Big)^{\frac{p}2}<\infty.
	\end{align*}
	We may denote by $M^p([0,T];\gamma(H;E))$ the space of such functions $(g_t)_{t\geq0}$. Let's recall some basic properties of stochastic integrals which we shall use in the proof; see \cite{[Dett],[Ondrejat]}.
	
	 \begin{enumerate}
\item For all $g \in M^2([0,T];\gamma(H,E))$ and $0\le u<t\le T$,
 \begin{align}\label{Iso-ine}
     \E\Big(\Big| \int_u^tg_s\,\d W_s\Big|_E ^2|\F_u\Big)\leq
C\E\Big(\int_u^t\n \xi_s\n ^2_{\gamma(H,E)}\, \d s\,|\F_u\Big).
 \end{align}
\item (Burkholder's inequality)
For all $0<p<\infty$ there exists a constant $C$, depending only on $p$ and $E$, such
that for all $g\in M^p([0,T];\gamma(H,E))$ and $t\in [0,T]$,
 \begin{align}\label{burk-ine}
    \E\sup_{s\in[0,t]}\Big| \int_0^sg_u\,\d W_u\Big|_E ^p\leq
C\E\Big(\int_0^t\n g_s\n ^2_{\gamma(H,E)}\,\d s\Big)^{\frac{p}{2}}.
 \end{align}
\end{enumerate}
%---------------------Theorem-----------------------
\begin{thm}\label{theo-Ito-2} Suppose that $r=2$, that is, $(E,|\cdot|_E)$ is a martingale type $2$ Banach space. Let $p\geq 2$ and let $(a_t)_{t\in[0,T]}$ be an $E$-valued progressively measurable process such that
$$\E\Big(\int_0^T|a_t|_E\,\d t\Big)^p<\infty$$
and let $(g_t)_{t\in[0,T]}$ be a process in $M^2([0,T];\gamma(H;E))$ and $\xi$ be a process in $\mathcal{M}_T^2(\mathcal{P}\otimes\mathcal{Z},\d t\times\mathbb{P}\times\nu;E)$.
 Let $X$ be a process given by
\begin{align}\label{Ito-process}
X_t=\int_0^t a_s\,\d s+\int_0^tg_s\,\d W_s+\int_0^t\int_Z \xi(s,z)\,\tilde{N}(\d s,\d z).
\end{align}
Then there exists a constant $C_p$, depending only on $p$ and $E$, such that
\begin{align}
       \E\sup_{0\leq t\leq T}|X_t|_E^p\leq & C_p\,\E\Big(\int_0^T|a_s|_E\;\d s\Big)^p+C_p\,\E\Big(\int_0^T\|g_s\|^2_{\gamma(H;E)}\d s\Big)^{\frac{p}2}\nonumber\\
       &+C_p\,\E \Big(\int_0^T\int_Z|\xi(s,z)|^{2}_E\;N(\d s,\d z)\Big)^{\frac{p}{2}}\quad \text{for }2\leq p<\infty,\label{burk-ineq-levy-1}
\end{align}
and
\begin{align}
       \E\sup_{0\leq t\leq T}|X_t|_E^p\leq & C_p\,\E\Big(\int_0^T|a_s|_E\;\d s\Big)^p+C_p\,\E\Big(\int_0^T\|g_s\|^2_{\gamma(H;E)}\d s\Big)^{\frac{p}2}+C_p\,\E\Big(\int_0^T\int_Z |\xi(s,z)|^p_E\;\nu(\d z)\,\d s\Big)\nonumber\\
		 &+C_p\,\E \Big(\int_0^T\int_Z|\xi(s,z)|^{2}_E\;\nu(\d z)\,\d s\Big)^{\frac{p}{2}} \quad\text{for }2\leq p<\infty.\label{burk-ineq-levy-2}
\end{align}
In particular, if  we make the additional assumption that
\begin{align}\label{app-Ito-2-assu}
\E N(t,Z)<\infty\text{ for }t\in(0,T]  \quad\text{ and}\quad\psi^\prime(X_s)(a_s)\leq 0 \quad\text{for all }(s,\omega)\in[0,T]\times\Omega,
\end{align}
then we have
\begin{align}
       \E\sup_{0\leq t\leq T}|X_t|_E^p\leq & C_p\,\E\Big(\int_0^T\|g_s\|^2_{\gamma(H;E)}\,\d s\Big)^{\frac{p}2}+C_p\,\E\Big(\int_0^T\int_Z |\xi(s,z)|^p_E\;\nu(\d z)\,\d s\Big)\nonumber\\
		& +C_p\,\E \Big(\int_0^T\int_Z|\xi(s,z)|^{2}_E\;\nu(\d z)\,\d s\Big)^{\frac{p}{2}}\quad \text{for }2\leq p<\infty.
\end{align}
\end{thm}
%-----------------Proof-------------------------------------

\begin{proof}[Proof of Theorem \ref{theo-Ito-2}]
Since the Poisson point process $\pi$ is $\sigma$-finite, there exists a sequence of sets $\{D_n\}_{n\in\mathbb{N}}$ such that $D_n\subset D_{n+1}$, $\cup_{n\in\mathbb{N}}D_n=Z$ and $\E N(T,D_n)<\infty$ for $n\in\mathbb{N}$. So there is contained only a finite number of jumps of $\pi$ in $D_n$ and at most a countably many number of jumps in $Z$. We shall denote by $\{\tau_m\}_{m=1}^{\infty}$ the corresponding sequence of jump times of  $\pi$ until time $T$ in $Z$.  In this way we order the set $\{s \in\mathcal{D}(\pi)\cap (0,T]: \pi(s)\in Z \}$ according to magnitude by $0<\tau_1<\tau_2<\cdots<\tau_m<\cdots\leq T$.  See more details in Appendix B.
Define a sequence $\{\xi^n\}_{n\in\N}$ of functions by
       \begin{align*}
           \xi^n(s,\omega,z):=\xi(s,\omega,z)1_{D_n}(z),\ (s,\omega,z)\in\mathbb{R}_+\times\Omega\times Z,\  n\in\N,
       \end{align*}
          and a sequence $\{X^n\}_{n\in\N}$ of process $X^n:=(X^n_t)_{t\geq0}$ by
             \begin{align*}
                  X^n_t=\int_0^ta_s\d s+\int_0^tg_s\,\d W_s+\int_0^t\int_Z \xi^n(s,z)\,\tilde{N}(\d s,\d z),\ t\geq0, \ n\in\mathbb{N}.
             \end{align*}
Similar to the proof of Theorem \ref{theo-Ito-1}, we have
    \begin{align}
         \psi(X_t^n)-\psi(X_0)
         &=\sum_{m=1}^{\infty}\Big{[} \psi(X^n_{t\wedge\tau_{m}})-\psi(X^n_{t\wedge\tau_{m}-}) \Big{]}+\sum_{m=1}^{\infty}\Big{[} \psi(X^n_{t\wedge\tau_{m}-})-\psi(X^n_{t\wedge\tau_{m-1}}) \Big{]}\nonumber\\
         &=: I_1+I_2.\label{app-burk-split}
    \end{align}
   Here as usual we set $\tau_0=0$.
    A similar argument as in the proof of Theorem \ref{theo-Ito-1} gives
    \begin{align}
           I_1&=\int_0^t\int_Z\Big{[} \psi(X^n_{s-}+\xi^n(s,z,\omega)) -\psi(X^n_{s-})    \Big{]}\tilde{N}(\d s,\d z)\nonumber\\
                   & +\int_0^t\int_Z\Big{[}\psi(X^n_{s-}+\xi^n(s,z,\omega)) -\psi(X^n_{s-})    \Big{]}\nu(\d z)\d s.\label{app-b-coro-eq1}
                   \end{align}
                   To approximate $I_2$, let $\Pi_j=\{0=t_0^j<t_1^j<\cdots<t_{l(j)}^j=T\}$, $j\in\mathbb{N}$ be any sequence of partitions of the interval $[0,T]$ whose meshes $\|\Pi_j\|:=\max_{0\leq i\leq l(j)-1 }|t_{i+1}^j-t_i^j|$ tend to $0$ as $j\rightarrow\infty$.   Set for $u\in[0,T]$,
		                 \begin{align}
			Y^{n,m}_t
			  =X^n_{\tau_{m-1}}+\int_0^t1_{[\tau_{m-1},T]}\Big[a(s)-\int_{Z}\xi^n(s,z)\,\nu(\d z)\Big]\d s+\int_0^t
			 1_{[\tau_{m-1},T]} g_s\,\d W_s.   	\label{def-Y_n}
		                 \end{align}
		                 Since the process $X^n$ is continuous in the random time interval $(t\wedge\tau_{m-1},t\wedge\tau_m)$, we have $X_t^n=Y_t^{n,m}$ for $t\in[\tau_{m-1},\tau_m)$, $m=1,2,\dots$.
Since the function $\psi(x)=|x|_E^p$, $p\geq 2$, is Fr\'{e}chet differentiable and its first derivative is Lipschitz continuous, we can apply Theorem 3.1 in \cite{[Neerven+Zhu]} to the process $Y^{n,m}$ and get
\begin{align*}
\psi(Y^{n,m}_{t\wedge\tau_m})-\psi(Y^{n,m}_{\tau_{m-1}})=&\int_{\tau_{m-1}}^{t\wedge\tau_m}\psi^\prime(Y^{n,m}_s)(a_s)\,\d s-\int_{\tau_{m-1}}^{t\wedge\tau_m}\int_Z\psi^\prime(Y^{n,m}_s)(\xi^n(s,z))\,\nu(\d z)\d s\\
&+\int_{\tau_{m-1}}^{t\wedge\tau_m}\psi^\prime(Y^{n,m}_s)(g_s)\,\d W_s
						+\lim_{j\rightarrow\infty}\sum_{i=0}^{l(j)-1} R(Y^{n,m}_{t\wedge\tau_m\wedge t_{i}^j},Y^{n,m}_{t\wedge\tau_m\wedge t_{i+1}^j}).
\end{align*}
Here $R(x,y)=\int_0^1(\psi'(x+\theta(y-x))(y-x)-\psi'(x)(y-x))\,\d \theta$.
Recall that  $Y^{n,m}_{t}=X^n_{\tau_{m-1}}$, for $t\in[0,\tau_{m-1}]$; $Y_{t}^{n,m}=X_t^n$ for $t\in[\tau_{m-1},\tau_m)$ and $Y_{t\wedge\tau_m}^{n,m}=Y_{\tau_{m}}^{n,m}=X^n_{\tau_m-}=X^n_{t\wedge\tau_m-}$,  for $t\geq\tau_m$ and also $\psi(Y^{n,m}_{t\wedge\tau_{m}})-\psi(Y^{n,m}_{\tau_{m-1}})=\psi(X^n_{t\wedge\tau_{m}-})-\psi(X^n_{t\wedge\tau_{m-1}})$.  Now we have
			\begin{align}
I_2=\sum_m\psi(X^{n}_{t\wedge\tau_m-})-\psi(X^{n}_{t\wedge\tau_{m-1}})=&\int_{0}^{t}\psi^\prime(X^{n}_s)(a_s)\,\d s-\int_{0}^{t}\int_Z\psi^\prime(X^{n}_s)(\xi^n(s,z))\,\nu(\d z)\d s\label{app-b-coro-eq2}\\
&+\int_{0}^{t}\psi^\prime(X^{n}_s)(g_s)\,\d W_s
					+\sum_m\lim_{j\rightarrow\infty}\sum_{i=0}^{l(j)-1} R(Y^{n,m}_{t\wedge\tau_m\wedge t_{i}^j},Y^{n,m}_{t\wedge\tau_m\wedge t_{i+1}^j}).\nonumber
\end{align}
	Using \eqref{app-b-coro-eq1} and \eqref{app-b-coro-eq2} in \eqref{app-burk-split} gives
	  \begin{align}
         \psi(X_t^n)-\psi(X_0)=&\int_{0}^{t}\psi^\prime(X^{n}_s)(a_s)\,\d s+\int_{0}^{t}\psi^\prime(X^{n}_s)(g_s)\,\d W_s
         +\int_0^t\int_Z \psi^\prime(X^n_{s-})(\xi^n(s,z,\omega))\, \tilde{N}(\d s,\d z)\nonumber\\
                   & +\int_0^t\int_Z\Big{[} \psi(X^n_{s-}+\xi^n(s,z,\omega)) -\psi(X^n_{s-}) -\psi^\prime(X^n_{s-}) \xi^n(s,z)   \Big{]}\,N(\d s,\d z)\label{app-b-coro-eq-m-1}\\
						&+R(Y),\nonumber
	\end{align}
	where $R(Y)=\sum_m\lim_{j\rightarrow\infty}\sum_{i=0}^{l(j)-1} R\left(Y^{n,m}_{t\wedge\tau_m\wedge t_{i}^j}, Y^{n,m}_{t\wedge\tau_m\wedge t_{i+1}^j}\right).$
	Now  taking first the supreme then the expectation to both sides of the equality \eqref{app-b-coro-eq-m-1} gives that
	                      \begin{align}\label{app-eq-45}
		                   \E\sup_{t\in[0,T]}\psi(X_t)\leq&\E\sup_{t\in[0,T]}\Big|\int_0^t\psi^\prime(X_s^n)(a_s)\,\d s\Big|+\E\sup_{t\in[0,T]}\Big|\int_0^t\psi^\prime(X_s)(g_s)\,\d W_s\Big|\nonumber\\
		&+\E\sup_{t\in[0,T]}\Big|\int_0^t\int_Z \psi^\prime(X^n_{s-})(\xi^n(s,z,\omega)) \tilde{N}(\d s,\d z)\Big|\nonumber\\
		&
		                    +\E\sup_{t\in[0,T]}\Big|\int_0^t\int_Z\Big{[} \psi(X^n_{s-}+\xi^n(s,z,\omega)) -\psi(X^n_{s-}) -\psi^\prime(X^n_{s-}) \xi^n(s,z)   \Big{]}N(\d s,\d z)\Big|\nonumber\\
		                    &+\E\sup_{0\leq t\leq T} R(Y)\nonumber\\
		                    &:=J_1+J_2+J_3+J_4+J_5.
		                  \end{align}
		                  It's worth mentioning here that if we impose the additional conditions  \eqref{app-Ito-2-assu}, i.e. $\E N(t,Z)<\infty$ and $\psi^\prime(X_s^n)(a_s)\leq 0$, then $X^n=X$ and the first term on the right hand side of  \eqref{app-b-coro-eq-m-1} satisfies $\int_{0}^{t}\psi^\prime(X^{n}_s)(a_s)\leq 0$. In this case we do not need to consider $J_1$.	
		
		              Now let $\varepsilon>0$.  By Young's inequality  we have
		               		                  \begin{align*}
		                  J_1\leq \varepsilon \E\sup_{0\leq t\leq T}|X_s^n|_E^p+C_{p,\frac{1}{\varepsilon^p}}\E\Big(\int_0^T|a_s|_E\,\d s\Big)^p.
		                  \end{align*}

		            For the term $J_2$, according to \cite[Lemma 3.1]{[Neerven+Zhu]}, the process $t\mapsto \psi^\prime(X^n_t)( g_t) $ is progressively measurable and satisfies $\E\Big(\int_0^T\| \psi^\prime(X^n_t)( g_t)\|^2_H\d t\Big)^{\frac12}<\infty $. Then we can apply Burkholder's inequality \eqref{burk-ine}, H\"{o}lder's inequality, and Young's inequality to obtain, for any $\varepsilon>0$,
			    \begin{align}
			          J_2&\leq C\, \Big(\int_0^T\|\psi^\prime(X^n_s)( g_s)\|^2_H \,\d s\Big)^{\frac12}\nonumber\\
			          &\leq C_p\,\E\Big(\int_0^T |X^n_s|_E^{2(p-1)}\|g_s\|^2_{\gamma(H;E)}\,\d s\Big)^{\frac{1}{2}}\nonumber\\
			          &\leq C_p\,\E\sup_{t\in[0,T]}|X_t^n|_E^{p-1}\Big(\int_0^T \|g_s\|^2_{\gamma(H;E)}\,\d s\Big)^{\frac{1}{2}}\\
		&\leq C_p\varepsilon\,\E\sup_{t\in[0,T]}|X^n_t|_E^{p}+C_{p,\frac{1}{\varepsilon}}\E\Big(\int_0^T\|g_s\|_{\gamma(H;E)}^2\,\d s\Big)^{\frac{p}{2}}.\nonumber
			    \end{align}
			    %----------
			    	 %-----------
	By adopting the same procedure as in the proof of Theorem \ref{Th-burk} (with  $r=2$), we get
	\begin{align*}
	J_3&\leq C_{p}\,\varepsilon\, \E\sup_{t\in[0,T]}|X_t^n|^p_E+C_{p,\frac{1}{\varepsilon}}
	\E \Big(\int_0^{T}\int_Z|\xi(s,z)|_E^2 \,N(\d s,\d z)\Big)^{\frac{p}{2}},\\
	J_4&\leq C_{p}\,\varepsilon\, \E\sup_{t\in[0,T]}|X_t^n|^p_E+C_{p,\frac{1}{\varepsilon}}
	\E \Big(\int_0^{T}\int_Z|\xi(s,z)|_E^2 \,N(\d s,\d z)\Big)^{\frac{p}{2}}.
	\end{align*}
	The estimate of $J_5$ is similar to that in \cite[Theorem 1.2]{[Neerven+Zhu]}. By applying Lemma 3.5 in \cite{[Neerven+Zhu]}, we have
	 \begin{align*}
		J_5&\leq \E\sup_{0\leq t\leq T}\sum_m\liminf_{j\rightarrow\infty}\sum_{i=0}^{l(j)-1} |R(Y^{n,m}_{t\wedge\tau_m\wedge t_{i}^j},Y^{n,m}_{t\wedge\tau_m\wedge t_{i+1}^j})|\\
		&\leq \sum_m\E\liminf_{j\rightarrow\infty}\sum_{i=0}^{l(j)-1} |R(Y^{n,m}_{\tau_m\wedge t_{i}^j},Y^{n,m}_{\tau_m\wedge t_{i+1}^j})|.
	\end{align*}
	Since $R(x,y)=\int_0^1(\psi'(x+\theta (y-x))(y-x)-\psi'(x)(y-x))\,\d \theta$, by \eqref{Ap-rem-eq-2} we infer
	\begin{align*}
	|R(x,y)|\leq C_p\int_0^1(|x+\theta(y-x)|_E+|x|_E)^{p-2}|x-y|_E\,\d\theta\leq C_p|x|_E^{p-2}|x-y|_E^2+C_p|x-y|^p_E.
	\end{align*}
	It follows that
	\begin{align}
	|R(Y^{n,m}_{\tau_m\wedge t_{i}^j},Y^{n,m}_{\tau_m\wedge t_{i+1}^j})|\leq C_p|Y^{n,m}_{\tau_m\wedge t_{i}^j}|_E^{p-2}|Y^{n,m}_{\tau_m\wedge t_{i+1}^j}-Y^{n,m}_{\tau_m\wedge t_{i}^j}|_E^2+C_p|Y^{n,m}_{\tau_m\wedge t_{i+1}^j}-Y^{n,m}_{\tau_m\wedge t_{i}^j}|^p_E.\label{esti-R}
	\end{align}
	%------------------------------
	For the first term, by \eqref{def-Y_n}, we obtain
\begin{align*}
\ & \sum_{i=0}^{l(j)-1}| Y^{n,m}_{\tau_m\wedge t_{i}^j}|_E ^{p-2} | Y^{n,m}_{\tau_m\wedge t_{i+1}^j}-Y^{n,m}_{\tau_m\wedge t_{i}^j}|_E ^2 \\
& \leq
 2\sum_{i=0}^{l(j)-1} | Y^{n,m}_{\tau_m\wedge t_i^j}|_E ^{p-2} \Big|  \int_{  t_i^j}^{ t_{i+1}^j}1_{[\tau_{m-1},\tau_m]}\big[a_s-\int_Z\xi^n(s,z)\,\nu(\d z)\big]\d s   \Big|_E ^2
\\
& \qquad+2\sum_{i=0}^{l(j)-1}| Y^{n,m}_{\tau_m\wedge t_i^j}|_E ^{p-2} \Big|  \int_{ t_i^j}^{ t_{i+1}^j}1_{[\tau_{m-1},\tau_m]}g_s
\,\d W_s\Big|_E ^2 =:K_1^{j}+K_2^j.
\end{align*}
   For the term $K_1^j$, it's easy to see
     \begin{align*}
                \lim_{j\rightarrow\infty}K_1^j=0,\
\mathbb{P}\text{-a.s.}
                \end{align*}
To estimate $K_2^j$, by using the It\^{o} isometry property \eqref{Iso-ine} and Young's inequality with $\e>0$, we infer
\begin{align*}
 \sum_m\E\, \liminf_j K_2^j &\leq \sum_m\liminf_j \, \E K_2^j \\
&=\sum_m\liminf_j\, \E\sum_{i=0}^{l(j)-1}| Y^{n,m}_{\tau_m\wedge t_i^j}|_E ^{p-2} \Big|  \int_{ t_i^j}^{ t_{i+1}^j}1_{[\tau_{m-1},\tau_m]}g_s
\,\d W_s\Big|_E ^2\\
 &=\sum_m \liminf_j\sum_{i=0}^{l(j)-1}\E\Big(| Y^{n,m}_{\tau_m\wedge t_i^j}|_E ^{p-2}\E\Big(\Big| \int_{t_i^j\wedge t}^{t_{i+1}^j}1_{[\tau_{m-1},\tau_m]}g_s
\,\d W_s\Big|_E^2\big|\F_{t_i^j}\Big)\Big)\\
 &\leq \sum_m
C\liminf_j\sum_{i=0}^{l(j)-1}\E\Big(| Y^{n,m}_{\tau_m\wedge t_i^j}|_E ^{p-2}\E\Big(\int_{t_i^j}^{t_{i+1}^j}1_{[\tau_{m-1},\tau_m]}\n g_s\n ^2_{\gamma(H,E)}
\,\d s\big|\F_{t_i^j}\Big)\Big)
\end{align*}
\begin{align*}
 &\leq \sum_m
C\liminf_j\sum_{i=0}^{l(j)-1}\E\Big(| Y^{n,m}_{\tau_m\wedge t_i^j}|_E ^{p-2}\Big(\int_{t_i^j}^{t_{i+1}^j}1_{[\tau_{m-1},\tau_m]}\n g_s\n ^2_{\gamma(H,E)}
\,\d s\Big)\Big)\\
   &\leq
C\sum_m \liminf_j\E\Big(\sup_{s\in[0,T]}| Y^{n,m}_{s\wedge\tau_m}|_E ^{p-2}\sum_{i=0}^{{l(j)-1}}\int_{
t_i^j}^{t_{i+1}^j}\n 1_{[\tau_{m-1},\tau_m]}g_s\n ^2_{\gamma(H,E)} \,\d s\Big)\\
&\leq
C\sum_m\E\Big(\sup_{s\in[0,T]}| X^{n}_{s\wedge\tau_m}|_E ^{p-2}\int_{
0}^{T}\n 1_{[\tau_{m-1},\tau_m]}g_s\n ^2_{\gamma(H,E)} \,\d s\Big)\\
&\leq
C\E\Big(\sup_{s\in[0,T]}| X^{n}_{s}|_E ^{p-2}\int_{
0}^{ T}\n g_s\n ^2_{\gamma(H,E)} \,\d s\Big)\\
& \le
C_p\varepsilon\,\E\Big(\sup_{s\in[0,T]}| X^n_s|_E ^{p}\Big)
+ C_{p,\frac1{\varepsilon}}\,\E  \Big(\int_{0}^{T}\n g_s\n ^2_{\gamma(H,E)}
\,\d s\Big)^{\frac{p}{2}},
\end{align*}
where we also used the fact that $\sup_{s\in[0,T]}|Y^{n,m}_{s\wedge\tau_{m}}|^{p-2}\leq \sup_{s\in[0,T]}|X_s^n|_E^{p-2}$.

Next we estimate the second term in \eqref{esti-R}. Observe that
\begin{align*}
\sum_{i=0}^{l(j)-1}|Y^{n,m}_{\tau_m\wedge t_{i+1}^j}-Y^{n,m}_{\tau_m\wedge t_{i}^j}|^p_E
&\leq  C\sum_{i=0}^{l(j)-1}  \Big|  \int_{  t_i^j}^{ t_{i+1}^j}1_{[\tau_{m-1},\tau_m]}\big[a_s-\int_Z\xi^n(s,z)\,\nu(\d z)\big]\d s   \Big|_E ^p
\\
& \qquad+C\sum_{i=0}^{l(j)-1} \Big|  \int_{ t_i^j}^{ t_{i+1}^j}1_{[\tau_{m-1},\tau_m]}g_s
\,\d W_s\Big|_E ^p=: K_3^j+K_4^j.
\end{align*}
Clearly,
\begin{align*}
 \lim_{j\rightarrow\infty} K_3^j&\leq
C\lim_{j\rightarrow\infty}\sup_{0\leq i\leq l(j)-1}\Big|  \int_{  t_i^j}^{ t_{i+1}^j}1_{[\tau_{m-1},\tau_m]}\big[a_s-\int_Z\xi^n(s,z)\nu(\d z)\big]\,\d s   \Big|_E^{p-1}\cdot
  \int_{0}^{ T}\Big|a_s-\int_Z\xi^n(s,z)\,\nu(\d z)\Big|_E\d s  \\
  & =0.
\end{align*}
  Moreover, by Burkholder's inequality \eqref{burk-ine},
  \begin{align*}
 \sum_m\E\, \liminf_j K_4^j \leq \sum_m\liminf_j \, \E K_4^j
&=\sum_m\liminf_j\, \E\sum_{i=0}^{l(j)-1} \Big|  \int_{ t_i^j}^{ t_{i+1}^j}1_{[\tau_{m-1},\tau_m]}g_s
\,\d W_s\Big|_E ^p\\
 &\leq \sum_m
C\liminf_j\sum_{i=0}^{l(j)-1}\E\Big(\int_{t_i^j}^{t_{i+1}^j}1_{[\tau_{m-1},\tau_m]}\n g_s\n ^2_{\gamma(H,E)}
\,\d s\Big)^{\frac{p}2}\\
   &\leq
C\,\sum_m \liminf_j\E\Big(\sum_{i=0}^{{l(j)-1}}\int_{
t_i^j}^{t_{i+1}^j}\n 1_{[\tau_{m-1},\tau_m]}g_s\n ^2_{\gamma(H,E)} \,\d s\Big)^{\frac{p}2}\\
&\leq
C\sum_m\E\Big(\int_{
0}^{T}\n 1_{[\tau_{m-1},\tau_m]}g_s\n ^2_{\gamma(H,E)} \,\d s\Big)^{\frac{p}{2}}\\
& \le
 C\,\E  \Big(\int_{0}^{T}\n g_s\n ^2_{\gamma(H,E)}
\,\d s\Big)^{\frac{p}{2}}.
\end{align*}
Collecting all the estimates, for any $\e>0$
we obtain
	 \begin{align*}
		J_5		
		\leq\varepsilon C_p\,\E\sup_{s\in[0,T]}|X^n_s|_E^p+C_{p,\frac{1}{\varepsilon}}\,\E\Big(\int_{0}^{T}\|g_s\|_{\gamma(H,E)}^2\,\d s\Big)^{\frac{p}{2}}.
	\end{align*}	
				Finally, we can collect together all terms and choose a suitable value of $\varepsilon$ to get
	%-----------------------
	\begin{align}
			 \E\sup_{t\in[0,T]}|X^n_t|_E^p\leq &C_p\E\Big(\int_0^T|a_s|_E\,\d s\Big)^p+ C_{p}\E\Big(\int_0^T\|g_s\|^2_{\gamma(H;E)}\,\d s\Big)^{\frac{p}{2}}\nonumber\\
			 	&+C_p\E \Big(\int_0^{T}\int_Z|\xi(s,z)|_E^2 \,N(\d s,\d z)\Big)^{\frac{p}{2}}. \label{theo-levy-burk-proof-1}
	\end{align}
	By applying \eqref{Lp-1} with $r=2$, we have
	\begin{align*}
			 \E\sup_{t\in[0,T]}|X_t-X^n_t|_E^p			 \leq &C_p\E\Big(\int_0^T\int_Z|\xi^n(s,z)-\xi(s,z)|_E^p\,\nu(\d z)\d s)\Big)\\
			 &+ C_p\E\Big(\int_0^T\int_Z|\xi^n(s,z)-\xi(s,z)|_E^2\,\nu(\d z)\d s)\Big)^{\frac{p}{2}}\rightarrow0,\text{ as }n\rightarrow\infty,
			 		\end{align*}
	and inequality \eqref{burk-ineq-levy-1} follows.
	Inequality \eqref{burk-ineq-levy-2}  follows immediately from \eqref{burk-ineq-levy-1} and \eqref{rem-cor-eq-main}.
\end{proof}

	%--------------------------------------------
%\vfill\newpage
	\section{\bf{Maximal inequalities and exponential estimates for stochastic convolutions w.r.t. L\'{e}vy processes}}\label{sec-max-ineq}
		\subsection{Maximal $L^p$ estimates for stochastic convolutions w.r.t. Poisson random measures }
		 Suppose that $(E,|\cdot|_E)$ is a martingale type $r$ Banach space, $1<r\leq 2$. Here, as before, $|\cdot|_E$ is the equivalent $r$-smooth norm. Let $(e^{tA})_{t\geq0}$ be a $C_0$-semigroup on $E$, with the generator denoted by $A$,  of contraction type with respect to the $r$-smooth norm $|\cdot|_E$ such that $\|e^{tA}\|\leq e^{t\alpha} $  for some $\alpha \geq 0$. We start by considering the following stochastic convolution driven by purely discontinuous noises:
	\begin{align}
	         X_t=\int_0^t\int_Z e^{(t-s)A}\xi(s,z)\,\tilde{N}(\d s,\d z).
	\end{align}
	
		\begin{thm}\label{Th-MI-mian-1} Let $(e^{tA})_{t\geq0}$ be a $C_0$-contraction semigroup on a martingale type $r$ Banach space $(E,|\cdot|_E)$, $1<r\leq 2$,
		and let $\xi\in\mathcal{M}_T^r(\mathcal{P}\otimes\mathcal{Z},\d t\times\mathbb{P}\times\nu;E)$. Then there exists a c\`{a}dl\`{a}g modification $\bar{X}$ of $X$ such that for $r\leq p<\infty$,
		\begin{align}\label{max-eq-1-1}
			        \E\sup_{0\leq t\leq T} \big|\bar{X}_t\big|_E^p\leq e^{\alpha p T}C_{p,r}\,\E\Big(\int_0^T\int_Z |\xi(s,z)|^r_E\, N(\d s,\d z)\Big)^{\frac{p}{r}}.
		\end{align}

	\end{thm}
	 %---------
	\begin{proof}
		\textbf{Step 1}  First fix $p\geq r$ and suppose that $\xi\in
		\mathcal{M}_T^r(\mathcal{P}\otimes\mathcal{Z},\d t\times\mathbb{P}\times\nu;\mathcal{D}(A))$ and $\alpha=0$, i.e. $\|e^{tA}\|_{L(E)}\leq1$.
		It is known (see \cite{Zhu2017}) that process
		$X_t$, $t\in[0,T]$ is a unique strong solution to the following problem
		\begin{align}\label{SPDE}
		\begin{split}
		   dX_t&=AX_tdt+\int_Z\xi(t,z)\,\tilde{N}(\d t,\d z),\ \ t\in[0,T],\\
		   X_0&=0.
		\end{split}
		\end{align}
		Apparently $X_t$ is an $E$-valued c\`{a}dl\`{a}g process. Let $\tau$ be a stopping time with values in $[0.T].$ We have
		\begin{align*}
		X_{t\wedge\tau}=\int_0^t1_{[0,\tau]}(s)AX_s\d s+\int_0^t\int_Z1_{[0,\tau]}(s)\xi(s,z)\,\tilde{N}(\d s,\d z).
		\end{align*}
	   Hence, by applying the It\^o  formula \eqref{theo-Ito-3} to $\psi(\cdot)=|\cdot|_E^p$ and then using the fact that $\psi^\prime(x)(Ax)\leq 0$ for all $x\in D(A)$ (see e.g. Lemma 4.7 in \cite{Zhu2017}), we obtain for $t\in[0,T]$, $\mathbb{P}$-a.s.
		\begin{align}
		\begin{split}\label{EQ-1}
		   \psi(X_{t\wedge\tau})&\leq \int_0^t\int_Z1_{[0,\tau]}(s)\psi(X_{s-}+\xi(s,z))-\psi(X_{s-})\,\tilde{N}(\d s,\d z)\\
		  &\hspace{1cm}+\int_0^t\int_Z1_{[0,\tau]}(s)\Big{[}\psi(X_{s-}+\xi(s,z))-\psi(X_{s-})-\psi^\prime(X_{s-})(\xi(s,z))\Big{]}\,\nu(\d z)\d s\\
		   &=     \int_0^t\int_Z1_{[0,\tau]}(s)\psi^\prime(X_{s-})(\xi(s,z))\,\tilde{N}(\d s,\d z)\\
		  &\hspace{1cm}+\int_0^t\int_Z1_{[0,\tau]}(s)\Big{[}\psi(X_{s-}+\xi(s,z))-\psi(X_{s-})-\psi^\prime(X_{s-})(\xi(s,z))\Big{]}\,N(\d s,\d z)\\
		&=:I_1(t)+I_2(t).
		\end{split}
		\end{align}
	 We again follow the same line of argument as used in deducing \eqref{burk-ineq-I_1} and \eqref{sec-1-eq-21} to get
	 	\begin{align*}
	\E\sup_{0\leq t\leq T}I_1(t)&\leq C_{p}\,\varepsilon \E\sup_{t\in[0,T]}|X_{t\wedge\tau}|^p_E+C_{p,\frac{1}{\varepsilon}}
	\E \Big(\int_0^{T}\int_Z1_{[0,\tau]}|\xi(s,z)|_E^r \,N(\d s,\d z)\Big)^{\frac{p}{r}},\\
	\E\sup_{0\leq t\leq T}I_2(t)&\leq C_{p}\,\varepsilon \E\sup_{t\in[0,T]}|X_{t\wedge\tau}|^p_E+C_{p,\frac{1}{\varepsilon}}
	\E \Big(\int_0^{T}\int_Z1_{[0,\tau]}|\xi(s,z)|_E^r \,N(\d s,\d z)\Big)^{\frac{p}{r}}.
	\end{align*}
		By choosing a suitable number $\varepsilon$ such that $\varepsilon C_{p,r}=\frac{1}{4}$, we get
	       	\begin{align*}
		       \E\sup_{0\leq t\leq T}|X_{t\wedge\tau}|^p_E\leq
		C_{p,r}\E\left(\int_0^t\int_Z1_{[0,\tau]}(s)|\xi(s,z)|_E^{r}\,N(\d s,\d z)\right)^{\frac{p}{r}}.
		\end{align*}
			 \textbf{Step 2}
	Suppose that $\alpha=0$. Take 	$\xi\in
		\mathcal{M}_T^p(\mathcal{P}\otimes\mathcal{Z},\d t\times\mathbb{P}\times\nu;E)$. Put $\xi^n(t,\omega,z)=nR(n,A)\xi(t,\omega,z)$ on $[0,T]\times\Omega\times Z$, where $R(n,A)=(nI-A)^{-1}$. Since $A$ is the infinitesimal generator of the contraction $C_0$-semigroup $e^{tA}$, $t\geq0$, by the Hille-Yosida theorem, we have $\xi^n\rightarrow\xi$ pointwisely on $[0,T]\times\Omega\times Z$ and $|\xi^n|_E\leq |\xi|_E$. Define,  for each $n\in\mathbb{N}$, a process $X^n$ by
		\begin{align*}
		    X^n_t=\int_0^te^{(t-s)A}\xi^n(s,z)\,\tilde{N}(\d s,\d z),\; t\in [0,T].
		\end{align*}
		Then for each $n\in\mathbb{N}$, $X^n$ is an $E$-valued c\`{a}dl\`{a}g process.
		By the discussion in step 1, we infer
		 \begin{align*}
		   \lim_{n,m\rightarrow\infty} \E\sup_{0\leq t\leq T}|X^n_t-X^m_t|_E^p=0,
		\end{align*}
		from which we may deduce by using the usual argument that $X^n$ is almost surely uniformly convergent on $[0,T]$ to some  c\`{a}dl\`{a}g process $\bar{X}$ satisfying $\lim\limits_{n\rightarrow\infty}\E\sup_{t\in[0,T]}|\bar{X}_t-X^n_t|_E^p=0$ and for any stopping time $\tau$ in $[0,T]$
	       	\begin{align}\label{ineq-proof-1}
		       \E\sup_{0\leq t\leq T}|\bar{X}_{t\wedge\tau}|_E^p\leq
		C_{p,r}\E\left(\int_0^T\int_Z1_{[0,\tau]}(s)|\xi(s,z)|^{r}_E\,N(\d s,\d z)\right)^{\frac{p}{r}}.
		\end{align}
	    Meanwhile, since
	            	\begin{align*}
		       \E|\bar{X}_t-X_t|_E^r
			   &\leq 2^r\E|\bar{X}_t-X^n_t|_E^p+2^r\E\left|\int_0^t\int_Z\Big{(}e^{(t-s)A}\xi^n(s,z)-e^{(t-s)A}\xi(s,z)\Big{)}\,\tilde{N}(\d s,\d z)\right|^r\\
	    	   &\leq 2^r\E|\bar{X}_t-X^n_t|^r_E+C_{r}\E\int_0^t\int_Z|\xi^n(s,z)-\xi(s,z)|^r_E\,\nu(\d z)\,\d s,
	  \end{align*}
	we infer that \dela{that $X^n_t$ converges to $X_t$ in $L_p(\Omega)$ for every $t\in[0,T]$. Hence }$\bar{X}$ is a modification of $X$. The required result then follows.   \\
		\textbf{Step 3} Let $\alpha>0$. Define
	\begin{align*}
		\tilde{X}_t=e^{-\alpha t}X_t=\int_0^t\int_Z e^{(A-\alpha I)(t-s)}\big(e^{-s\alpha}\xi(s,z)\big)\,\tilde{N}(\d s,\d z).
	\end{align*}
	Note that $e^{(A-\alpha I)t}$ is a $C_0$-semigroup of contractions. Applying \eqref{ineq-proof-1}, we have
	\begin{align*}
	   \E\sup_{0\leq t\leq T}|X_t|^p_E\leq e^{\alpha p T} \E\sup_{0\leq t\leq T}|\tilde{X}_t|^p_E\leq e^{\alpha p T}C_{p,r}\E\Big(\int_0^T\int_Z |\xi(s,z)|^r_E\, N(\d s,\d z)\Big)^{\frac{p}{r}}.
	\end{align*}
	\end{proof}
	%-----------------------
		\begin{cor}   Let $(e^{tA})_{t\geq0}$ be a $C_0$-contraction semigroup on a martingale type $r$ Banach space $(E,|\cdot|_E)$, $1<r\leq 2$
		and let $\xi\in\mathcal{M}_T^r(\mathcal{P}\otimes\mathcal{Z},\d t\times\mathbb{P}\times\nu;E)$. Then there exists a c\`adl\`ag modification $\bar{X}$ such that
		\begin{align}
			       \E\sup_{t\in[0,T]} \big|\bar{X}_t\big|_E^p\leq  &e^{\alpha pT} C_{p,r}\E\Big(\int_0^T\int_Z |\xi(s,z)|^r_E \;\nu(\d z)\,\d s\Big)^{\frac{p}{r}}\quad\text{for}0< p\leq r.\label{Lp-estimates-1}
		\end{align}
	\end{cor}
	\begin{proof}
	By applying \eqref{max-eq-1-1} with $p=r$, we have
	\begin{align*}
	      \E\sup_{0\leq t\leq T} \Big|\int_0^t\int_Z e^{(t-s)A}\xi(s,z)\;\tilde{N}(\d s,\d z)\Big|_E^r&\leq e^{\alpha p T}C_{r}\E\Big(\int_0^T\int_Z1_{[0,\tau]} |\xi(s,z)|^r_E\; N(\d s,\d z)\Big)\\
	      &=e^{\alpha p T}C_{r}\E\Big(\int_0^T1_{[0,\tau]}\int_Z |\xi(s,z)|^r_E\; \nu(\d z)\d s\Big).
	\end{align*}
	 Let us set
	 \begin{align}
	 Z_t=\Big|\int_0^t\int_Z e^{(t-s)A}\xi(s,z)\;\tilde{N}(\d s,\d z)\Big|_E^r\text{ and }A_t=\int_0^t\int_Z|\xi(s,z)|_E^r\;\nu(\d z)\d s,\;\; t\in[0,T] .
	 \end{align}
	 We find that for any stopping time $\tau$ in $[0,T]$,
	$$ \E Z_{\tau}\leq C_r\E A_{\tau}.$$
	Moreover, $Z$ is a c\`{a}dl\`{a}g process or more precisely it has a c\`{a}dl\`{a}g version and the process $A$ is increasing. Put $k=\frac{p}{r}$. For $p\in(0,r)$, notice that $k\in(0,1)$. By means of Proposition IV4.7 from \cite{Yor}, we deduce
	\begin{align*}
	\E \sup_{0\leq t\leq T}Z_t^{k}\leq \frac{2-k}{1-k}\E A_T^{k}.
	\end{align*}
		Hence we infer for $0< p<r$,
	\begin{align}\label{max-ine-proof-less-1}
	\E \sup_{0\leq t\leq T}\Big|\int_0^t\int_Z e^{(t-s)A}\xi(s,z)\;\tilde{N}(\d s,\d z)\Big|_E^p \leq e^{\alpha pT}C_{p,r}\E\Big(\int_0^T\int_Z|\xi(s,z)|_E^r\;\nu(\d z)\d s\Big)^{\frac{p}{r}} .
	\end{align}
This proves inequality \eqref{Lp-estimates-1}.
		\end{proof}
	In the same way as we deduce inequality \eqref{Lp-1} from Theorem \ref{Th-burk}, we can get the following $L^p$ inequality for stochastic convolutions from the preceding theorem immediately.
	%--------------------
	\begin{cor}\label{cor-max-eq-1}\dela{  Let $E$ be a martingale type $r$ Banach space, $1<r\leq 2$.}Let $(e^{tA})_{t\geq0}$ be a $C_0$-contraction semigroup on a martingle type $r$ Banach space $(E,|\cdot|_E)$, $1<r\leq 2$
		and let $\xi\in\mathcal{M}_T^r(\mathcal{P}\otimes\mathcal{Z},\d t\times\mathbb{P}\times\nu;E)$. Then there exists a c\`adl\`ag modification $\bar{X}$ such that for $r\leq p<\infty$,
		\begin{align}
		\begin{split}\label{max-eq-1}
			       \E\sup_{t\in[0,T]} \big|\bar{X}_t\big|_E^p\leq & e^{\alpha pT}C_{p,r}\Big[\E\Big(\int_0^T\int_Z |\xi(s,z)|^p_E \;\nu(\d z)\,\d s\Big)+\E\Big(\int_0^T\int_Z |\xi(s,z)|^r_E \;\nu(\d z)\,\d s\Big)^{\frac{p}{r}}\Big].
			       		\end{split}
		\end{align}
	\end{cor}

%----------------------------------------
	\begin{rem}
		 Note that if $(e^{tA})_{t\geq0}$ is an analytic semigroup with generators $A$ such that $-A$ has a bounded $H^{\infty}$-calculus,  we can always find an equivalent norm on $E$ for which $E$ is $2$-smooth and $(e^{tA})_{t\geq0}$ is a contraction $C_0$-semigroup; see \cite{[VerWei]}. For a related earlier result see \cite[Corollary 6.2]{Brz1}.
	Both our Theorems 3.1 and 3.3 from \cite{[DMN]} are applicable to $C_0$-semigroups of positive contractions on $L^q(S)$ spaces, where $S$ is a measure space and $q \in (1,\infty)$; see Example 3.1 in \cite{[DMN]}. However, more precise understanding of the relationship  between our results and those from \cite{[DMN]} would require further analysis.   Our results have roots in the PhD thesis of the first named author which was known to the authors of  \cite{[DMN]}.

		  \dela{ To do this properly, by \cite{[VerWei]}, we may define an equivalent norm $\|x\|_E=\|T_x\|_{\gamma(\mathbb{R}_+,E)}$, where the operator $T_x$ is defined by $T_x(t)=(-A)^{\frac{1}{2}}e^{tA}x$, $t\in\mathbb{R}_+$. Since there exists an isometry from $\gamma(\mathbb{R}_+,E)$ to $L^2(\Omega,E)$, it is enough to show that $L^2(\Omega,E)$ is $2$-smooth with respect to the norm $\|\cdot\|_{L^2(\Omega,E)}$. Observe that for all $u\in L^2(\Omega,E)$
			\begin{align*}
				    (\|y\|^2_{L^2(\Omega,E)})'(u)=\int_{\Omega}|(y(\omega)|_E^2)'(u(\omega))\d \mathbb{P}(\omega)
			\end{align*}
			and
			 \begin{align*}
				    \Big|(\|y\|^2_{L^2(\Omega,E)})'(u)-(\|x\|^2_{L^2(\Omega,E)})'(u)\Big|_{\mathbb{R}}&=\Big|\int_{\Omega}(|(y(\omega)|_E^2)'(u(\omega))-(|(x(\omega)|_E^2)'(u(\omega))\d\mathbb{P}(\omega)\Big|_{\mathbb{R}} \\
				&\leq \int_{\Omega}\|(|(y(\omega)|_E^2)'-(|(x(\omega)|_E^2)'\|_{L(E,\mathbb{R})}|u(\omega)|\d\mathbb{P}(\omega)\\
				&\leq L\int_{\Omega} |y(\omega)-x(\omega)|_E|u(\omega)|_E \d\mathbb{P}(\omega)\\
				&\leq L\|y-x\|_{L^2(\Omega,E)}\|u\|_{L^2(\Omega,E)},
			\end{align*}
			where the constant $L$ comes from the Lipschitz continuity of the first derivative of $|\cdot|_E^2$, since $E$ is $2$-smooth.
		 Therefore, the first derivative of $\|\cdot\|^2_{L^2(\Omega,E)}$ is Lipschitz continuous. This implies the space $L^2(\Omega,E)$ with the norm $\|\cdot\|_{L^2(\Omega,E)}$ is $2$-smooth. Hence the space $(E,\|\cdot\|_E)$ is a $2$-smoothness Banach space. Also, we see that
		\begin{align*}
			  \|e^{tA}x\|_E=\|T_{e^{tA}x}\|_{\gamma(\mathbb{R}_+,E)}=\|T_x\|_{\gamma([s,\infty),E)}\leq\|T_x\|_{\gamma(\mathbb{R}_+,E)}=\|x\|_E,
		\end{align*}
		which shows that $(e^{tA})_{t\geq0}$ is a contraction semigroup on $(E,\|\cdot\|_E)$.}
	\end{rem}

%----------------------------------------

	%------------------------------------Exponential Tail estimates-------------------
	\subsection{Exponential tail estimates}
	With more effort the method we described above to derive \eqref{max-eq-1-1} and \eqref{max-eq-1} can be applied to obtain the following exponential tail estimates for stochastic convolutions driven by compensated Poisson random measures.
	\begin{thm}\dela{Assume that $E$ is a martingale type $2$ Banach space.} Let $(e^{tA})_{t\geq0}$ be a $C_0$-contraction semigroup on a martingale type $2$ Banach space $(E,|\cdot|_E)$
		and let $\xi\in\mathcal{M}_T^2(\mathcal{P}\otimes\mathcal{Z},\d t\times\mathbb{P}\times\nu;E)$.
	 If there exists $\lambda>0$ and $M_{\lambda}>0$ such that
	    \begin{align}
	    	    \int_0^T\int_Z e^{\lambda^{\frac12}|\xi(s,z)|_E}\lambda |\xi(s,z)|_E^2\;\nu(\d z)\;\d s\dela{+\int_0^T\int_Z|\xi(s,z)|^2_E\,\nu(\d z)\d s+\int_0^T\|g_s\|^2_{\gamma(H,E)}\d s}\leq M_{\lambda},\quad\mathbb{P}\text{-a.s.,}
		    	     \end{align}
	then for every $R>0$,
	\begin{align}\label{exp-tail-eq}
		    \mathbb{P}\Big( \sup_{0\leq t\leq T}|X_t|_E \geq R \Big)\leq C_{\lambda}e^{-(1+\lambda R^2)^{\frac{1}{2}}}
	\end{align}
	with $C_\lambda=e^{1+3CM_{\lambda}}$.
	\end{thm}

	\begin{proof}
		Define $f_{\lambda}=(1+\lambda|x|_E^2)^{\frac{1}{2}}$, $x\in E$. It is easy to see that $f^\prime_{\lambda}(x)=\lambda(1+\lambda |x|_E^2)^{-\frac{1}{2}}|x|_Ef_x$, where $f_x$ is the derivative of $|\cdot|_E$ at the point $x\neq 0$. One can show that for every $x,y\in E$,
		\begin{align}\label{sec-3-Lip}
			    |f^\prime_{\lambda}(x)-f^\prime_{\lambda}(y)|_{L(E;\mathbb{R})}\leq C\lambda |x-y|_E,
		\end{align}
		where the constant $C$ depends on the constant appearing in the martingale type $2$ Banach space property.
		As in the proof of Theorem 3.1, we may assume that $\xi\in M_T^2(\mathcal{P}\otimes\mathcal{Z},\d t\times\mathbb{P}\times\nu;\mathcal{D}(A))$ and consider the process
		\begin{align}
			X_t=\int_0^t AX_s\d s+\int_0^t\int_Z \xi(s,z)\;\tilde{N}(\d s,\d z).
		\end{align}
		Since $f^\prime_{\lambda}(x)(Ax)\leq 0$, for all $x\in E$, applying the It\^o  formula \eqref{Ito-formula-3} yields
		\begin{align*}
			 f_{\lambda}(X_t)=&f_{\lambda}(X_0)+\int_0^tf_{\lambda}'(X_s)(Au_{s})\;\d s+\int_0^t\int_Z\Big( f_{\lambda}(X_{s-}+\xi(s,z))-f_{\lambda}(X_{s-})\Big)\;\tilde{N}(\d s,\d z)\\
			                   &+\int_0^t\int_Z \Big(f_{\lambda}(X_{s-}+\xi(s,z))-f_{\lambda}(X_{s-})-f^\prime_{\lambda}(X_{s-})(\xi(s,z))\Big)\;\nu(\d z)\,\d s\\
			\leq& 1+ \int_0^t\int_Z\Big( f_{\lambda}(X_{s-}+\xi(s,z))-f_{\lambda}(X_{s-})\Big)\;\tilde{N}(\d s,\d z) \\
			 &+\int_0^t\int_Z \Big(f_{\lambda}(X_{s-}+\xi(s,z))-f_{\lambda}(X_{s-})-f^\prime_{\lambda}(X_{s-})(\xi(s,z))\Big)\;\nu(\d z)\,\d s.
		\end{align*}
	   Let us set
		   	\begin{align*}
					Y_t&=\int_0^t\int_Z\Big( f_{\lambda}(X_{s-}+\xi(s,z))-f_{\lambda}(X_{s-})\Big)\,\tilde{N}(\d s,\d z)\\
		    		&-\int_0^t\int_Z\Big( e^{f_{\lambda}(X_{s-}+\xi(s,z))-f_{\lambda}(X_{s-})}-1-\big(f_{\lambda}(X_{s-}+\xi(s,z))-f_{\lambda}(X_{s-})\big)\Big)\,\nu(\d z)\,\d s.
		    	\end{align*}
	 Again by applying the It\^o  formula \eqref{theo-Ito-3} to $Y_t$, we get
		\begin{align*}
			    e^{Y_t}=&1
-   \int_0^t\int_Z e^{Y_{s-}}\Big(e^{f_{\lambda}(X_{s-}+\xi(s,z))-f_{\lambda}(X_{s-})}-1-\big(f_{\lambda}(X_{s-}+\xi(s,z))-f_{\lambda}(X_{s-})\big)\Big)\,\nu(\d z)\,\d s\\
			&+\int_0^t\int_Z e^{Y_{s-}+f_{\lambda}(X_{s-}+\xi(s,z))-f_{\lambda}(X_{s-})}-e^{Y_{s-}}\,\tilde{N}(\d s,\d z)\\
			&+ \int_0^t\int_Z \Big(e^{Y_{s-}+f_{\lambda}(X_{s-}+\xi(s,z))-f_{\lambda}(X_{s-})}-e^{Y_{s-}}-e^{Y_{s-}}\big(f_{\lambda}(X_{s-}+\xi(s,z))-f_{\lambda}(X_{s-})\big)\Big)\,\nu(\d z)\,\d s\\
			=&1+\int_0^t\int_Z \big(e^{Y_{s-}+f_{\lambda}(X_{s-}+\xi(s,z))-f_{\lambda}(X_{s-})}-e^{Y_{s-}}\big)\,\tilde{N}(\d s,\d z).
		 \end{align*}
		Therefore, $Z_t=e^{Y_t}$ is a nonnegative local martingale. %--------
		%----------
		By using the mean value theorem twice, the fact $|f_\lambda'(x)|_{L(E;\mathbb{R})}\leq \lambda^{\frac{1}{2}}$, and also \eqref{sec-3-Lip}, we have
		\begin{align*}
			&  \Big|e^{f_{\lambda}(X_{s-}+\xi(s,z))-f_{\lambda}(X_{s-})}-1-f_{\lambda}'(X_{s-})(\xi(s,z))\Big|_E\\
		    =& e^{-f_{\lambda}(X_{s-})}\Big[ e^{f_{\lambda}(X_{s-}+\theta\xi(s,z))}f^\prime_{\lambda}(X_{s-}+\theta\xi(s,z))(\xi(s,z))-e^{f_{\lambda}(X_{s-})}f_{\lambda}'(X_{s-})(\xi(s,z))\Big]\\
			\leq& e^{f_{\lambda}(X_{s-}+\theta\xi(s,z))-f_{\lambda}(X_{s-})}\big|f_{\lambda}'(X_{s-}+\theta \xi(s,z))-f^\prime_{\lambda}(X_{s-}) \big|_E|\xi(s,z)|_E\\
			&+e^{-f_{\lambda}(X_{s-})}\Big[\big(e^{f_{\lambda}(X_{s-}+\theta\xi(s,z))}-e^{f_{\lambda}(X_{s-})}\big)(f^\prime_{\lambda}(X_{s-})(\xi(s,z))\Big]\\
			 =& e^{f_{\lambda}(X_{s-}+\theta\xi(s,z))-f_{\lambda}(X_{s-})}\big|f_{\lambda}'(X_{s-}+\theta \xi(s,z))-f^\prime_{\lambda}(X_{s-}) \big|_E|\xi(s,z)|_E\\
			&+e^{f_{\lambda}(X_{s-}+\theta_1\xi(s,z))-f_{\lambda}(X_{s-})}\big[f_{\lambda}'(X_{s-}+\theta_1\xi(s,z))(\theta\xi(s,z))f_{\lambda}'(X_{s-})(\xi(s,z))\big] \\
			\leq& 2C\lambda e^{\lambda^{\frac{1}{2}}|\xi(s,z)|}|\xi(s,z)|_E^2+C\lambda e^{\lambda^{\frac{1}{2}}|\xi(s,z)|}|\xi(s,z)|_E^2,
		\end{align*}
		where $0<\theta,\theta_1<1$. Hence
		\begin{align*}
			  f_{\lambda}(X_t)&\leq1+\log Z_t
			  +\int_0^t\int_Z\Big( e^{f_{\lambda}(u_{s-}+\xi(s,z))-f_{\lambda}(u_{s-})}-1-\big(f_{\lambda}(X_{s-}+\xi(s,z))-f_{\lambda}(X_{s-})\big)\Big)\nu(\d z)\,\d s \\
			&+\int_0^t\int_Z \Big(f_{\lambda}(X_{s-}+\xi(s,z))-f_{\lambda}(X_{s-})-f^\prime_{\lambda}(X_{s-})(\xi(s,z))\Big)\,\nu(\d z)\,\d s\\
			&\leq   1+\log Z_t+3C\lambda\int_0^T\int_Z e^{\lambda^{\frac{1}{2}}|\xi(s,z)|}|\xi(s,z)|^2\nu(\d z)\,\d s\\% 	
			&\leq 1+\log Z_t+3C M_{\lambda}.
		\end{align*}
	It follows that
	\begin{align}
	\begin{split}\label{expo-proof-eq-1}
		  \mathbb{P}\Big( \sup_{0\leq t\leq T}|X_t|\geq R\Big)&= \mathbb{P}\Big( \sup_{0\leq t\leq T}f_{\lambda}(X_t)\geq (1+\lambda R^2)^{\frac{1}{2}}\Big)\\
		&\leq \mathbb{P}\Big(\sup_{0\leq t\leq T}\log Z_t\geq (1+\lambda R^2)^{\frac{1}{2}}-1-3CM_{\lambda}\Big) \\
		&\leq\mathbb{P}\Big(\sup_{0\leq t\leq T} Z_t\geq e^{(1+\lambda R^2)^{\frac{1}{2}}-1-3C M_{\lambda}}\Big)\\
		&\leq e^{1+3C M_{\lambda}-(1+\lambda R^2)^{\frac{1}{2}}}\E\sup_{0\leq t\leq T}Z_t \\
		&\leq e^{1+3C M_{\lambda}-(1+\lambda R^2)^{\frac{1}{2}}}\\
		&\leq C_{\lambda}e^{-(1+\lambda R^2)^{\frac{1}{2}}},
		\end{split}
	\end{align}
	where $C_{\lambda}=e^{1+3C M_{\lambda}}$.
	\dela{
	\dela{Set $\beta(\lambda)=1+3C\lambda M-(1+\lambda R^2)^{\frac{1}{2}}$.} If $R^2\leq 6CM_{\lambda}$, \eqref{exp-tail-eq} always holds since
	\begin{align*}
		     \mathbb{P}\Big( \sup_{0\leq t\leq T}|X_t|_E\geq R\Big)\leq 1< 3e^{-\frac{1}{2}}\leq 3e^{-\frac{R^2}{12 CM_{\lambda}}}.
	\end{align*}
When $R^2>6C M$, let us choose $\lambda_0>0$ such that $1+\lambda_0 R^2=\Big(\frac{R^2}{6CM_{\lambda}}\Big)^2$.
	Since \eqref{expo-proof-eq-1} holds for any $\lambda>0$, we conclude
	\begin{align*}
		   \mathbb{P}\Big( \sup_{0\leq t\leq T}|X_t|_E\geq R\Big)\leq e^{1+3C\lambda_0M_{\lambda}-(1+\lambda_0 R^2)^{\frac{1}{2}} }< e^{1-\frac{R^2}{12CM_{\lambda}}}< 3e^{-\frac{R^2}{12CM_{\lambda}}}.
	\end{align*}
 }
	\end{proof}

%---------------------------subsection--------------------------------------------------------
	 \subsection{Maximal $L^p$ estimates for stochastic convolutions w.r.t. L\'{e}vy processes}
	 Now let us consider the issue of maximal $L^p$ estimates for the stochastic convolutions driven by a more general L\'{e}vy-type process. Here we assume that $r=2$, i.e. $E$ is a martingale type $2$ Banach space.  Let $(e^{tA})_{t\geq0}$ be a $C_0$-semigroup on E, with the generator denoted by $A$,  of contraction type with respect to the equivalent $2$-smooth norm $|\cdot|_E$ such that $\|e^{tA}\|\leq e^{t\alpha} $  for some $\alpha \geq 0$.

	 We will establish a type of maximal inequality for the following L\'{e}vy-type stochastic convolution
		               \begin{align}\label{sec-3-eq1}
			      X_t=\int_0^te^{(t-s)A}g_s \d W_s+\int_0^t\int_Z e^{(t-s)A}\xi(s,z)\tilde{N}(\d s,\d z), \;\;\; t\geq 0.
			\end{align}
%------------------------
%---------Theorem-----------
	\begin{thm}\label{max-Levy}
	Let $(e^{tA})_{t\geq0}$ be a $C_0$-contraction semigroup on a martingale type $2$ Banach space $(E,|\cdot|_E)$
		and let $(g_t)_{t\in[0,T]}$ be a process in $M([0,T];\gamma(H;E))$ and $\xi\in\mathcal{M}_T^2(\mathcal{P}\otimes\mathcal{Z},\d t\times\mathbb{P}\times\nu;E)$.
	   There exists a c\`adl\`ag modification $\bar{X}$ such that for $2\leq p<\infty$ and some constant $C_p$ depending on $p$ such that
	\begin{align}\label{MI-main-1}
		\E\sup_{0\leq t\leq T} |\bar{X}_t|_E^p\leq e^{\alpha T}C_{p}&\Big[\E\big(\int_0^T \|g_s\|^2_{\gamma(H,E)}\d s\big)^{\frac{p}2}+\E\Big(\int_0^T\int_Z |\xi(s,z)|^p_E\;\nu(\d z)\,\d s\Big)\nonumber\\
		 &+\E \Big(\int_0^T\int_Z|\xi(s,z)|^{2}_E\;\nu(\d z)\,\d s\Big)^{\frac{p}{2}}\Big].
	\end{align}
	\end{thm}

	%------------------------
	\begin{proof} Note that \eqref{MI-main-1} is trivially satisfied if the right-hand side is infinite. So let us assume 	\begin{align}\label{RHS-finite}
	\E\big(\int_0^T \|g_s\|^2_{\gamma(H,E)}\d s\big)^{\frac{p}2}+\E\Big(\int_0^T\int_Z |\xi(s,z)|^p_E\;\nu(\d z)\,\d s\Big)
		 +\E \Big(\int_0^T\int_Z|\xi(s,z)|^{2}_E\;\nu(\d z)\,\d s\Big)^{\frac{p}{2}}<\infty.
	\end{align}
	Just as in the proofs of Theorem \ref{Th-MI-mian-1}, we may assume that $\|e^{tA}\|\leq 1$ and show that inequality \eqref{MI-main-1} holds for  $g\in M^2([0,T],\gamma(H;\mathcal{D}(A)))$ and $\xi\in\mathcal{M}^2_T(\mathcal{P}\otimes\mathcal{Z},dt\times\mathbb{P}\times\nu;\mathcal{D}(A))$. Since the Poisson point process $\pi$ is $\sigma$-finite, there exists a sequence of sets $\{D_n\}_{n\in\mathbb{N}}$ such that $\cup_{n\in\mathbb{N}}D_n=Z$ and $\E N(t,D_n)<\infty$ for every $0<t<\infty$ and $n\in\mathbb{N}$.
	
	       	  Let us define a sequence $\{X^n\}_{n\in\N}$ of process $X^n:=(X^n_t)_{t\geq0}$ by
	\begin{align}\label{sec-3-Th-eq1}
	         X^n_t=\int_0^te^{(t-s)A}g_s\, \d W_s+\int_0^t\int_Z1_{D_n}(z) e^{(t-s)A}\xi(s,z)\,\tilde{N}(\d s,\d z),\ \ n\in\mathbb{N}.
	\end{align} 	
	Then (see \cite{Zhu2017} and \cite{[Neid]}) $X^n$ is a strong solution to the following equation:
	\begin{align}
	\d X_t^n=AX_t^n\,\d t+g_t\,\d W_t+\int_{D_n}\xi(t,z)\,\tilde{N}(\d t,\d z).
	\end{align}
	Hence we may apply Theorem \ref{theo-Ito-2} and use the fact that $\psi^\prime(x)(Ax)\leq 0$ for all $x\in D(A)$ to infer	\begin{align*}
	 \E\sup_{t\in[0,T]}|X_t^n|^p_E\leq C_{p}\E\Big(\int_0^T|g_s|^2\,\d s\Big)^{\frac{p}{2}}+C_{p}\E\Big(\int_0^T\int_Z|\xi(s,z)|_E^2\,\nu(\d z)\d s\Big)^{\frac{p}{2}}
				+C_{p}\E\int_0^T\int_Z|\xi(s,z)|^p\,\nu(\d z)\d s.
	\end{align*}
Meanwhile, by inequality \eqref{max-eq-1} and \eqref{RHS-finite} we observe
			\begin{align*}
				   \E\sup_{t\in[0,T]}|X_t^n-X_t|^p_E&=\E\sup_{t\in[0,T]}\Big|\int_0^T\int_Z e^{(t-s)A}(\xi(s,z)1_{D_n}-\xi(s,z))\,\tilde{N}(\d s,\d z)\Big|_E^p\\
				                        &\leq C_p\E\Big(\int_0^T\int_Z|\xi(s,z)1_{D_n}-\xi(s,z)|^2_E\,\nu(\d z)\d s\Big)^{\frac{p}{2}}\\
				                        &\hspace{1cm}+C_p\E\Big(\int_0^T\int_Z|\xi(s,z)1_{D_n}-\xi(s,z)|^p_E\,\nu(\d z)\d s\Big)\\
				&\rightarrow0,\ \ \text{as }n\rightarrow\infty.
				\end{align*}
			Note that here the constant $C_p$ is independent of $n$.
				Therefore, an argument similar to that in the proof of Theorem \ref{Th-MI-mian-1} shows that there exists a c\`{a}dl\`ag modification $\bar{X}$ of $X$ such that
				\begin{align*}
				 \E\sup_{t\in[0,T]}|\bar{X}_t|^p_E\leq C_{p}\E\Big(\int_0^T|g_s|^2\,\d s\Big)^{\frac{p}{2}}+C_{p}\E\Big(\int_0^T\int_Z|\xi(s,z)|_E^2\,\nu(\d z)\d s\Big)^{\frac{p}{2}}
				+C_{p}\E\int_0^T\int_Z|\xi(s,z)|^p\,\nu(\d z)\d s.
				\end{align*}   	
	\end{proof}
%-----------------Remark-----------------------------

\section{Application to stochastic 2D quasi-geostrophic equations}\label{SNSE}
We consider the following stochastic quasi-geostrophic equation in $\R^2$,
 \begin{align}
 \begin{split}\label{eq-quasi-geo}
& \d \theta(t)+\big[(v(t)\cdot \nabla )\theta(t)-\Delta\theta(t)\big]\d t= \int_Z\xi(t,z)\tilde{N}(\d t,\d z),\quad t>0,\\
& \theta(0)=\theta_0.
\end{split}
 \end{align}
  Here $\theta:\R^+\times\R^2\rightarrow\R$ denotes the temperature, $v:\R^2\rightarrow \R^2$ is the 2D velocity field and $\theta_0\in L^2(\mathbb{R}^2,\mathbb{R})$. We refer to \cite{Brez+Mot,LRZ,RZZ} for the background and more references on this model. Let $\psi :\R^2\rightarrow \R$ be the stream function which satisfies
  \begin{align*}
  (-\Delta)^{\frac12}\psi=-\theta.
   \end{align*}
  The velocity is expressed in terms of the stream function by
  \begin{align*}
  v=\rm{curl}~ \psi.
  \end{align*}
Clearly, the velocity $v$ can be represented in terms of the temperature by
  \begin{align*}
 v=\mathcal{R}\theta=(-\mathcal{R}_2\theta,\mathcal{R}_1\theta),
   \end{align*}
  where $\mathcal{R}_j\theta=\mathcal{F}^{-1}\Big[\frac{\xi_j}{|\xi|}\mathcal{F}\theta\Big]$, $j=1,2$ is the $j$-th Riesz transform.

  \begin{rem}
  Note that since the operator $-(-\Delta)^{\frac{\alpha}2}$ is also a generator of a $C_0$ contraction semigroup on $L^p(\R^2)$ for $p\in[1,\infty)$ (see e.g. \cite{Kwasnicki, Ma}), our results could be generalized to a fractionally dissipative quasi-geostrophic equation for some $\alpha \in (0,2)$ in a similar way. It will be further investigated in future works.
  \end{rem}
  %-----------------------
  Let $C^{\infty}_c(\R^2;\R^{d})$ be  the space of all infinitely differential functions $\phi:\R^{2}\rightarrow\R^{d}$ with compact support contained in $\R^2$.
For $m \in \mathbb{N}$ and $p \in [1,\infty]$,  denote by $W^{m,p}(\R^2)$ the Banach space of all elements of $L^p(\R^2;\R)$ whose weak partial derivatives up to order $m$ belong to $L^p(\R^2;\R)$ as well. The space $W^{1,2}(\R^2)$ will be denoted by $H^1(\R^2)$.
 It is well known that the set $C^{\infty}_c(\R^2;\R)$ is dense in the space $W^{m,p}(\R^2)$ if $p\in [1,\infty)$.   Let $\mathcal{V}$ denotes the space of all $v\in C_c^{\infty}(\R^2;\R^{2})$ such that $\text{div }v=0$.
We will denote by $L^p_{sol}(\R^2;\R^{2})$ the closure of $\mathcal{V}$ in $L^p(\R^2;\R^2)$.

  %---------

  %---------------
  Now let $b:C_c^{\infty}(\R^2;\R^{2})\times C_c^{\infty}(\R^2;\R)\times C_c^{\infty}(\R^2;\R) \rightarrow \R$ be a trilinear form defined by
    \begin{align}\label{eq-b-0}
            b(v,\phi,\eta)&:
            =\int_{\R^2} (v\cdot \nabla) \phi \,\eta\; \d x
            =\sum_{i=1}^2\int_{\R^2}v_iD^i\phi\,\eta\; \d x,
  \end{align}
  whenever
 $v \in C_c^{\infty}(\R^2;\R^{2})$ and $\phi,\eta \in C_c^{\infty}(\R^2;\R)$ such that the integral on
the right-hand side exists.

 If $\textrm{div}\,v=0$, then we have
  \begin{align*}
        \int_{\R^2} (v\cdot \nabla) \phi\,\eta\; \d x=-\int_{\R^2} (v\cdot \nabla) \eta\,\phi \;\d x.
  \end{align*}
  This gives that
  \begin{align}\label{eq-b-1}
            b(v,\phi,\eta)=-b(v,\eta,\phi).
  \end{align}
  So we infer
   \begin{align}\label{eq-b-2}
            b(v,\phi,\phi)=0.
  \end{align}
 Using the H\"older inequality and equality \eqref{eq-b-1} we can deduce  the following estimate:
\begin{equation}\label{eqn:4.00}
\vert b(v,\phi, \eta)\vert= \vert b(v,\eta,\phi)\vert \leq \vert v\vert_{L^4}   \vert \phi \vert_{L^4}\vert \nabla \eta\vert
_{L^2}.
\end{equation}
Hence $b$ can be extended to a trilinear continuous form on $L^4_{sol}(\R^2;\R^{2})\times L^4(\R^2;\R)\times H^1(\R^2)$.
Define a bilinear map $B$ by $B(v,\phi):=b(v,\phi,\cdot)$. It follows from \eqref{eqn:4.00} that
\begin{align*}
       B:L^4_{sol}(\R^2;\R^{2})\times L^4(\R^2;\R)\rightarrow H^{-1}(\R^2)
\end{align*}
and it satisfies that
\begin{align}\label{eq1-operator-B}
        \|B(v,\phi)\|_{H^{-1}(\R^2)}\leq  \vert v\vert_{L^4}   \vert \phi \vert_{L^4},\quad v\in L^4_{sol}(\R^2;\R^{2}),\;\phi\in L^4(\R^2;\R).
\end{align}
Here $H^{-1}(\R^2)$ is defined as the dual space of $H^1(\R^2)$.
Also we have for $v\in L^4_{sol}(\R^2;\R^{2})$, $\phi\in L^4(\R^2;\R)$, and $w\in H^{1}(\R^2)$,
\begin{align*}
 \langle B(v,\phi),w\rangle =b(v,\phi,w).
\end{align*}
Hence by \eqref{eq-b-2},
\begin{align*}
 \langle B(v,\phi),\phi\rangle =b(v,\phi,\phi)=0, \quad v\in L^4_{sol}(\R^2;\R^{2}),\;\phi\in L^4(\R^2;\R).
\end{align*}
 \begin{rem} Observe that $\mathcal{R}$ is a linear continuous operator such that for all $1<p<\infty$, $\mathcal{R}:L^p(\R^2;\R)\rightarrow L^p_{sol}(\R^2;\R^{2})$ (see \cite{Stein}). Clearly, for all $\theta \in L^p(\R^2;\R)$, applying the Fourier transform gives that
      \begin{align*}
           \mathcal{F}[\rm{div}\, \mathcal{R}\theta](\xi)&=\mathcal{F}[D_1(-\mathcal{R}_2\theta)+D_2(\mathcal{R}_1\theta)](\xi)\\
           &=-i\xi_1\mathcal{F}(\mathcal{R}_2\theta)+i\xi_2\mathcal{F}(\mathcal{R}_1\theta)\\
           &=-i\xi_1\frac{\xi_2}{|\xi|}\hat{\theta}+i\xi_2\frac{\xi_1}{|\xi|}\hat{\theta}\\
           &=i[-\xi_1\xi_2+\xi_2\xi_1]\frac{1}{|\xi|}\hat{\theta}=0.
      \end{align*}
Thus we have $\rm{div}\, \mathcal{R}\theta=0$.
  \end{rem}

  Let us define the operator $A:D(A)\subset L^4(\R^2;\R)\rightarrow L^4(\R^2;\R)$ by
  \begin{align*}
       &D(A)=W^{2,4}(\R^2)\cap W^{1,4}(\R^2),\\
       &Au=-\Delta u,\quad u\in D(A).
  \end{align*}
  Then \eqref{eq-quasi-geo} can be rewritten as
  \begin{align}
   \begin{split}\label{eq-quasi-geo-ab}
        \d\theta(t)&=-[A\theta(t)+B(\mathcal{R}\theta(t),\theta(t))]\d t+ \int_Z\xi(t,z)\tilde{N}(\d t,\d z),\\
        \theta(0)&=\theta_0.
        \end{split}
  \end{align}
  %------------------------
   \begin{rem}Note that Equation \eqref{eq-quasi-geo-ab} is closely related to the Navier-Stokes equation
    \begin{align}
   \begin{split}\label{eq-NSE}
        \d u(t)&=-[A_1u(t)+B(u(t),u(t))]\d t+ \int_Z\xi(t,z)\tilde{N}(\d t,\d z),\\
        u(0)&=u_0
        \end{split}
  \end{align}
  in the following sense: both equations have the same linear part $\Delta \theta$ and $\Delta u$ and the nonlinearities satisfy the cancellation property $\langle B(\mathcal{R}u,v),v\rangle=0$ and $\langle B(u,v),v\rangle=0$ respectively. For the Navier-Stokes equation,  $\mathcal{R}$ is replaced by $I$. The main difference between these two equations is that for quasi-geostrophic equation the operator $-A$ generates a contraction semigroup on $L^p$ for $1\leq p<\infty$, however, for the Navier-Stokes equation, the negative Stokes operator $-A_1$ generates a contraction semigroup only for $p=2$.
    \end{rem}

  \begin{lem}\label{B-coer}
There exists a constant $C$ such that for all $T>0$ and all $\theta_1,\theta_2\in L^4(0,T; L^4(\R^2;\R))$, the following inequality holds:
  \begin{align*}
       |B(\mathcal{R}\theta_1,\theta_1)-B(\mathcal{R}\theta_2,\theta_2)|_{L^2(0,T;H^{-1}(\R^2))}\leq C(|\theta_1|_{L^4(0,T;L^4(\R^2;\R))}+|\theta_2|_{L^4(0,T;L^4(\R^2;\R))} )|\theta_1-\theta_2|_{L^4(0,T;L^4(\R^2;\R))}.
  \end{align*}
  \end{lem}

  \begin{proof}
  By \eqref{eq1-operator-B} and the fact that Riesz transforms are bounded on $L^p$ for any $p\in(1,\infty)$ we obtain
     \begin{align*}
  |B(\mathcal{R}\theta,\phi)|_{H^{-1}(\R^2)}\leq  \vert \mathcal{R}\theta\vert_{L^4(\R^{2};\R^{2})}\vert \phi \vert_{L^4(\R^{2};\R)}\leq C \vert \theta\vert_{L^4(\R^{2};\R)}  \vert \phi \vert_{L^4(\R^{2};\R)},\quad \theta,\phi\in L^4(\R^2;\R).
    \end{align*}

  Then we infer
       \begin{align*}
   \int_0^T |B(\mathcal{R}\theta(t),\theta(t))|^2_{H^{-1}(\R^2)} \d t\leq C\int_0^T|\theta(t)|_{L^4}^4 \d t.
    \end{align*}
      %-----------------------
    If $\theta_1,\theta_2\in L^4(0,T;L^4(\R^2;\R))$ and $w\in H^{1}(\R^2)$ we have
    \begin{align*}
       \langle  B(\mathcal{R}\theta_1,\theta_1)-B(\mathcal{R}\theta_2,\theta_2),w\rangle&=b(\mathcal{R}\theta_1,\theta_1,w)-b(\mathcal{R}\theta_2,\theta_2,w)\\
       &=b(\mathcal{R}\theta_1,\theta_1-\theta_2,w)+b(\mathcal{R}\theta_1-\mathcal{R}\theta_2,\theta_2,w)\\
       &=\langle B(\mathcal{R}\theta_1,\theta_1-\theta_2),w\rangle+\langle B(\mathcal{R}(\theta_1-\theta_2),\theta_2),w\rangle.
    \end{align*}
    It follows that
     \begin{align*}
     |B(\mathcal{R}\theta_1,\theta_1)-B(\mathcal{R}\theta_2,\theta_2)|_{H^{-1}(\R^2)}\leq C(|\theta_1|_{L^4}+|\theta_2|_{L^4})|\theta_1-\theta_2|_{L^4}.
    \end{align*}
    Hence we have for $\theta_1,\theta_2\in L^4(0,T; L^4(\R^2;\R))$,
    \begin{align*}
         & \int_0^T   |B(\mathcal{R}\theta_1(t),\theta_1(t))-B(\mathcal{R}\theta_2(t),\theta_2(t))|_{H^{-1}(\R^2)}^2 \d t \\
          &\hspace{1cm}\leq 4C\Big(\int_0^T|\theta_1(t)-\theta_2(t)|^4_{L^4} \d t\Big)^{\frac12}\Big(\int_0^T|\theta_1(t)|^4_{L^4}+|\theta_2(t)|^4_{L^4} \d t\Big)^{\frac12}.
    \end{align*}
 \end{proof}

 Define a process $Z$ by
  \[Z(t)=\int_0^te^{-(t-s)A}\xi(s,z) \tilde{N}(\d s,\d z), \;\; t\geq 0.\]
  Equivalently, $Z$ is the unique solution to the following stochastic Langevin equation (see \cite[Lemma 3.2]{Zhu2017}):
  \begin{align*}
  & \d Z=-AZ\d t+\int_Z\xi(t,z)\tilde{N}(\d t,\d z),\\
  &Z(0)=0.
  \end{align*}
Let us fix a number $s\in (0,\frac12)$.
 \begin{assu}\label{4.1}
 Assume that $\xi:[0,T]\times\Omega\times Z\rightarrow W^{-s,4}(\R^2)$ is a predictable process which satisfies
 \begin{align}
 \E\int_0^T\int_Z|\xi(s,z)|^2_{W^{-s,4}(\R^2)}\nu(\d z)\d s<\infty.
\end{align}
 \end{assu}
For the definition of the spaces $W^{s,p}(\R^2)$ see e.g. \cite{Triebel}. We aim to look for a solution of \eqref{eq-quasi-geo-ab} with the following form
  \begin{align}\label{eq-2}
  \theta(t)=e^{-t A}\theta_{0}+Y(t)+Z(t).
  \end{align}
   For this purpose we shall first prove the following result.
   \begin{prop}Under Assumption \ref{4.1}, we have for $T>0$, $Z\in L^{4}(0,T;L^{4}(\R^{2};\R))$ $\mathbb{P}$-a.s., that is,
   \begin{align}\label{Z-in-L4}
        \int_0^T |Z(t)|^4_{L^4}\,\d t<\infty,  \quad\mathbb{P}\text{-a.s.}
   \end{align}
   \end{prop}

  \begin{proof}
By the Gagliardo-Nirenberg interpolation inequality (see e.g. \cite[Proposition 5.6 and Remark 5.8]{Brezis+Miron}) we have for $0<s<\infty$\begin{align*}
|Z|_{L^4}^4\leq C|Z|_{W^{s,4}(\R^2)}^2|Z|_{W^{-s,4}(\R^2)}^2.
\end{align*}
It follows that
\begin{align*}
       \int_0^T |Z(t)|^4_{L^4} d t&\leq C\int_0^T |Z(t)|^2_{W^{s,4}(\R^2)}|Z(t)|^2_{W^{-s,4}(\R^2)}\d t \\
       &\leq C\sup_{0\leq t\leq T}|Z(t)|^2_{W^{-s,4}(\R^2)}\int_0^T|Z(t)|^2_{W^{s,4}(\R^2)}\d t.
\end{align*}
Since the space $E=W^{-s,4}(\R^2)$ is of martingale type $2$ (see \cite{Brz1}) and $-A$ generates a contraction semigroup in $E$, we may apply inequality \eqref{Lp-estimates-1} to get
\begin{align}\label{sc-eq-1}
\E \sup_{0\leq t\leq T}|Z(t)|^2_{W^{-s,4}(\R^2)}\leq C\E\int_0^T|\xi(t,z)|^2_{W^{-s,4}(\R^2)}\nu(\d z)\d t<\infty,
\end{align}
which also gives a $W^{-s,4}(\R^{2};\R)$-valued c\`{a}dl\`ag modification of $Z(t)$.

On the other hand because $W^{s,4}(\R^2;\R)$ is also of martingale type $2$, by using inequality (1.7) from \cite{[Brz-Hau]} we have
\begin{align}
 \begin{split}\label{sc-eq-2}
    \E \int_0^T|Z(t)|^2_{W^{s,4}(\R^2)}\d t&=\E \int_0^T|A^{\frac{s}2}Z(t)|^2_{L^4(\R^{2};\R)}\d t\\
    &=\E \int_0^T|Z(t)|^2_{D(A^{\frac{s}2})}\d t\\
    &\leq \E \int_0^T\int_Z |\xi(t,z)|^2_{D(A^{\frac{s}2-\frac12})}\nu(\d z)\d t\\
    &=\E \int_0^T\int_Z |\xi(t,z)|^2_{W^{s-1,4}(\R^2)}\nu(\d z)\d t\\
    &\leq \E\int_{0}^{T}\int_{Z}|\xi(t,z)|^{2}_{W^{-s,4}}\nu(\d z)\d t<\infty,
    \end{split}
\end{align}
provided $-s\geq s-1$, i.e. $s\leq \frac12$.
 \end{proof}

  \begin{rem}
 By the classical Gagliardo-Nirenberg inequality \cite[Inequality (2.2)]{Niren}, we infer that
\[
|Z|^4_{L^4(0,T;L^4)}\leq |Z|^2_{L^{\infty}(0,T;L^{2})}|Z|^2_{L^2(0,T;H^1(\R^2))}.
\]
 So one sufficient condition for \eqref{Z-in-L4} is
  \begin{align*}
       \E\int_0^T |Z(t)|^2_{H^1(\R^2)}\d t<\infty
  \end{align*}
  and
   \begin{align*}
       \E\sup_{0\leq t\leq T} |Z(t)|^2_{L^2}\d t<\infty.
  \end{align*}
  \end{rem}

 Notice that \eqref{eq-quasi-geo-ab} can be rewritten as
  \begin{align}\label{eq-Y}
       \begin{split}
&  \d Y+AY\d t+B(\mathcal{R}(Y+Z),Y+Z)\d t=0,\\
 & Y(0)=Y_0.
      \end{split}
  \end{align}
  %where $Y_0=\theta_0$.
  %-------------------------------
  \begin{prop}
Assume that $Z(\cdot)\in L^4(0,T;L^4(\R^2;\R))$. Then for every $Y_0 \in L^2(\R^2;\R)$, there exists a unique solution $Y$  to the equation (\ref{eq-Y})
  and it satisfies
\begin{enumerate}
\item[]
  \begin{align}\label{Y-L-4}
Y(\cdot)\in L^4(0,T,L^4(\R^2;\R)),
 \end{align}
\item[]
 \begin{align}\label{Y-L-2-H-V}
 Y(\cdot)\in L^{\infty}(0,T;L^2(\R^2;\R))\cap L^2(0,T;H^1(\R^2)).
  \end{align}
\end{enumerate}
Moreover, we have
\begin{align}\label{eqn-QGEs-apriori-1}
\sup_{t\in[0,T]}
      |Y(t)|^2_{L^{2}}\leq e^{C_1 \int_0^T|Z(t)|^4_{L^4}\d t}|Y_0|^2_{L^2}+C_2\int_0^Te^{C_1\int_s^T|Z(r)|^4_{L^4}\d r}|Z(s)|^4_{L^4}\d s
\end{align}
and
\begin{align}\label{eqn-QGEs-apriori-2}
   \int_0^T |\nabla Y(t)|^2_{L^2}\d t
 \leq |Y_0|^2_{L^2}+C_1\sup_{t\in[0,T]} |Y(t)|^2_{L^2}\int_0^T|Z(t)|^4_{L^4}\d t+ C_2\int_0^T |Z(t)|^2_{L^4}\d t.
 \end{align}

\end{prop}

%-------------
\begin{proof}
Let us put
\[
\rho:=|Y_0|^2_{L^2}\, e^{C_1 \int_0^T|Z(t)|^4_{L^4}\d t}+C_2\int_0^Te^{C_1\int_t^T|Z(s)|^4_{L^4}\d s}|Z(t)|^4_{L^4}\d t.
\]
Note that   $\vert Y_0 \vert_{L^2}^2\leq  \rho$.
Because of Lemma \ref{B-coer},  applying similar arguments as in \cite[Theorem 15.2.5]{Da+Zab-1996} and then the Banach fixed point theorem yields the existence of some $T_1\in (0,T]$ and
 unique local solution in $L^4(0,T_1;L^4(\R^2;\R))$ to \eqref{eq-Y}. Below, we will show that  this solution could be extended to the whole interval $[0,T]$.

For this aim, by applying Lemma III.1.2 in \cite{Temam_2001} we have for $t \in [0,T_1]$,
\begin{align*}
   \frac12 \frac{\d}{\d t}|Y(t)|^2_{L^{2}}&=\langle Y(t), \frac{\d}{\d t}Y(t)\rangle_{L^{2}}\\
   &=-\langle Y(t), AY(t)\rangle +\langle B(\mathcal{R}(Y(t)+Z(t)),Y(t)+Z(t)),Y(t) \rangle\\
   &=-\int_{\R^{2}}\|\nabla Y(t)  \|^2\d x+\langle B(\mathcal{R}(Y(t)+Z(t)),Y(t)+Z(t)),Y(t)\rangle\\
   &=-|\nabla Y(t)|^2_{L^2}-b\big(\mathcal{R}(Y(t)),Y(t),Z(t)\big)-b\big(\mathcal{R}(Z(t)),Y(t),Z(t)\big).
\end{align*}
Then by \eqref{eqn:4.00} and the boundedness of the operator $\mathcal{R}$ on $L^{4}$ we obtain that
\begin{align*}
& |b(\mathcal{R}(Y(t)),Y(t),Z(t))|\leq C|Y(t)|_{L^4}|\nabla Y(t)|_{L^2}|Z(t)|_{L^4},\\
 &|b(\mathcal{R}(Z(t)),Y(t),Z(t))|\leq C|Z(t)|_{L^4}|\nabla Y(t)|_{L^2}|Z(t)|_{L^4}.
\end{align*}
Due to the Ladyzhenskaya inequality (see Lemma III 3.3 in \cite{Temam_2001})
\begin{align}\label{ineq-Lad}
|Y(t)|_{L^4}\leq 2^{\frac14}|{\nabla} Y(t)|_{L^2}^{\frac12}|Y(t)|_{L^2}^{\frac12},
\end{align}
 we have
\begin{align*}
        |b(\mathcal{R}(Y(t)),Y(t),Z(t))|&\leq
        C2^{\frac14}|\nabla Y(t)|_{L^2}^{\frac32}|Y(t)|^{\frac12}_{L^2}|Z(t)|_{L^4}.
\end{align*}
By Young's inequality we get
\begin{align*}
        |b(\mathcal{R}(Y(t)),Y(t),Z(t))|\leq
        \frac14|\nabla Y(t)|_{L^2}^{2}+\frac{27}{2}C^4|Y(t)|^2_{L^2}|Z(t)|^4_{L^4}.
\end{align*}
Again, Young's inequality gives that
\begin{align*}
        |b(\mathcal{R}(Z(t)),Y(t),Z(t))|&\leq C|Z(t)|^2_{L^4}|\nabla Y(t)|_{L^2}\\
        &\leq \frac14 |\nabla Y(t)|_{L^2}^2+C^2|Z(t)|^4_{L^4}.
\end{align*}
It follows that
\begin{align}
   \frac12 \frac{\d}{\d t}|Y(t)|^2_{L^2}+\frac12|\nabla Y(t)|_{L^2}^2
 & \leq  \frac{27}{4}C^4|Y(t)|^2_{L^2}|Z(t)|^4_{L^4}+C^2|Z(t)|^2_{L^4}\nonumber\\
  &=\frac12C_1 |Y(t)|^2_{L^2}|Z(t)|^4_{L^4}+ \frac12C_2 |Z(t)|^2_{L^4},\label{est-04.01}
\end{align}
where $C_1=\frac{27}{2}C^4$ and $C_2=2C^2$.

 Now by Gronwall's lemma we have for $t\in [0,T_1]$,
\begin{align}
      |Y(t)|^2_{L^2}\leq e^{C_1 \int_0^T|Z(s)|^4_{L^4}\d s}|Y_0|^2_{L^2}+C_2\int_0^Te^{C_1\int_s^T|Z(r)|^4_{L^4}\d r}|Z(s)|^4_{L^4}\d s,
\end{align}
 where the constants $C_1,C_2$ are independent of $Y_0$ and $Z$.

 Integrating \eqref{est-04.01} we obtain
\begin{align*}
   \int_0^{T_1}|\nabla Y(t)|^2_{L^2}\d t
 \leq |Y_0|^2_{L^2}+C_1\sup_{t\in[0,T_1]} |Y(t)|^2_{L^2}\int_0^{T}|Z(t)|^4_{L^4}\d t+ C_2\int_0^T |Z(t)|^2_{L^4}\d t.
\end{align*}
Since by assumption $Z(\cdot)\in L^4(0,T;L^4(\R^2;\R))$, it follows from the inequality \eqref{est-04.01} that  $|Y(t)|^2_{L^2}\leq \rho$, for all $t \in [0,T_1]$. In other words, the local solutions cannot blow up in finite time. Thus, by a simple and standard contradiction argument we infer that $T_1=T$. Inequalities \eqref{eqn-QGEs-apriori-1} and \eqref{eqn-QGEs-apriori-2} follow from the above two inequalities. The proof is complete.
\end{proof}

\begin{defn} A $W^{-s,4}(\R^2)$-valued adapted c\`{a}dl\`{a}g process $\theta(t)$, $t\in[0,T]$, is a solution to \eqref{eq-quasi-geo} if it satisfies
\begin{align*}
\int_0^T|\theta(t)|^4_{L^4}\d t<\infty\ \text{a.s.}
\end{align*}
and for $t\in[0,T]$, the following equality holds a.s.,
\begin{align*}
         \theta(t)=\theta_0-\int_0^t e^{-(t-s)A}B(\mathcal{R}\theta(s),\theta(s))\d s+\int_0^te^{-(t-s)A}\xi(s,z)\,\tilde{N}(\d s,\d z).
\end{align*}
\end{defn}
%------------------
\begin{thm} Suppose Assumption \ref{4.1} holds. Then for any $\theta_0\in W^{-s,4}(\mathbb{R}^{2})$, there exists a unique solution $\theta(t)$, $t\in[0,T]$, to \eqref{eq-quasi-geo}.
\end{thm}
\begin{proof}
Notice that if $Z$ satisfies \eqref{Z-in-L4} and $Y$ satisfies \eqref{Y-L-4}, then $\theta$ defined by \eqref{eq-2} with $Y_{0}=0$ satisfies
\begin{align}
\theta (\cdot)\in L^4([0,T],L^4(\R^2;\R))\quad\text{a.s.}
\end{align}
Now let us show the uniqueness. Assume that $\theta, \tilde{\theta}$ are two solutions of \eqref{eq-quasi-geo}.
We define
\begin{align*}
\tilde{Y}(t)=\tilde{\theta}(t)-Z(t),
\end{align*}
 hence both $Y$ and $\tilde{Y}$ solve \eqref{eq-Y}.  Moreover, since $Y,\tilde{Y}\in L^4(0,T;L^4(\R^2;\R))$ a.s., then by the uniqueness of \eqref{eq-Y} we have $Y=\tilde{Y}$ a.s.
\end{proof}
%--------------------

\begin{rem}
One might also study \eqref{eq-quasi-geo} by following the method from previous works \cite{[Brz-Hau-Zhu],Br+Liu+Zhu}. However, our approach here is different. For example, comparing with \cite{[Brz-Hau-Zhu]}, where $u\in L^{\infty}(0,T;L^{2}(D))\cap L^2(0,T;H^1_0(D))$, in this paper we get the solution of  $\theta\in L^4(0,T;L^4(\R^2))$ under much weaker assumptions on the jump process than that in \cite{[Brz-Hau-Zhu]}, in which even the Ornstein-Uhlenbeck process $Z$ is assumed to satisfy $Z\in  L^2(0,T;H^1_0(D))$. A generalization of the current result to multiplicative noise will be investigated in future work.
\end{rem}

\begin{rem}
Stochastic Navier-Stokes equations with Poisson noise in $L^p$-spaces are investigated in \cite{FRS}. The key tool employed in the proof of existence and uniqueness of a local mild solution theory in \cite{FRS} is a type of maximal inequality that was recently developed by the first two authors of this paper and Hausenblas in \cite{Zhu2017}. It's worth mentioning that \cite{Zhu2017} deals with maximal inequalities with respect to stronger norms. More precisely, the $p$-th power of the norm is assumed to be of $C^2$. Our results here avoid such conditions and our assumptions here are much easier to verify. The advantage of the maximal inequalities here becomes more transparent when one deals with closed subspace of Banach spaces.
\end{rem}

	%-----------------------------------Appendix----------------------

 \appendix
\section{Stochastic integration w.r.t. Poisson random
measures}
In this section we give a brief review of  some basic terminology and known results on the stochastic integral w.r.t. Poisson random measures.

Suppose that $(Z,\mathcal{Z})$ is a measurable space and $\nu$ is a nonnegative $\sigma$-finite measure on it.  It is known that there exists a stationary Poisson point process $\pi=(\pi(t))_{t\geq 0}$ on $(Z,\mathcal{Z})$ with the intensity measure $\nu$; see \cite{[Sato],[Rong]}. Let $N$ be the counting measure associated with $\pi$ which is defined by
  \begin{align}\label{eq-N-0}
     N(U,\omega):=\sum_{s\in\mathcal{D}(\pi(\omega))}I_U(s,\pi(s,\omega)),\ \ U\in\mathcal{B}((0,\infty))\otimes\mathcal{Z}, \  \omega\in\Omega.
 \end{align}
 In particular, we have
 \begin{align}\label{eq-N-1}
     N((0,t]\times A,\omega)=\sum_{s\in(0,t]\cap\mathcal{D}(\pi(\omega))}I_A(\pi(s,\omega)),\ \ A\in\mathcal{Z},\ \  0<t<\infty.
 \end{align}

According to \cite[Theorem 3.1]{[Ito-0]}, $N$ is a Poisson random measure with  the intensity measure $\nu$. That is for every $U\in\mathcal{B}(0,\infty)\otimes\mathcal{Z}$ with $\E N(U)<\infty$, the random variable $N(U)$ is Poisson distributed and for any pairwise disjoint sets $U_1,\dots,U_n\in\mathcal{B}(0,\infty)\otimes\mathcal{Z}$,
 the random variables $N(U_1),\dots N(U_n)$ are independent. Here as usual we shall employ the notation
\begin{align*}
    \tilde{N}=N-\Leb\otimes\nu
\end{align*}
to denote the compensated Poisson random measure of $N$.

We will define the stochastic integral of $\mathcal{P}\otimes \mathcal{Z}$-measurable functions under the martingale type $r$ Banach space setting.  Let us recall the definition of martingale type $r$ Banach spaces. 	
	      \begin{defn}
A Banach space $E$ with norm $|\cdot|_E$ is of martingale type $r$, for $r\in[1,2]$ if there exists a constant $C_r(E)>0$ such that for any $E$-valued discrete martingale $\{M_k\}_{k=1}^n$ the following inequality holds
\begin{align}\label{M-type-r-def}
      \E|M_n|_E^r\leq C_r(E)\sum_{k=0}^n\E|M_k-M_{k-1}|_E^r,
\end{align}
with $M_{-1}=0$ as usual.
\end{defn}
           Note that every real separable Banach space is of martingale type 1 and every Hilbert space is of martingale type 2. If a real separable Banach space is of martingale type $p$ for some $1 < p \leq2$ then it is of martingale type $r$ for all $r\in [1, p]$.
            All $L^p$ spaces for $p\geq 1$ and Sobolev spaces $W^{k,p}$ for $p\geq 1$ and $k>0$ are martingale type $p\wedge 2$ Banach spaces. For more details of this subject we refer the reader to \cite{[Pisier]}.

For $T\in (0,\infty)$, let $\mathcal{M}^r_T(\mathcal{P}\otimes\mathcal{Z},\d t\times\mathbb{P}\times\nu;E)$ denote the
linear space consisting of (equivalence classes of) all $\mathcal{P}\otimes \mathcal{Z}$-measurable functions
$\xi:[0,T]\times\Omega\times Z\rightarrow E$ such that
\begin{align}\label{integrability}
     \int_0^T\int_Z\mathbb{E} |\xi(t,z)|^r_{E}\,\nu(\d z)\,\d t<\infty.
\end{align}
Note that $\mathcal{M}^r_T(\mathcal{P}\otimes\mathcal{Z},\d t\times\mathbb{P}\times\nu;E)$
is the usual $L^r$ space of $E$-valued functions on $[0,T]\times\Omega\times Z$ with respect to the $\sigma$-field $\mathcal{P}\otimes\mathcal{Z}$ and the measure
$\Leb\otimes\mathbb{P}\otimes\nu$.

Let $\mathcal{M}^r_{step}([0,T]\times \Omega\times Z,\mathcal{P}\otimes\mathcal{Z};E)$ be the space of all functions $f:[0,T]\times\Omega\times Z\rightarrow E$ for which there exists a partition $0=t_0< t_1<\dots<t_n=T$ and a finite family $A_{j-1}^k$, $j=1,\dots,n$, $k=1,\dots,m$, of sets from
$\mathcal{Z}$ with $\nu(A_{j-1}^k)<\infty$  such that
   \begin{align}\label{step function}
       f(t,\omega,z)=\sum_{k=1}^{m}\sum_{j=1}^{n}\xi^{k}_{j-1}(\omega)1_{(t_{j-1},t_j]}(t)1_{A^k_{j-1}}(z), \ \ (t,\omega,z)\in[0,T]\times\Omega\times Z.
   \end{align}
 Here $\xi^k_{j-1}$ is an $E$-valued $\mathcal{F}_{t_{j-1}}$-measurable random variable, for every $j=1,\dots,n$ and $k=1,\dots,m$, and for each $j=1,\dots,n$, the sets $A_{j-1}^k$, $k=1,\dots,m,$ are pairwise disjoint.

For each $\xi \in \mathcal{M}^r_{step}([0,T]\times \Omega\times Z,\mathcal{P}\otimes\mathcal{Z};E)$, we set
  \begin{align*}
      \bar{I}(\xi):=\sum_{k=1}^{m}\sum_{j=1}^{n}\xi^k_{j-1}(\omega)\tilde{N}((t_{j-1},t_j]\times A_{j-1}^k).
  \end{align*}

%-------------------------
\begin{lem}\cite[Lemma C.2]{[Brz-Hau]} Assume that $E$ is a martingale type $r$ Banach space, $r\in(1,2]$. For every $\xi\in \mathcal{M}^r_{step}([0,T]\times \Omega\times Z,\mathcal{P}\otimes\mathcal{Z};E)$, $\bar{I}(\xi)\in  L^r(\Omega,\mathcal{F},E)$, $\E I(\xi)=0$, and
\begin{align*}
  \E|\bar{I}(\xi)|^r\leq C\,\E\int_0^T\int_Z|\xi(t,z)|^r_E\,\nu(\d z)\d t,
\end{align*}
where the constant $C$ depends on $C_r(E)$ from \eqref{M-type-r-def}.
\end{lem}

\begin{thm}\label{thm-SI-def}\cite[Theorem C.1]{[Brz-Hau]} Assume that $E$ is a martingale type $r$ Banach space, $r\in(1,2]$. For every $\xi\in \mathcal{M}^r_T(\mathcal{P}\otimes\mathcal{Z},\d t\times\mathbb{P}\times\nu;E)$, there exists a unique bounded linear operator
\begin{align*}
      I: \mathcal{M}^r_T(\mathcal{P}\otimes\mathcal{Z},\d t\times\mathbb{P}\times\nu;E)\rightarrow L^r(\Omega,\mathcal{F},E)
\end{align*}
such that
\begin{align*}
  \E|I(\xi)|^r\leq C\,\E\int_0^T\int_Z|\xi(t,z)|^r_E\,\nu(\d z)\d t.
\end{align*}
In particular, if $\xi \in \mathcal{M}^r_{step}([0,T]\times \Omega\times Z,\mathcal{P}\otimes\mathcal{Z};E)$, then $I(\xi)=\bar{I}(\xi)$.

 Moreover, the process $\int_0^t\int_Z \xi(s,z)\tilde{N}(\d s, \d z):=I(1_{[0,t]}\xi)$, $t\in[0,T]$, is an $E$-valued $r$-integrable c\`adl\`ag martingale with mean 0 and it satisfies
\begin{align}\label{sec-2-SI}
      \E\Big|\int_0^{t}\int_Z\xi(s,z)\,\tilde{N}(\d s,\d z)\Big|^r_E\leq \E \int_0^{t}\int_Z|\xi(s,z)|_E^r\,\nu(\d z)\d s,\quad t\in[0,T].
      \end{align}
\end{thm}

Analogously,  if $\tau$ is a stopping time with $\mathbb{P}\{\tau\leq T\}=1$,  we shall define $\int_0^{\tau}\int_Zf(s,z)\tilde{N}(ds,dz):=I(1_{(0,\tau]}\xi)$ and by Theorem \ref{thm-SI-def} it satisfies
\begin{align}\label{sec-2-stopped-SI}
      \E\Big|\int_0^{\tau}\int_Z\xi(s,z)\,\tilde{N}(\d s,\d z)\Big|^r_E\leq \E \int_0^{\tau}\int_Z|\xi(s,z)|_E^r\,\nu(\d z)\d s.
      \end{align}
\begin{rem}
It's worth mentioning that when the function is predictable, the stochastic integral studied in \cite{[Brz-Hau]} for progressively measurable functions coincides with the stochastic integral we defined here for predictable functions. But for predictable but not progressively measurable functions, the stochastic integral in \cite{[Brz-Hau]} differs from the one considered in \cite{[Ikeda]} where the stochastic integral is defined as a limit of a sequence of Bochner integrals.
\end{rem}
%---------------------------------------

%-----------------------
 \begin{lem}\label{sec-2-main-lem} Assume that $E$ is a Banach space. Let $\mathcal{L}$ be a linear space of $\mathcal{B}([0,T])\otimes\mathcal{F}_T\otimes\mathcal{Z}$-measurable $E$-valued (resp. $\mathbb{R}_+$-valued) functions defined on $[0,T] \times \Omega \times Z $.
		If $\mathcal{K}$ is a linear subspace of $\mathcal{L}$ satisfying the following two conditions.
		\begin{enumerate}
			\item $\mathcal{K}$ contains all functions of the form
			          \begin{align}\label{sec-2-eq-31}
				                     \sum_{i=1}^nx_i1_{F_i}1_{(t_i,t_{i+1}]}1_{B_i}+x_01_{F_0}1_{\{0\}}1_{B_0},\ \
				      \end{align}
				where $0=t_0<t_1<\cdots<t_n=T$, $x_i\in E$, (resp. $x_i\geq0$), $B_i\in \mathcal{Z}$ with $\nu(B_i)<\infty$, and $\ F_i\in\mathcal{F}_{t_i}\ \text{for } i=0,\dots,n;   $
				 \item \label{condition-2}
				if $\{f_n\}$ is a sequence (resp. monotone increasing sequence) in $\mathcal{K}$, $f\in\mathcal{L}$ and
				   $|f_n-f|_E$ decreases to $0$ on $[0,T]\times \Omega\times Z$, then $f\in\mathcal{K}$,
		\end{enumerate}
		then $\mathcal{K}$ contains all $\mathcal{P}\otimes\mathcal{Z}$-measurable $E$-valued (resp. $\mathbb{R}_+$-valued) functions in $\mathcal{L}$.
	\end{lem}

	\begin{proof}
	   Note that for every $E$-valued
		$\mathcal{P}\otimes\mathcal{Z}$-measurable function $f$ in $\mathcal{L}$, we can always find a sequence of simple functions $\{f^n\}$ of the form $\sum_{j=1}^m x_j1_{A_j}$, $x_j\in E$, $A_j\in\mathcal{P}\otimes\mathcal{Z}$ such that $|f^n-f|_E$ decreases to $0$, as $n\rightarrow\infty$.  In particular, if $f$ is a $\mathcal{P}\otimes\mathcal{Z}$-measurable and positive function, then  $f$ is a monotone increasing limit of  $\mathcal{P}\otimes\mathcal{Z}$-measurable and positive simple functions. Hence by condition \ref{condition-2} and linearity of $\mathcal{K}$ we only need to show that for every $x\in E$ and $A\in\mathcal{P}\otimes\mathcal{Z}$, $x1_A\in\mathcal{K}$.

		It is known (e.g. \cite{[Metivier]}) that the predictable $\sigma$-field is also generated by the following family of sets
		\begin{align}
		   \mathcal{R}= \{(s,t]\times F:\ 0\leq s\leq t\leq T,\ F\in\mathcal{F}_s\}\cup\{\{0\}\times F,\ F\in\mathcal{F}_0\}.
		\end{align}
		Let $\mathcal{G}:=\{B\in Z: \nu(B)<\infty\}$. Since $(Z,\mathcal{Z},\nu)$ is a $\sigma$-finite measure space, $\mathcal{Z}$ contains an exhausting sequence $(D_j)_{j\in\mathbb{N}}$ of sets such that $\nu(D_j)<\infty$ for all $j\in\mathbb{N}$. So it is easy to see that $\sigma(\mathcal{G})=\mathcal{Z}$ and $\mathcal{G}$ contains the exhausting sequence $(D_j)_{}j\in\mathbb{N}$. Hence we have $\mathcal{P}\otimes\mathcal{Z}=\sigma(\mathcal{R})\otimes\sigma(\mathcal{G})=\sigma(\mathcal{R}\times\mathcal{G})$. In other words, $\mathcal{P}\otimes\mathcal{Z}$ is generated by
		    \begin{align}\label{sec-2-sf}
		   \hat{\mathcal{R}}= \{(s,t]\times F\times B:\ 0\leq s\leq t\leq T,\ F\in\mathcal{F}_s,\ B\in\mathcal{G}\}\cup\{\{0\}\times F\times B,\ F\in\mathcal{F}_0,\ B\in\mathcal{G}\}.
		\end{align}
		Let us set
		\begin{align*}
			  \mathcal{P}'=\{A: A\in\mathcal{P}\otimes\mathcal{Z}\ s.t.\ x1_A\in\mathcal{K},\ \forall\ x\in E\}
		\end{align*}
		Then it is straightforward to see that  $\mathcal{P}'$ is a Dynkin system. Let $\mathcal{A}$ be the collection of all finite unions of sets in $\hat{\mathcal{R}}$.
		Clearly, $\mathcal{A}$ is closed under finite intersections. Indeed, take $A=\cup_{i=1}^k A_i$ and $B=\cup_{j=1}^l$ from $\mathcal{A}$, where $A_i, B_j\in\hat{\mathcal{R}}$. Then $A\cap B=\cup_{i\leq k,j\leq l}(A_i\cap B_j)\in\mathcal{A}$ which shows that $\mathcal{A}$ admits finite intersections, i.e. $\mathcal{A}$ is a $\pi$ system. Also by (1), we find $\mathcal{A}\subset \mathcal{P}'$. Hence we may apply Dynkin's lemma to get $\mathcal{P}\otimes\mathcal{Z}=\sigma(\mathcal{A})\subset\mathcal{P}'$. Therefore, we obtain $\mathcal{P}'=\mathcal{P}\otimes\mathcal{Z}$ which shows that for every $x\in E$ and $A\in\mathcal{P}\otimes\mathcal{Z}$, $x1_A\in\mathcal{K}$.
	\end{proof}
	\begin{prop}\label{A.prop-2} Assume that $E$ is a Banach space. If
		$ \xi:  [0,T] \times \Omega \times Z \to E $ is a $\mathcal{P}\otimes\mathcal{Z}$-measurable function such that either $$\E\int_0^T\int_Z|\xi(s,z)|_E\,N(\d s,\d z)<\infty$$ or $\E\int_0^T\int_Z|\xi(s,z)|_E\,\nu(\d z)\d s<\infty$, then we have
		\begin{align}\label{sec-2-eq-13}
			\E\int_0^T\int_Z\xi(s,z)\,N(\d s,\d z)=\E\int_0^T\int_Z\xi(s,z)\,\nu(\d z)\d s.
		\end{align}
	\end{prop}
	\begin{proof} We first show that equality \eqref{sec-2-eq-13} holds for any positive $\mathcal{P}\otimes\mathcal{Z}$-measurable function $f$.
		 Let $\mathcal{L}_1:=\{f:[0,T]\times \Omega\times Z\rightarrow \mathbb{R}_+\,\big{|} \, f\,\text{is }\mathcal{B}([0,T])\otimes\mathcal{F}_T\otimes\mathcal{Z}/\mathcal{B}(\mathbb{R}_+)\text{-measurable}\}$.   	 Define
		\begin{align*}
			\mathcal{K}_1=\{f:[0,T]\times\Omega\times Z\rightarrow \mathbb{R}_+\,\big{|}\,f\text{ is in }\mathcal{L}_1\text{ and satisfies } \eqref{sec-2-eq-13} \}.
		\end{align*}
		Take a function $f$ of the form
		  \begin{align}
	                     \sum_{i=1}^na_i1_{F_i}1_{(t_i,t_{i+1}]}1_{B_i}+a_01_{F_0}1_{\{0\}}1_{B_0},\ \ t_i<t_{i+1},\,a_i\geq0,\ B_i\in \mathcal{G},\ F_i\in\mathcal{F}_{t_i},\ \text{for } i=0,\cdots,n.
	      \end{align}
		 Then by the independence of $1_{F_i}$ and $N((t_i,t_{i+1}]\times B_i)$ and the stationary property of the Poisson random measure $N$, we have
		\begin{align*}
			     \E\int_0^T\int_Zf(s,z)N(\d s,\d z)&=\E\sum_{i=1}^ma_i1_{F_i}N((t_i,t_{i+1}]\times B_i)
			                                        =\sum_{i=1}^ma_i\E1_{F_i}\E N((t_i,t_{i+1}]\times B_i)\\
			                                        &=\sum_{i=1}^ma_i\E1_{F_i}\E \nu(B_i)(t_{i+1}-t_i)
			=\E\int_0^T\int_Zf(s,z)\,\nu(\d z)\d s.
		\end{align*}
Suppose that $\{f_n\}\subset \mathcal{K}_1$ is a monotone increasing sequence of positive functions and $f_n$ converges to $f$. Then by the monotone convergence theorem
\begin{align*}
	     \E\int_0^T\int_Zf(s,z)N(\d s,\d z)& =\sup_n  \E\int_0^T\int_Zf^n(s,z)\,N(\d s,\d z) \\
	                                       & = \sup_n\E\int_0^T\int_Zf^n(s,z)\,\nu(\d z)\d s= \E\int_0^T\int_Zf(s,z)\,\nu(\d z)\d s.
	\end{align*}
	It follows from Lemma \ref{sec-2-main-lem} that $\mathcal{K}_1$ contains all $\mathcal{P}\otimes\mathcal{Z}$-measurable positive functions.
	
		                    	Now let us take an $E$-valued and $\mathcal{P}\otimes\mathcal{Z}$-measurable function $\xi$.
		From above discussion, we know
		      \begin{align*}
			\E\int_0^T\int_Z|\xi(s,z)|_E\,\nu(\d z)\d s=\E\int_0^T\int_Z|\xi(s,z)|_E\,N(\d s,\d z),
			\end{align*}
	   from which deduce  that whenever one side of the equality makes sense, so does the other. Hence if we suppose either
		    $ \E\int_0^T\int_Z|\xi(s,z)|_E\nu(\d z)\d s<\infty$ or $\E\int_0^T\int_Z|\xi(s,z)|_EN(\d z,\d s)<\infty$,	
			then both integrals $\int_0^T\int_Z\xi(s,z)\nu(\d z)\d s$ and $\int_0^T\int_Z\xi(s,z)N(\d z,\d s)$ are well defined as Lebesgue-Bochner integrals, $\mathbb{P}$-a.s.

			   Define
			\begin{align*}
				\mathcal{L}_2&=\{\xi:[0,T]\times\Omega\times Z\rightarrow E,\big{|}\,\xi\,\text{is }\mathcal{B}([0,T])\otimes\mathcal{F}_T\otimes\mathcal{Z}\text{-measurable and }
				\E\int_0^T\int_Z|\xi(s,z)|_E\,\nu(\d z)\d s<\infty\},     \\
				\mathcal{K}_2&=\{\xi:[0,T]\times\Omega\times Z\rightarrow E\,\big{|}\,
			 \xi\in\mathcal{L}_2\,
				\text{and }\eqref{sec-2-eq-13} \text{ holds}\}.
			\end{align*}
			We can repeat the argument before to show that every function $\xi$ of the form \eqref{sec-2-eq-31} is in $\mathcal{L}_2$ and  satisfies \eqref{sec-2-eq-31}.
		   Now suppose that $\{\xi^n\}_{n\in\mathbb{N}}\subset \mathcal{K}_2$, $\xi\in\mathcal{L}_2$, and $|\xi^n-\xi|_E$ decreases to $0$ as $n\rightarrow\infty$. Then according to the monotone convergence theorem, we obtain
			\begin{align*}
				   \lim_{n\rightarrow\infty} \E\int_0^T\int_Z|\xi^n(s,z)-\xi(s,z)|_E\,N(\d s,\d z)=0  \text{ and }\lim_{n\rightarrow\infty} \E\int_0^T\int_Z|\xi^n(s,z)-\xi(s,z)|_E\,\nu(\d z)\d s=0
				\end{align*}
  Since $\xi^n$ satisfies \eqref{sec-2-eq-31}, we infer
\begin{align*}
	       \E\int_0^T\int_Z\xi(s,z)\,N(\d s,\d z)&=\lim_{n\rightarrow\infty}\E\int_0^T\int_Z\xi^n(s,z)\,N(\d s,\d z)\\
	&= \lim_{n\rightarrow\infty}\E\int_0^T\int_Z\xi^n(s,z)\,\nu(\d z)\d s= \E\int_0^T\int_Z\xi(s,z)\,\nu(\d z)\d s.
	\end{align*}  		
			Again by Lemma \ref{sec-2-main-lem}, $\mathcal{K}_2$ contains all $\mathcal{P}\otimes\mathcal{Z}$-measurable functions in $\mathcal{L}_2$.
	\end{proof}
	%-----------------------------------------------------------
	
	\begin{prop} Let $E$ be a Banach space. Suppose that $\xi:[0,T]\times\Omega\times Z\rightarrow E$ is a $\mathcal{B}([0,T])\otimes\mathcal{F}_T\otimes\mathcal{Z}$-measurable function such that $\E\int_0^T\int_Z|\xi(s,z)|_EN(\d s,\d z)<\infty$. Then
		\begin{align}\label{A.prop-eq-2}
			 \int_0^T\int_Z\xi(s,\cdot,z) \,N(\d s,\d z)(\omega)=\sum_{s\in(0,T]\cap\mathcal{D}(\pi)}\xi(s,\omega,\pi(s,\omega))
		\end{align}
	 \end{prop}
  \begin{proof}
	   Since $\xi$ is a $\mathcal{B}([0,T])\otimes\mathcal{F}_T\otimes\mathcal{Z}$-measurable function and $\E\int_0^T\int_Z|\xi(s,z)|_EN(\d s,\d z)<\infty$, there exists a set $\hat{\Omega}\subset \Omega$ with probability $1$ such that for every $\omega\in\hat{\Omega}$, $\xi(\cdot,\omega,\cdot)$ is $\mathcal{B}([0,T])\otimes\mathcal{Z}$-measurable and $\int_0^T\int_Z|\xi(s,\cdot,z)|_EN(\d s,\d z)(\omega)<\infty$. This implies that $\xi(\cdot,\omega,\cdot)$ is Bochner integrable with respect to $N$ $\mathbb{P}$-a.s. Given $\omega\in\Omega$, we can always find a sequence of functions $\{\xi^n\}$ of the form
	\begin{align}\label{A.prop-2-eq-1}
		        \sum_{i=1}^nx_i1_{A_i},\ \ A_i\in \mathcal{B}([0,T])\otimes\mathcal{Z}
		\end{align}
		such that $|\xi^n(t,\omega,z)-\xi(t,z)|_E$ decreases to $0$ as $n\rightarrow\infty$ for every $(t,z)\in[0,T]\times Z$. So it's enough to verify \eqref{A.prop-eq-2} for the function $\xi$ of the form \eqref{A.prop-2-eq-1}. For this, observe that
		\begin{align*}
			    \int_0^t\int_Z\xi(s,\cdot,z)\,N(\d s,\d z)(\omega)&= \sum_{i=1}^nx_iN(A_i\cap(0,t]\times Z)(\omega)= \sum_{i=1}^nx_i\sum_{s\in(0,t]\cap\mathcal{D}(\pi)}1_{A_i}(\pi(s,\omega))\\
			&= \sum_{s\in(0,t]\cap\mathcal{D}(\pi)}\sum_{i=1}^nx_i1_{A_i}(\pi(s,\omega))=\sum_{s\in(0,t]\cap\mathcal{D}(\pi)}\xi(s,\omega,\pi(\omega)).
			\end{align*}
	\end{proof}
	
	%-------------------------------------------
\begin{prop}\label{A.prop-3} Let $E$ be a martingale type $r$ Banach space.
 If
	$\xi\in\mathcal{M}^1_T(\mathcal{P}\otimes\mathcal{Z},\d t\times\mathbb{P}\times\nu;E)\cap \mathcal{M}^r_T(\mathcal{P}\otimes\mathcal{Z},\d t\times\mathbb{P}\times\nu;E)$, then  for each $t\geq0$, $\mathbb{P}$-a.s.
	\begin{align}\label{A.prop-3-1}
	     \int_0^t\int_Z\xi(s,\cdot,z)\,\tilde{N}(ds,dz)&=\int_0^t\int_Z \xi(s,\cdot,z)\, N(ds,dz)-\int_0^t\int_Z\xi(s,\cdot,z)\,\nu(dz)ds\nonumber\\
	   & = \sum_{s\in(0,t]\cap
	     \mathcal{D}(\pi)}\xi(s,\cdot,\pi(s))-\int_0^t\int_Z\xi(s,\cdot,z)\,\nu(dz)ds.
	\end{align}
	Here the integral $\int_0^t\int_Z\xi(s,\cdot,z)\,\tilde{N}(ds,dz)$ on the left side is understood as the stochastic integral taking values in $E$.
	\end{prop}

	\begin{proof}
	 The proof could be done in the same manner as earlier in the proof of Proposition \ref{A.prop-2}. First the equality \eqref{A.prop-3-1} can be easily verified for functions of simple structures \eqref{sec-2-eq-31}.
	Next, an approximating step introduced in Lemma \ref{sec-2-main-lem} allows us to extend the equality to a general $\mathfrak{F}$-predictable process in $\mathcal{M}^1_T(\mathcal{P}\otimes\mathcal{Z},\d t\times\mathbb{P}\times\nu;E)$. To do this, suppose that $\{\xi_n\}_{n\in\mathbb{N}}$ is a sequence of functions in $ \mathcal{M}^1_T(\mathcal{P}\otimes\mathcal{Z},\d t\times\mathbb{P}\times\nu;E)\cap\mathcal{M}^r_T(\mathcal{P}\otimes\mathcal{Z},\d t\times\mathbb{P}\times\nu;E)$ such that for every $n\in\mathbb{N}$, $\xi^n$ satisfies \eqref{A.prop-3-1} and $|\xi^n-\xi|_E$ decreases to $0$, as $n\rightarrow\infty$ on $[0,T]\times\Omega\times Z$. Hence we have
	\begin{align*}
		&   \lim_{n\rightarrow\infty} \E\int_0^T\int_Z|\xi^n(s,z)-\xi(s,z)|_E\,\nu(\d z)\d s=0,\\
		&   \lim_{n\rightarrow\infty} \E\int_0^T\int_Z|\xi^n(s,z)-\xi(s,z)|_E\,N(\d s,\d z)=0.
	\end{align*}
	By \eqref{sec-2-eq-10}, we may find a subsequence $\xi^{n_k}$ such that for every $t\in[0,T]$,
	\begin{align*}
   \int_0^t\int_Z\xi(s,\cdot,z)\,\tilde{N}(\d s,\d z)& =\lim_{k\rightarrow\infty}\int_0^t\int_Z\xi^{n_k}(s,\cdot,z)\,\tilde{N}(\d s,\d z)\\
                                            &= \lim_{k\rightarrow\infty}\Big[\int_0^t\int_Z\xi^{n_k}(s,\cdot,z)\,N(\d s,\d z)-\int_0^t\int_Z\xi^{n_k}(s,\cdot,z)\,\nu(dz)ds\Big]\\
&= \int_0^t\int_Z\xi(s,\cdot,z)\,N(\d s,\d z)-\int_0^t\int_Z\xi(s,\cdot,z)\,\nu(dz)ds, \ \ \mathbb{P}\text{-a.s.}
	\end{align*}
	\end{proof}
	
	\begin{rem}
	Proposition \ref{A.prop-3} is still valid under the weaker constrain on $E$ and $\xi$. Namely, we only assume that $E$ is a Banach space and $\xi \in \mathcal{M}^1_T(\mathcal{P}\otimes\mathcal{Z},\d t\times\mathbb{P}\times\nu;E)$. Then $E$ is of martingale type $1$, and we can still define the stochastic integral $\int_0^t\int_Z\xi(s,\cdot,z)\tilde{N}(ds,dz)$ which is an $L^1$-integrable martingale and satisfies
	\begin{align*}
	    \E\Big{|}\int_0^t\int_Z\xi(s,\cdot,z)\,\tilde{N}(\d s,\d z)\Big{|}_E\leq C\E \int_0^t\int_Z|\xi(s,\cdot,z)|_E\,\nu(\d z)\d s.
	\end{align*}
	In this case, the stochastic integrable coincides with the Bochner integral almost everywhere, that is \eqref{A.prop-3-1} holds.
	It is worth noting that \eqref{A.prop-3-1} is only satisfied when the function $\xi$ is predictable (specifically $\mathcal{P}\otimes\mathcal{Z}$-measurable).
	\end{rem}
%---------------------Appendix B------------------

\section{It\^o's  formula}\label{app-Ito}

Let $E$ be a martingale of type $r$ ($1<r\leq 2$) Banach space.
In this section, we will study  It\^{o}'s formula for the processes of the following type:
\begin{align}\label{Ito-process-B}
X_t=X_0+\int_0^t a(s)\,\d s+\int_0^tg(s)\,\d W_s+\int_0^t\int_Z \xi(s,z)\,\tilde{N}(\d s,\d z)+\int_0^t\int_Z \eta(s,z)\,N(\d s,\d z),\ t\geq0.
\end{align}
Here $a :\R_+ \times \Omega \to E $ is an $E$-valued progressively measurable process \dela{on the space $(\R_+\times\Omega,\mathcal{B}(\R_+)\otimes\mathcal{F})$} such that for all $t\geq0$, $\int_0^t\|a(s,\omega)\|ds<\infty$, $\mathbb{P}$-a.s., $N$ is a Poisson random measure associated with a Poisson point process $\pi$,
 $\xi\in\mathcal{M}^p_{loc}(\hat{\mathcal{P}},\nu;E)$, $\eta\in\mathcal{M}_{loc}(\hat{\mathcal{P}},N;E)$, and $g=(g(t))_{t\in[0,T]}$ is a process in $M([0,T];\gamma(H;E))$. Assume that $|\xi(t,\omega,z)|_E|\eta(t,\omega,z)|_E=0$ for all $(t,\omega,z)\in\R_+\times\Omega\times Z$. Without loss of generality, we may assume that the process $X$ is right-continuous with left-limits.
 %---------
 %---------
 In order to make the paper self-contained, we formulate the following auxiliary lemma.%--------------------------
\begin{lem}\label{lem-holder}{\bf{(see \cite{[Ana]})}}
  Let $E$ and $F$ be separable Banach spaces and $\phi:E\rightarrow F$ be a Fr\'{e}chet differentiable function whose  Fr\'{e}chet  derivative $\phi^\prime: E\rightarrow L(E;F)$ is locally $\alpha$-H\"{o}lder continuous.  In other words, for all $r>0$, there exists $H=H(r)<\infty$ such that
  \begin{align*}
  \sup_{|x|_E,|y|_E\leq r,x\neq y}\frac{\|\phi'(x)-\phi'(y)\|_{L(E;F)}}{|x-y|^{\alpha}_E}\leq H(r).
  \end{align*}
  Then we have that  for every $x,y\in E$,
  \begin{align}
         \phi(y)=\phi(x)+\phi^\prime(x)(y-x)+R_{\phi}(x,y),
  \end{align}
  where $ R_{\phi}(x,y)=\int_0^1(\phi^\prime(x+\theta (y-x))(y-x)-\phi^\prime(x)(y-x)) d\theta $.
   Moreover,  for every $r>0$, there exists $C=C(r)>0$ such that
   \begin{equation*}|R_{\phi}(x,y)|_F\leq C|y-x|_E^{\alpha+1} \mbox{ for all } |x|_E,|y|_E\leq r.
   \end{equation*}
\end{lem}

\begin{thm}\label{theo-Ito-1}{\bf{(It\^o's  formula 1)}}  Assume that $E$ is a martingale type $2$ Banach space. Let $X$ be a process given by \eqref{Ito-process-B}. Let $\phi:E\rightarrow E$ be twice Fr\'{e}chet differentiable and its first and second derivative are all continuous. Then for every $t>0$, we have $\mathbb{P}$-a.s.
\begin{align}
\begin{split}\label{Ito-formula-1}
       \phi(X_t)&=\phi(X_0)+\int_0^t \phi^\prime(X_{s}) (a(s))\,\d s+\int_0^t\phi^\prime(X_s)(g(s))\,\d W_s+\frac12\int_0^t \text{tr }\phi^{\prime\prime}(X_s)(g(s),g(s))\,\d s\\
       &\hspace{1cm}+  \int_0^t\int_Z \Big{[}\phi( X_{s-}+\eta(s,z))-\phi(X_{s-})\Big{]}\, N(\d s,\d z)\\
       &\hspace{1cm}+\int_0^t\int_Z \Big{[}\phi( X_{s-}+\xi(s,z))-\phi(X_{s-})\Big{]} \,\tilde{N}(\d s,\d z)\\
       &\hspace{1cm}+\int_0^t\int_Z \Big{[}\phi( X_{s-}+\xi(s,z))-\phi(X_{s-})-\phi^\prime(X_{s-})(\xi(s,z))\Big{]}\,\nu(\d z)\d s.
      \end{split}
\end{align}
\end{thm}

\begin{thm}\label{theo-Ito-3}{\bf{(It\^o's  formula 2)}}
Assume that $E$ is a martingale type $r$ Banach space, $r\in(1,2]$.
Let $\phi:E\rightarrow E$ be a function of class $C^1$ such that the first derivative $\phi^\prime:E\rightarrow L(E;E)$ is locally $(r-1)$-H\"{o}lder continuous. Let $X$ be a non-Gaussian L\'{e}vy process given by
\begin{align}\label{Ito-process-2}
X_t=X_0+\int_0^t a(s)\,\d s+\int_0^t\int_Z \xi(s,z)\,\tilde{N}(\d s,\d z)+\int_0^t\int_Z \eta(s,z)\,N(\d s,\d z),\ t\geq0.
\end{align}
Then for every $t>0$, we have $\mathbb{P}$-a.s.
\begin{align}
\begin{split}\label{Ito-formula-3}
       \phi(X_t)&=\phi(X_0)+\int_0^t \phi^\prime(X_{s}) (a(s))\,\d s+
       \int_0^t\int_Z \Big{[}\phi( X_{s-}+\eta(s,z))-\phi(X_{s-})\Big{]} \,N(\d s,\d z)\\
       &\hspace{1cm}+\int_0^t\int_Z \Big{[}\phi( X_{s-}+\xi(s,z))-\phi(X_{s-})\Big{]}\, \tilde{N}(\d s,\d z)\\
       &\hspace{1cm}+\int_0^t\int_Z \Big{[}\phi( X_{s-}+\xi(s,z))-\phi(X_{s-})-\phi^\prime(X_{s-})(\xi(s,z))\Big{]}\,\nu(\d z)\d s.
      \end{split}
\end{align}
\end{thm}

%-------------------------------

\begin{proof}[Proof of Theorem \ref{theo-Ito-1}]

Without loss of generality, we may assume that the process $X$ is bounded, namely, there exists $r>0$ such that
\begin{align}\label{eq-135}
         \sup_{0\leq s\leq t}|X_s|_E\leq r.
\end{align}
 Then we can relax the boundedness assumption \eqref{eq-135} by the usual localization argument.
        Since the Poisson point process $\pi$ is $\sigma$-finite, there exists a sequence of sets $\{D_n\}_{n\in\mathbb{N}}$ such that $D_n\subset D_{n+1}$, $\cup_{n\in\mathbb{N}}D_n=Z$, and $\E N(t,D_n)<\infty$ for every $0<t<\infty$ and $n\in\mathbb{N}$. Define a sequence $\{\xi^n\}_{n\in\N}$ of functions by
       \begin{align*}
           \xi^n(s,\omega,z):=\xi(s,\omega,z)1_{D_n}(z),\ (s,\omega,z)\in\mathbb{R}_+\times\Omega\times Z,\  n\in\N.
       \end{align*}
 Since $|\xi^n|_E\leq |\xi|_E$ and by the assumption, $\xi\in\mathcal{M}^r_T(\mathcal{P}\otimes\mathcal{Z},\d t\times\mathbb{P}\times\nu;E)$, we infer that $\xi^n\in\mathcal{M}^r_T(\mathcal{P}\otimes\mathcal{Z},\d t\times\mathbb{P}\times\nu;E)$. By the definition of stochastic integrals, we have
 \begin{align*}
         \int_0^t\int_Z\xi^n(s,z)\,\tilde{N}(\d s,\d z)&= \int_0^t\int_Z1_{D_n}\xi(s,z)\,\tilde{N}(\d s,\d z)=\int_0^t\int_{D_n}\xi(s,z)\,\tilde{N}(\d s,\d z).
 \end{align*}
 Now applying Proposition \ref{A.prop-3} yields that
        \begin{align}\label{eq661}
        \int_0^t\int_Z\xi^n(s,z)\,\tilde{N}(\d s,\d z)=
        \sum_{s\in(0,t]\cap
     \mathcal{D}(\pi)}\xi(s,\pi(s))1_{D_n}(\pi(s))-\int_0^t\int_{D_n}\xi(s,z)\,\nu(\d z)\d s.
        \end{align}
        Similarly, we define a sequence $\{g^n\}_{n\in\N}$ of functions by
               \begin{align*}
           \eta^n(s,\omega,z)=\eta(s,\omega,z)1_{D_n}(z),\ \ \ (s,\omega,z)\in\mathbb{R}_+\times\Omega\times Z,\ n\in\N.
       \end{align*}
   It follows that
\begin{align}\label{eq660}
        \int_0^t\int_Z\eta^n(s,z)N(\d s,\d z)=\sum_{s\in(0,t]\cap
     \mathcal{D}(\pi)}\eta(s,\pi(s))1_{D_n}(\pi(s)).
             \end{align}
     Since $\E N(t,D_n)<\infty$, $t\geq0$,
        we see that for almost every $\omega\in\Omega$, the set $\{s\leq t:\pi(s,\omega)\in D_n\cap\mathcal{D}(\pi)\}$ contains only finitely many points in each time interval $(0,t]$, for $t>0$.
 Hence we may denote these points according to their magnitude by $0=\tau_0(\omega)<\tau_1(\omega)<\tau_2(\omega)<\cdots<\tau_m(\omega)<\cdots$.  In other words, we put
 \begin{align*}
 \tau_0&=0;\\
    \tau_m&=\inf\{s\in(0,t]\cap\mathcal{D}(\pi): \pi(s)\in D_n; s>\tau_{m-1}\},\ m\geq1.
 \end{align*}
         The random times $\tau_1,\tau_2,\dots$ form a random configuration of points in $(0,t]$ with $\pi(\tau_i)\in D_n$ and for each $m$, the random time $\tau_m$ is a stopping time. Indeed, for every $u>0$, we find that
         \begin{align*}
                \{\tau_m\leq u\}=\{N(u,D_n)\geq m\}\in\mathcal{F}_u.
         \end{align*}
             Let us define a sequence $\{X^n\}_{n\in\N}$ of process $X^n:=(X^n_t)_{t\geq0}$ by
             \begin{align*}
                  X^n_t=X_0^t+\int_0^ta(s)\d s+\int_0^tg(s)\d W_s+\int_0^t\int_Z \xi^n(s,z)\tilde{N}(\d s,\d z)+\int_0^t\int_Z \eta^n(s,z)N(\d s,\d z),\ t\geq0, \ n\in\mathbb{N}.
             \end{align*}
It follows from \eqref{eq661} and \eqref{eq660} that for every $n\in\N$ and all $t\geq0$, $\mathbb{P}$-a.s.
                \begin{align}
                \begin{split}\label{interlacing-pr}
                  X^n_t
                  %------------
                  \dela{
                  &=X_0+\int_0^ta(s)\d s+\sum_{s\in(0,t]\cap
     \mathcal{D}(\pi)}\xi(s,\pi(s))1_{D_n}(\pi(s))-\int_0^t\int_{D_n}\xi(s,z)\nu(\d z)\d s\\
     &\hspace{1cm}+ \sum_{s\in(0,t]\cap
     \mathcal{D}(\pi)}\eta(s,\pi(s))1_{D_n}(\pi(s))\\}
     %------------
     &= X_0+\int_0^t a(s)\d s+\int_0^tg(s)\d W_s-\int_0^t\int_Z\xi^n(s,z)\nu(\d z)\d s+\sum_m\xi^n(\tau_m,\pi(\tau_m),\cdot)1_{\{\tau_m\leq t\}}\\
     &\hspace{1cm}+\sum_m \eta^n(\tau_m,\pi(\tau_m),\cdot)1_{\{\tau_m\leq t\}}.
     \end{split}
             \end{align}
    Note that
    \begin{align*}
         \phi(X_t^n)-\phi(X_0)&=\sum_m\Big{[} \phi(X^n_{t\wedge\tau_{m}})-\phi(X^n_{t\wedge\tau_{m-1}}) \Big{]}\\
         &=\sum_m\Big{[} \phi(X^n_{t\wedge\tau_{m}})-\phi(X^n_{t\wedge\tau_{m}-}) \Big{]}+\sum_m\Big{[} \phi(X^n_{t\wedge\tau_{m}-})-\phi(X^n_{t\wedge\tau_{m-1}}) \Big{]}\\
         &=: I_1+I_2.
    \end{align*}
 Here $X^n_{t\wedge\tau_{m}}=(X^n)_t^{\tau_m}$, $t\geq0$ is the process $X^n$ stopped at time $\tau_m$, and $X^n_{t\wedge\tau_m-}=(\bar{X}^n)_{t}^{\tau_m}$, $t\geq0$ is the process $X^n$ stopped strictly before time $\tau_m$.  Namely,
               \begin{eqnarray*}
X^n_{t\wedge\tau_{m}}(\omega)=(X^n)_t^{\tau_m}(\omega)=\left\{
                    \begin{array}{cc}
                             X^n_t(\omega) & if\ t\leq \tau_m(\omega),\\
                       X^n_{\tau_m(\omega)}(\omega) & if\ t\geq \tau_m(\omega).
                   \end{array}
                           \right.
                        \end{eqnarray*}
                        and
               \begin{eqnarray*}
                    X^n_{t\wedge\tau_m-}(\omega)=   (\bar{X}^n)_t^{\tau_m}(\omega)=\left\{
                    \begin{array}{cc}
                             X^n_t(\omega) & if\ t< \tau_m(\omega)\\
                       X^n_{\tau_m(\omega)-}(\omega) & if\ t\geq \tau_m(\omega).
                   \end{array}
                           \right.
                        \end{eqnarray*}
 Note that the jumps of $X^n$ occur only at times $\{\tau_m\}$. So $X^n_{t\wedge\tau_m}\neq X^n_{t\wedge\tau_m-}$ if and only if $$\xi^n(\tau_m,\pi(\tau_m),\cdot)1_{\{\tau_m\leq t\}}+ \eta^n(\tau_m,\pi(\tau_m),\cdot)1_{\{\tau_m\leq t\}}\neq 0.$$  Since by assumption $|f|_E\cdot|g|_E=0$, we infer that $X^n_{t\wedge\tau_m}\neq X^n_{t\wedge\tau_m-}$ if and only if
             \begin{align*}
                      \xi^n(\tau_m,\pi(\tau_m),\cdot)1_{\{\tau_m\leq t\}}\neq 0\ \text{and } \eta^n(\tau_m,\pi(\tau_m),\cdot)1_{\{\tau_m\leq t\}}=0
             \end{align*}
             or
                          \begin{align*}
                      \xi^n(\tau_m,\pi(\tau_m),\cdot)1_{\{\tau_m\leq t\}}= 0\ \text{and } \eta^n(\tau_m,\pi(\tau_m),\cdot)1_{\{\tau_m\leq t\}}\neq 0
             \end{align*}
             Hence
             \begin{align*}
                   X^n_{t\wedge\tau_m}&= X^n_{t\wedge\tau_m-}+\xi^n(\tau_m,\pi(\tau_m))1_{\{\tau_m\leq t\}}+ \eta^n(\tau_m,\pi(\tau_m))1_{\{\tau_m\leq t\}}\\
                 &  =
                   \left\{
                    \begin{array}{cc}
                             X^n_{t\wedge\tau_m-}+ \xi^n(\tau_m,\pi(\tau_m))1_{\{\tau_m\leq t\}} & if\  \xi^n(\tau_m,\pi(\tau_m))\neq 0,\ \eta^n(\tau_m,\pi(\tau_m))=0,\\
                       X^n_{t\wedge\tau_m-}+ \eta^n(\tau_m,\pi(\tau_m))1_{\{\tau_m\leq t\}} & if\  \xi^n(\tau_m,\pi(\tau_m))= 0,\ \eta^n(\tau_m,\pi(\tau_m))\neq0.
                   \end{array}
                   \right.
             \end{align*}
It follows that
             \begin{align*}
                   I_1&=\sum_m\Big{[} \phi(X^n_{t\wedge\tau_{m}})-\phi(X^n_{t\wedge\tau_{m}-}) \Big{]}\\
                   & =\sum_m\Big{[} \phi(X^n_{t\wedge\tau_{m}})-\phi(X^n_{t\wedge\tau_{m}-}) \Big{]}1_{\{\|\xi(\tau_m,\pi(\tau_m))\|\neq0\}\cap \{\|\eta(\tau_m,\pi(\tau_m))\|=0\}}1_{\{\tau_m\leq t\}}\\
                   &+ \sum_m\Big{[} \phi(X^n_{t\wedge\tau_{m}})-\phi(X^n_{t\wedge\tau_{m}-}) \Big{]}1_{\{\|\xi(\tau_m,\pi(\tau_m))\|=0\}\cap\{\|\eta(\tau_m,\pi(\tau_m))\|\neq0\}}1_{\{\tau_m\leq t\}}\\
                                      & =\sum_m\Big{[} \phi(X^n_{t\wedge\tau_m-}+ \xi^n(\tau_m,\pi(\tau_m)))-\phi(X^n_{t\wedge\tau_{m}-}) \Big{]}1_{\{\|\xi(\tau_m,\pi(\tau_m))\|\neq0\}\cap \{\|\eta(\tau_m,\pi(\tau_m))\|=0\}}1_{\{\tau_m\leq t\}}\\
                   &+ \sum_m\Big{[} \phi(X^n_{t\wedge\tau_m-}+ \eta^n(\tau_m,\pi(\tau_m)))-\phi(X^n_{t\wedge\tau_{m}-}) \Big{]}1_{\{\|\xi(\tau_m,\pi(\tau_m))\|=0\}\cap\{\|\eta(\tau_m,\pi(\tau_m))\|\neq0\}}1_{\{\tau_m\leq t\}}\\
                   &=\int_0^t\int_Z\Big{[} \phi(X^n_{s-}+\xi^n(s,z,\omega)) -\phi(X^n_{s-})    \Big{]}N(ds,dz)\\
                   &+\int_0^t\int_Z\Big{[} \phi(X^n_{s-}+\eta^n(s,z,\omega)) -\phi(X^n_{s-})    \Big{]}N(ds,dz)\\
                    &=\int_0^t\int_Z\Big{[} \phi(X^n_{s-}+\xi^n(s,z,\omega)) -\phi(X^n_{s-})    \Big{]}\tilde{N}(ds,dz)\\
                   &+\int_0^t\int_Z\Big{[} \phi(X^n_{s-}+\eta^n(s,z,\omega)) -\phi(X^n_{s-})    \Big{]}N(ds,dz)\\
                   & +\int_0^t\int_Z\Big{[} \phi(X^n_{s-}+\xi^n(s,z,\omega)) -\phi(X^n_{s-})    \Big{]}\nu(dz)ds
             \end{align*}

    Notice that there are no jumps of $X$ in the random time interval $(t\wedge\tau_{m-1},t\wedge\tau_m)$, in other words $X$ contains only the continuous components. Set for $u\in[0,T]$,
		                 \begin{align}\label{stopped-I_2}
			Y^{n,m}_t
			  =X^n_{\tau_{m-1}}+\int_0^t1_{[\tau_{m-1},T]}\Big[a(s)-\int_{Z}\xi^n(s,z)\,\nu(\d z)\Big]\d s+\int_0^t
			 1_{[\tau_{m-1},T]} g(s)\,\d W_s.   	
		                 \end{align}
		                 Then we have $Y^{n,m}_{t}=X^n_{\tau_{m-1}}$ for $t\in[0,\tau_{m-1}]$, $Y_t^{n,m}=X_t^n$ for $t\in[\tau_{m-1},\tau_m)$, and $Y_{\tau_m}^{n,m}=Y_{\tau_{m}-}^{n,m}=X_{\tau_{m}-}$.		                 By applying the It\^{o} formula of Gaussian processes in \cite{[Neid]} to the process $Y^{n,m}_u$, we obtain
		                                 \begin{align*}
\phi(Y^{n,m}_{t\wedge\tau_{m}})=&\phi(X^n_{\tau_{m-1}})-\int_{\tau_{m-1}}^{t\wedge\tau_{m}}\phi^\prime(Y^{n,m}_s)(a(s))\d s-\int_{\tau_{m-1}}^{t\wedge\tau_{m}}\phi^\prime(Y^{n,m}_s)(\xi^n(s,z))\nu(\d z)\d s\\
    &+\int_{\tau_{m-1}}^{t\wedge\tau_{m}}\phi^\prime(Y^{n,m}_s)(g(s))\d W_s
    +\frac12\int_{\tau_{m-1}}^{t\wedge\tau_{m}} \text{tr }\phi^{\prime\prime}(Y^{n,m}_s)(g(s),g(s))\d s,\text{  for }t\in[\tau_{m-1},\tau_m]\;\;\mathbb{P}\text{-a.s.}
     \end{align*}
     Moreover, when $t<\tau_{m-1}$, we have $$\phi(Y^{n,m}_{t\wedge\tau_{m}})-\phi(X^n_{\tau_{m-1}})=0=\phi(X^n_{t\wedge\tau_{m}-})-\phi(X^n_{t\wedge\tau_{m-1}});$$  when $\tau_{m-1}\leq t<\tau_m$, $$\phi(Y^{n,m}_{t\wedge\tau_{m}})-\phi(X^n_{\tau_{m-1}})=\phi(Y^{n,m}_t)-\phi(X^n_{\tau_{m-1}})=\phi(X^n_t)-\phi(X^n_{\tau_{m-1}})=\phi(X^n_{t\wedge\tau_{m}-})-\phi(X^n_{t\wedge\tau_{m-1}});$$
     when $t\geq\tau_m$, we have
     $$\phi(Y^{n,m}_{t\wedge\tau_{m}})-\phi(X^n_{\tau_{m-1}})=\phi(Y_{\tau_m}^{n,m})-\phi(X^n_{\tau_{m-1}})=\phi(X^n_{\tau_m-})-\phi(X^n_{\tau_{m-1}})=\phi(X^n_{t\wedge\tau_{m}-})-\phi(X^n_{t\wedge\tau_{m-1}}).$$
     Hence,
		                      \begin{align*}
\phi(X^n_{t\wedge\tau_{m}-})-\phi(X^n_{t\wedge\tau_{m-1}})=&\int_{\tau_{m-1}}^{t\wedge\tau_{m}}\phi^\prime(X^n_s)(a(s))\d s-\int_{\tau_{m-1}}^{t\wedge\tau_{m}}\phi^\prime(X^n_s)(\xi^n(s,z))\nu(\d z)\d s\\
    &+\int_{\tau_{m-1}}^{t\wedge\tau_{m}}\phi^\prime(X^n_s)(g(s))\d W_s
    +\frac12\int_{\tau_{m-1}}^{t\wedge\tau_{m}} \text{tr }\phi^{\prime\prime}(X^n_s)(g(s),g(s))\d s.
     \end{align*}
     Therefore, we have
     \begin{align}
     \begin{split}\label{eq-132}
    I_2= \sum_m\Big{[} \phi(X^n_{t\wedge\tau_{m}-})-\phi(X^n_{t\wedge\tau_{m-1}}) \Big{]}=&\int_{0}^{t}\phi^\prime(X^n_s)(a(s))\,\d s-\int_0^t\phi^\prime(X^n_s)(\xi^n(s,z))\,\nu(\d z)\d s\\
    &+\int_{0}^{t}\phi^\prime(X^n_s)(g(s))\,\d W_s
    +\frac12\int_0^t \text{tr }\phi^{\prime\prime}(X^n_s)(g(s),g(s))\,\d s.
    \end{split}
     \end{align}

 %-------------------------------------
                         Combining $I_1$ and $I_2$ yields that
                  \begin{align}
                  \begin{split}\label{ito-formula-proof-eq1}
                   \phi(X_t^n)-\phi(X_0)
                   &=\int_0^t \phi^\prime(X^n_s)(a(s))\,ds+\int_{0}^{t}\phi^\prime(X^n_s)(g(s))\,\d W_s
    +\frac12\int_0^t \text{tr }\phi^{\prime\prime}(X^n_s)(g(s),g(s))\,\d s\\
    &+\int_0^t\int_Z\Big{[} \phi(X^n_{s-}+\xi^n(s,z,\omega)) -\phi(X^n_{s-})    \Big{]}\tilde{N}(ds,dz)\\
                   &+\int_0^t\int_Z\Big{[} \phi(X^n_{s-}+\eta^n(s,z,\omega)) -\phi(X^n_{s-})    \Big{]}N(ds,dz)\\
                   & +\int_0^t\int_Z\Big{[} \phi(X^n_{s-}+\xi^n(s,z,\omega)) -\phi(X^n_{s-})- \phi^\prime(X^n_{s-})(\xi^n(s,z))    \Big{]}\nu(dz)ds.
                   \end{split}
                  \end{align}
                 This shows that It\^{o} formula \eqref{Ito-formula-1} holds for the process $X^n$.
    Now let us consider the general case. On the basis of the inequality \eqref{Lr-1}  and the Lebesgue dominated convergence theorem, we infer
    \begin{align*}
      &  \lim_{n\rightarrow\infty} \E\left| \int_0^t\int_Z \xi^n(s,z)\tilde{N}(ds,dz)-\int_0^t\int_Z \xi(s,z)\tilde{N}(ds,dz)   \right|_E^2\\
        &\leq C\lim_{n\rightarrow\infty} \E\int_0^t\int_Z |\xi^n(s,z)-\xi(s,z)|_E^2\;\nu(dz)ds=0.
    \end{align*}
   This allows  us to find a subsequence such that $\int_0^t\int_Z \xi^n(s,z)\tilde{N}(ds,dz) $ uniformly converges to $\int_0^t\int_Z \xi(s,z)\tilde{N}(ds,dz) $ on any finite interval $[0,T]$ $\mathbb{P}$-a.s. Similarly, we can prove that $\int_0^t\int_Z \eta^n(s,\omega,z) N(ds,dz) $ uniformly converges to $\int_0^t\int_Z \eta(s,\omega,z) N(ds,dz) $ on $[0,T]$ $\mathbb{P}$-a.s. as well. Hence we infer that $X^n_s$ converges uniformly to $X_{s}$ as $n\rightarrow\infty$, $\mathbb{P}$-a.s. on $[0,T]$. Also, $X^n_{s-}$ converges uniformly to $X_{s-}$ $\mathbb{P}$-a.s. on $[0,T]$ as $n\rightarrow\infty$. Hence by the continuity of $\phi$, $\phi^\prime$, and $\phi^{\prime\prime}$, we infer that
   \begin{align}
   \begin{split}\label{convergence-eq}
       &\int_0^t\phi^\prime(X_s^n)(a(s))\,\d s\rightarrow \int_0^t\phi^\prime(X_s)(a(s))\,\d s,\hspace{1cm}\mathbb{P}\text{-a.s.},\\
       &\int_0^t\phi^\prime(X_s^n)(g(s))\,\d W_s\rightarrow \int_0^t\phi^\prime(X_s)(g(s))\,\d W_s\hspace{1cm}\text{in }L_2,\\
       &\int_0^t\int_Z\Big{[} \phi(X^n_{s-}+\xi^n(s,z,\omega)) -\phi(X^n_{s-})    \Big{]}\tilde{N}(ds,dz)\\
       &\longrightarrow \int_0^t\int_Z\Big{[} \phi(X_{s-}+\xi(s,z,\omega)) -\phi(X_{s-})    \Big{]}\tilde{N}(ds,dz)\hspace{1cm}\text{in }L_2,\\
      & \int_0^t\int_Z\Big{[} \phi(X^n_{s-}+\eta^n(s,z,\omega)) -\phi(X^n_{s-})    \Big{]}N(ds,dz)\\
      &\longrightarrow \int_0^t\int_Z\Big{[} \phi(X_{s-}+\eta(s,z,\omega)) -\phi(X_{s-})    \Big{]}N(ds,dz),\hspace{1cm}\mathbb{P}\text{-a.s.}\\
      &\int_0^t\int_Z\Big{[} \phi(X^n_{s-}+\xi^n(s,z,\omega)) -\phi(X^n_{s-})- \phi^\prime(X^n_{s-})(\xi^n(s,z))    \Big{]}\nu(dz)ds,
\\
&\longrightarrow \int_0^t\int_Z\Big{[} \phi(X_{s-}+\xi(s,z,\omega)) -\phi(X_{s-})- \phi^\prime(X_{s-})(\xi(s,z))    \Big{]}\nu(dz)ds,\hspace{1cm}\mathbb{P}\text{-a.s.}
\end{split}
   \end{align}
   and
   \begin{align*}
   \frac12\int_0^t \text{tr }\phi^{\prime\prime}(X^n_s)(g(s),g(s))\,\d s\rightarrow \frac12\int_0^t \text{tr }\phi^{\prime\prime}(X_s)(g(s),g(s))\,\d s\hspace{1cm}\mathbb{P}\text{-a.s.}
   \end{align*}
  Hence the It\^{o} formula follows by passing the limit in the equality \eqref{ito-formula-proof-eq1} along a subsequence.
   \end{proof}

%-----------------------------end of Proof-------
\begin{proof}[Proof of Theorem \ref{theo-Ito-3}]
The proofs are almost identical to the one given in the proof of Theorem \ref{theo-Ito-1}, the major change being the
proof of the term $I_2$. Since $\phi^\prime$ is $(r-1)$-H\"{o}lder continuous, by applying Lemma \ref{lem-holder},  routine computations give rise to
\begin{align*}
I_2=\sum_m\Big{[} \phi(X^n_{t\wedge\tau_{m}-})-\psi(X^n_{t\wedge\tau_{m-1}}) \Big{]}
		  		=\int_{0}^{t}\phi^\prime(X_s)(a_s)\,\d s-\int_0^t\int_Z\xi^n(s,z)\,\nu(\d z)\d s.
\end{align*}\end{proof}

\section*{Acknowledgment}
 The authors would like to thank the referees for their very careful reading of the manuscript and especially for their very valuable suggestions and constructive comments.

%-------------------------

	% ------------------------------------------------------------------------

	%\subsection*{Acknowledgment}
	%Many thanks to our \TeX-pert for developing this class file.
	% ------------------------------------------------------------------------
   % \end{document}
	% ------------------------------------------------------------------------

\end{document}